\documentclass{scrartcl}
\usepackage[top=2cm, bottom=2.5cm, left=2.5cm, right=2.5cm]{geometry}

\usepackage[T1]{fontenc}
\usepackage[utf8]{inputenc}
\usepackage{amsmath, amsthm, amssymb}
\usepackage{dsfont}
\usepackage{graphicx}
\usepackage[format=plain]{caption}%ben?tigt, um mit dem Befehl \caption*{} eine Bildunterschrift zu setzen, die nicht mit Abbildung X eingeleitet wird. 'normal'=\setcapindent{0}
\usepackage{capt-of}%ben?tigt, um 'figure' und 'table' nebeneinander zu setzen, siehe 'http://www.faqs.org/faqs/de-tex-faq/part6/' Punkt 6.1.7, aufgerufen am 01.08.11.
\usepackage{subcaption}
\usepackage{bm}
\usepackage{csquotes}
\usepackage{xcolor}
\usepackage{dlfltxbcodetips} %gebraucht fuer grosses X fuer kartesische Produkte

\newcommand{\norm}[1]{\left\| #1 \right\|}  %Norm
\newcommand{\scprd}[1]{\left\langle #1 \right\rangle}  %Skalarprodukt
\renewcommand{\d}{\,\mathrm{d}} %Differenzial in Integralen
\newcommand{\e}{\mathrm{e}} %eulersche Zahl
\newcommand{\imunit}{\mathrm{i}} %imaginaere Einheit
\newcommand{\ddt}[1]{\frac{\mathrm{d}} {\mathrm{d}t} #1}   %Ableitung
\newcommand{\N}{\mathbb{N}}  %Zahlbereiche

\newcommand{\R}{\mathbb{R}}
\newcommand{\C}{\mathbb{C}}

\newcommand{\ran}{\operatorname{ran}}
\newcommand{\rk}{\operatorname{rk}}
\newcommand{\diag}{\operatorname{diag}}
\newcommand{\supp}{\operatorname{supp}}
 
\newcommand{\dist}{\operatorname{dist}}
\newcommand{\dom}{\operatorname{dom}}

\DeclareMathOperator*{\argmax}{argmax}
\DeclareMathOperator*{\argmin}{argmin}
\newcommand{\conv}{\operatorname{conv}}

\newcommand{\eps}{\varepsilon}
\renewcommand{\phi}{\varphi}
\newcommand{\ul}{\underline}
\newcommand{\ol}{\overline}
\newcommand{\tr}{\operatorname{tr}}
\newcommand{\diam}{\operatorname{diam}}

\newcommand{\Cov}{\operatorname{Cov}}

\newcommand{\dimx}{d_\mathrm{x}}
\newcommand{\dimy}{d_\mathrm{y}}
\newcommand{\dimp}{d_\mathrm{\theta}}
\newcommand{\Rpsd}{\R_{\mathrm{psd}}}
\newcommand{\Rpd}{\R_{\mathrm{pd}}}
\newcommand{\lse}{\widehat{\theta}}
\newcommand{\lin}{\mathrm{lin}}
\newcommand{\X}{\mathcal{X}}
\newcommand{\Y}{\mathcal{Y}}
\newcommand{\Xifin}{\Xi_{\mathrm{fin}}}

\newcommand{\pdirwidth}{\operatorname{pdirwidth}}
\newcommand{\pwidth}{\operatorname{pwidth}}
\newcommand{\curvconst}{C_{\Psi,\mathcal{M}}}
\newcommand{\convexconst}{\mu_{\Psi,\mathcal{M}}}
\newcommand{\faces}{\operatorname{faces}}
\newcommand{\fact}{\mathrm{fact}}
\newcommand{\rand}{\mathrm{rand}}

\numberwithin{equation}{section}
\newtheorem{thm}{Theorem}[section]
\newtheorem{cor}[thm]{Corollary}
\newtheorem{prop}[thm]{Proposition}
\newtheorem{lm}[thm]{Lemma}
\newtheorem{cond}[thm]{Condition}

\theoremstyle{definition} \newtheorem{algo}[thm]{Algorithm}
\theoremstyle{definition} \newtheorem{ex}[thm]{Example}
\theoremstyle{definition}

\title{Adaptive discretization algorithms for locally optimal experimental design}
\author{Jochen Schmid, Philipp Seufert, Michael Bortz
\\\small Fraunhofer Institute for Industrial Mathematics (ITWM), 67663 Kaiserslautern, Germany\\
\small jochen.schmid@itwm.fraunhofer.de}  

\date{}

\begin{document}

\maketitle

\begin{abstract}
\small{ \noindent 
We develop adaptive discretization algorithms for locally optimal experimental design of nonlinear prediction models. %We thereby improve %thereby improving ...
With these algorithms, we refine and improve a pertinent state-of-the-art algorithm %from the literature 
in various respects. %In particular, our algorithms only require the approximate solutions of the relevant finite-dimensional subproblems and thus reduce the computational effort significantly, in general. Additionally, and most importantly, we establish novel ...
We establish novel termination, convergence, and convergence rate results for the proposed algorithms. In particular, we prove a sublinear convergence rate result under very general assumptions on the design criterion and, most notably, a linear convergence result under the additional assumption that the design criterion is strongly convex and the design space is finite. %In practice, these additional assumptions do not constitute serious restrictions.
Additionally, we prove the finite termination at approximately optimal designs, %along with 
including upper bounds on the number of iterations until termination. And finally, we illustrate the practical use of the proposed algorithms by means of two application examples from chemical engineering: one with a stationary model and one with a dynamic model. 
}
\end{abstract}

%{ \small \noindent 
%%AMS subject classification (2020): 
%%90C34, %Semi-infinite programming
%%90C25, %Convex programming
%%90C31 %Sensitivity, stability, parametric optimization
%%\\
%Index terms: adaptive discretization algorithm, locally optimal experimental design, nonlinear models, implicit models, convergence rate, convex optimization on probability measures 
%} 

\section{Introduction}

%What do we consider in this paper?
In this paper, we are concerned with locally optimal experimental design for %generally nonlinear 
nonlinear prediction models. 
Such models are parametric models $f: \X \times \Theta \to \Y$ describing a functional input-output relationship of interest, that is, for every model parameter value $\theta$, the model 
\begin{align}
\X \ni x \mapsto f(x,\theta) \in \Y
\end{align}
predicts the considered output quantities $y \in \Y = \R^{\dimy}$ as a function of the considered input quantities $x \in \X \subset \R^{\dimx}$. 
%Such models are parametric models $f: \X \times \Theta \to \Y$ describing the functional relationship between relevant input and output quantities $x \in \X \subset \R^{\dimx}$ and $y \in \Y \subset \R^{\dimy}$, %of interest,
%that is, they depend on model parameters $\theta \in \Theta \subset \R^{\dimp}$ and for every model parameter value $\theta$, the function 
%\begin{align}
%\X \ni x \mapsto f(x,\theta) \in \R^{\dimy}
%\end{align}
%is a model for the input-output relationship of interest. 
\smallskip

%What is locally optimal experimental design in a nutshell?
In a nutshell, locally optimal experimental design~\cite{Ch53, FeLe, PrPa} is about finding maximally informative experimental designs $\xi$ for $f$ locally around a given reference value $\ol{\theta}$ of the model parameters. 
As usual, an (approximate) experimental design $\xi$ \cite{FeLe, PrPa} is a -- discrete or non-discrete -- probability measure on a suitable subset of the input space, the so-called design space $X \subset \X$. In essence, an experimental design $\xi$ indicates at which design points $x \in X$ experiments should be performed and how much weight each of these experiments $x$ should be given (which, practically speaking, %can be translated
boils down to how often the individual experiments $x$ should be repeated~\cite{Fe, FeLe, PuRi92}). %Actually, 
In practice, one is only interested in discrete designs, that is, designs of the form
\begin{align}
\xi = \sum_{i=1}^n w_i \delta_{x_i}
\end{align}
with finitely many design points $x_1, \dots, x_n$ and corresponding importance weights $w_1, \dots, w_n$; but for theoretical investigations, non-discrete designs are important as well. 
In order to quantify the information content of a given design $\xi$ for $f$ around $\ol{\theta}$, one usually takes %utilizes, employs, resorts to 
the so-called information matrix of $\xi$ for $f$ around $\ol{\theta}$, that is, the positive semidefinite matrix
\begin{align} \label{eq:information-matrix-intro}
M_f(\xi,\ol{\theta}) := \int_X D_\theta f(x,\ol{\theta})^\top \varsigma^{-1} D_\theta f(x,\ol{\theta}) \d \xi(x)
\in \Rpsd^{\dimp \times \dimp}
\end{align}
with $\varsigma \in \Rpd^{\dimy \times \dimy}$ denoting the positive definite covariance matrix of the measurement errors for the considered output quantity. 
\smallskip

%What is locally optimal experimental design in more precise terms?
With this notion of information content, the basic task %goal 
of locally optimal experimental design can concisely be formulated as finding a design $\xi^*$ that maximizes the information matrix map $\xi \mapsto M_f(\xi,\ol{\theta}) \in \Rpsd^{\dimp \times \dimp}$ w.r.t.~the standard (L\"owner) ordering on the set %$\Rpsd^{\dimp \times \dimp}$ 
of positive semidefinite matrices. Since this matrix-valued optimization problem, however, has no solution in general~\cite{FeLe} (Example~2.1), one minimizes a suitable scalar-valued function $\xi \mapsto \Psi(M_f(\xi,\ol{\theta})) \in \R \cup \{\infty\}$ of the information matrix instead. Specifically, one considers convex antitonic functions $\Psi: \Rpsd^{\dimp \times \dimp} \to \R \cup \{\infty\}$ which, in this context, are called design criteria. So, summarizing, the basic task of locally optimal experimental design can be formulated as follows: for a given design criterion $\Psi$, find a solution $\xi^*$ to the locally optimal design problem
\begin{align} \label{eq:OED(X)-intro}
\min_{\xi \in \Xi(X)} \Psi(M_f(\xi,\ol{\theta})),
\end{align}
where $\Xi(X)$ denotes the convex set of all designs on $X$. Since the design criterion $\Psi$ is convex by assumption and the information matrix map $\xi \mapsto M_f(\xi,\ol{\theta})$ is linear by~\eqref{eq:information-matrix-intro}, the locally optimal design problem~\eqref{eq:OED(X)-intro} is a convex optimization problem. It is finite- or infinite-dimensional depending on whether the design space $X$ is finite or infinite. 
\smallskip

%Which iterative algorithms specifically tailored to locally optimal design problems exist in the literature?
In the extensive literature on locally optimal experimental design, there exists a large variety of iterative algorithms that are specifically tailored to solving convex design problems of the kind~\eqref{eq:OED(X)-intro}. Important instances of such algorithms can be found in the works~\cite{Fe, Wy70, At73, SiTiTo78, Bo86, Yu10, Yu11, YaBiTa13}, among others. See also~\cite{AtDoTo, FeLe} and especially \cite{PrPa} for a good overview of these and many more iterative algorithms for (locally) optimal experimental design. %known until 2013.

\subsection{Contributions of this paper}

%What are our main contributions, very roughly speaking?
In the present paper, we build upon the most recent and efficient of the aforementioned optimal design algorithms, namely the adaptive discretization algorithm from~\cite{YaBiTa13}. We refine and improve this algorithm in various respects and, most importantly, we establish novel termination, convergence, and convergence rate results for our adaptive discretization algorithms. In particular, we establish a sublinear convergence rate result under very general assumptions on the design criterion $\Psi$ and, most notably, a linear convergence rate result under the additional assumption that the design criterion $\Psi$ is strongly convex and the design space $X$ is finite. %As corollaries of these convergence rate results, we further obtain finite termination results including upper bounds on the number of iterations until termination. 
So far, the convergence rate of optimal design algorithms like~\cite{Fe, Wy70, At73, SiTiTo78, Bo86, Yu10, Yu11, YaBiTa13} has not been investigated systematically yet, to the best of our knowledge. %Instead, in the literature,  only the mere convergence of the algorithms is established.
With the present paper, we %intend to 
close this gap in the literature. We also illustrate the practical use of the proposed algorithms by means of two %concrete experimental design problems 
application examples from chemical engineering, the first one being based on a stationary process model $f$ and the second one being based a dynamic process model $f$. %one based on a stationary process model $f$ and one based a dynamic process model $f$.   
\smallskip

%What is the baseline algorithm: how does it work, in a nutshell?
In order to explain our %(theoretical)
contributions in more detail, we first have to briefly recap how the algorithm from~\cite{YaBiTa13} works. It is an iterative algorithm that in every iteration proceeds in two steps. In the  first step, the algorithm computes a solution $\xi^k$ to the discretized design problem
\begin{align} \label{eq:OED(X^k)-intro}
\min_{\xi \in \Xi(X^k)} \Psi(M_f(\xi,\ol{\theta})),
\end{align}
where $X^k \subset X$ is the current discretization (finite subset) of $X$. In the second step, the algorithm checks how far the solution $\xi^k$ to~\eqref{eq:OED(X^k)-intro} is from being already optimal for the original design problem~\eqref{eq:OED(X)-intro}. And for that purpose, it searches for a worst violator $x^k$ of the (necessary and sufficient) optimality condition of $\xi^k$ for~\eqref{eq:OED(X)-intro}, that is, it computes a solution $x^k$ to the worst-violator problem
\begin{align} \label{eq:Viol(xi^k)-intro}
\min_{x \in X} \psi(M_f(\xi^k,\ol{\theta}), x)
\end{align} 
where $\psi$ is a suitable directional derivative of $\Psi$, the so-called sensitivity function of $\Psi$. In the case where $x^k$ actually is not a violator of the optimality condition, $\xi^k$ already is a solution to the original problem~\eqref{eq:OED(X)-intro} and the iteration is terminated. In the opposite case, however, the worst violator $x^k$ is added to the discretization according to %and the discretization is updated as follows:
\begin{align} \label{eq:X^k+1-intro}
X^{k+1} := \supp \xi^k \cup \{x^k\}
\end{align}
and the iteration is continued with the updated discretization~\eqref{eq:X^k+1-intro}. In this relation, $\supp \xi^k \subset X^k$ denotes the support of the design $\xi^k$, %from the first step
of course.
\smallskip

%What are our main improvements to the baseline algorithms?
We refine and improve this algorithm mainly in two respects. (i) Instead of the exact optimality condition for $\xi^k$, our algorithms work with the less restrictive (necessary and sufficient) condition for $\eps$-approximate optimality of $\xi^k$, namely
\begin{align} \label{eq:eps-optimality-condition-intro}
\psi((M_f(\xi^k,\ol{\theta}), x) \ge -\eps \qquad (x \in X)
\end{align}  
with some user-specified $\eps \in [0,\infty)$. (ii) Instead of exact solutions of the subproblems~\eqref{eq:OED(X^k)-intro} and~\eqref{eq:Viol(xi^k)-intro}, our algorithms only require approximate solutions $\xi^k$ and $x^k$ to these problems. In particular, the strict versions of our algorithms require only approximately worst violators $x^k$ of~\eqref{eq:eps-optimality-condition-intro} and the relaxed versions of our algorithms even require only arbitrary violators $x^k$ of~\eqref{eq:eps-optimality-condition-intro}. 
It is clear that with these modifications, the computational effort per iteration is significantly reduced, in general.
\smallskip

%Which termination, convergence, and convergence rate results do we establish for our algorithms?
We then establish novel termination, convergence, and convergence rate results for the proposed algorithms.
A bit more specifically, we begin with basic termination and convergence results under minimal assumptions, namely essentially the lower semicontinuity, convexity, and continuous directional differentiability of the design criterion $\Psi$. %Indeed, for $\eps > 0$ we show the termination of our strict and relaxed algorithms at an $\eps$-optimal design and for $\eps = 0$ we show the convergence of our strict algorithms to an optimal design. %We begin with basic termination results for our algorithms with $\eps > 0$ and basic convergence results for our strict algorithms with $\eps = 0$.  
We then establish a sublinear convergence rate result for our strict algorithms under the additional assumption that the design criterion $\Psi$ is locally $L$-smooth. Specifically, we show that for the iterates $(\xi^k)_{k\in K}$ of our strict algorithms, the optimality errors can be bounded above as
\begin{align} \label{eq:sublinear-convergence-intro}
\Psi(M_f(\xi^k,\ol{\theta})) - \Psi^* \le C/(k+1) \qquad (k \in K),
\end{align}
where $C$ is a constant depending on the initial optimality error $\Psi(M_f(\xi^0,\ol{\theta})) - \Psi^*$, the local smoothness constant of $\Psi$, and the diameter of the matrix set $M_f(\Xi(X),\ol{\theta})$. 
And finally, we establish a linear convergence rate result for our strict algorithms under the additional assumption that the design criterion $\Psi$ is locally $L$-smooth and $\mu$-strongly convex and the design space $X$ is finite. Specifically, we show that for the iterates $(\xi^k)_{k\in K}$ of our strict algorithms, the optimality errors can be bounded above as
\begin{align} \label{eq:linear-convergence-intro}
\Psi(M_f(\xi^k,\ol{\theta})) - \Psi^* \le C r^k \qquad (k \in K),
\end{align}
where $C := \Psi(M_f(\xi^0,\ol{\theta})) - \Psi^*$ is the initial optimality error and $r \in [1/2, 1)$ is a constant depending on the local smoothness constant $L$, the strong convexity constants $\mu$ of $\Psi$, and on the diameter and the pyramidal width~\cite{LaJa15} of the matrix set $M_f(\Xi(X),\ol{\theta})$ (which is a polytope by the assumed finiteness of $X$). As we will see, %It should be noticed
from a practical point of view, the assumptions of the sublinear and also of the linear convergence rate result pose no serious restrictions. In particular, %Among others, 
they are satisfied for the arguably most commonly used design criteria, namely the A- and the D-criterion.

\subsection{Conventions on notation}

In the entire paper, we write $d := \dimp \in \N$ for the number of model parameters and we use the symbols %$\R^{d \times d}_{\mathrm{sym}}$, $\Rpsd^{d \times d}$ and $\Rpd^{d \times d}$ denote the set of the symmetric, the positive semidefinite and, respectively, the positive definite $d \times d$ matrices over $\R$, 
\begin{align}
\Rpd^{d \times d} \subset \Rpsd^{d \times d} \subset \R^{d \times d}_{\mathrm{sym}}
\end{align}
to denote, respectively, the sets of positive definite, positive semidefinite, and symmetric $d \times d$ matrices over $\R$. 
%
%Write $M \le N$ to indicate that $M$ is less than or equal to $N$ in the sense of standard (L\"owner) ordering on $\R^{d \times d}$, that is, $v^\top M v \le v^\top N v$ for all $v \in \R^d$. Very often/As is common practice/As usual, we will also write multiples $c I$ of the identity matrix simply as $c$.
%
%Additionally, $\norm{M} := \max_{|v| \le 1} |Mv|$ denotes the operator norm (Schatten-$\infty$ norm) of $ M \in \R^{d\times d}$ induced by the Euclidean norm $|\cdot|$ on $\R^d$, while $|M| := \tr(M^\top M)^{1/2}$ denotes the Frobenius norm (Schatten-$2$ norm) of $ M \in \R^{d\times d}$. 
Additionally, for a given matrix $M \in \R^{d\times d}$,
\begin{align}
\norm{M} := \max_{|v| \le 1} |Mv|
\qquad \text{and} \qquad
|M| := \tr(M^\top M)^{1/2}
\end{align}
will always denote the operator norm (Schatten-$\infty$ norm) of $M$ induced by the $\ell^2$-norm $|\cdot|$ on $\R^d$ and, respectively, the Frobenius norm (Schatten-$2$ norm) of $M$. 
Also, for any metric space $X$, we denote by $\Xi(X)$ the set of probability measures on $\mathcal{B}_X$ (Borel sigma-algebra of $X$) and by $\delta_x \in \Xi(X)$ we denote the point measure concentrated at $x$. And finally, monotonicity -- especially antitonicity and monotonic increasing- and decreasingness -- are always understood in the non-strict sense (with non-strict inequalities). In the case of strict monotonicity, we will always explicitly indicate this. %(by means of a \enquote{strict} specifier). 

%Write $M \le N$ to indicate that $M$ is less than or equal to $N$ in the sense of standard (L\"owner) ordering on $\R^{d \times d}$, that is, $v^\top M v \le v^\top N v$ for all $v \in \R^d$. Very often/As is common practice/As usual, we will (follow common practice and) write multiples $c I$ of the identity matrix simply as $c$.
%B_{\delta}(x), \diag(a_1, \dots, a_d)
%Apart from the Euclidean norm on $\R^d$, the symbol $|\cdot|$ also denotes the Schatten-$2$ norm (Frobenius norm) on $\R^{d \times d}$, that is, ... And moreover, we also write $|S|$ for the cardinality of a set $S$. 
%\diam, \dist

\section{Setting and preliminaries}

In this section, we introduce the formal setting of locally optimal experimental design and collect the necessary preliminaries for our termination, convergence, and convergence rate results. %In particular, we establish the strong convexity of a large class of practically relevant design criteria. %In particular, we establish the $L$-smoothness of almost all practically relevant design criteria (except for the E-criterion) and the $\mu$-strong convexity of a large class of practically relevant design criteria. 
%Additionally, we discuss solvability conditions for the considered design problems~\eqref{eq:OED(X)-intro} as well as necessary and sufficient $\eps$-optimality conditions for~\eqref{eq:OED(X)-intro}.  
%At a first reading, one can safely skim through the more technical results and only refer back to them when necessary.

\subsection{Setting}

In this section, we introduce %set out, detail, present, record, explain  
the setting of locally optimal experimental design in formal terms. As indicated in the introduction, the starting point is a -- linear or nonlinear -- parametric model %$f: \X \times \Theta \to \Y$ 
\begin{align}
f: \X \times \Theta \to \Y
\end{align}
for the prediction of the considered output quantities $y \in \Y = \R^{\dimy}$ as a function of relevant input quantities $x \in \X$. As usual, such a model is called linear iff it is affine-linear w.r.t.~to the model parameters $\theta \in \Theta$, that is, iff it is of the form $f(x,\theta) = c(x) + J(x)\theta$ for some arbitary functions $c: \X \to \R^{\dimy}$ and $J: \X \to \R^{\dimy \times \dimp}$. All one needs to assume about $f$ %about this model $f$ for locally optimal design 
to get locally optimal design going are a few mild regularity conditions (Condition~\ref{cond:regularity}) and a representativity condition (Condition~\ref{cond:observation-representability}), which is completely standard in  experimental design and parameter estimation~\cite{FeLe, PrPa, SeWi}.

\begin{cond} \label{cond:regularity}
\begin{itemize}
\item[(i)] $X$ is a compact metric space (the design space) and $X$ is contained in the set $\X$ (the input space)
\item[(ii)] $\Theta$ is a compact subset of $\R^{\dimp}$ (the model parameter space) that is equal to the closure of its interior, in short, $\ol{\Theta^\circ} = \Theta$
\item[(iii)] $\Theta \ni \theta \mapsto f(x,\theta)$ is differentiable for every $x \in X$ and $X \ni x \mapsto D_\theta f(x,\ol{\theta})$ is continuous for every $\ol{\theta} \in \Theta$. %As usual, $D_{\theta}f(x,\overline{\theta})$ denotes the derivative (Jacobi matrix) of $f(x,\cdot)$  at $\overline{\theta}$.
\end{itemize}
\end{cond}

\begin{cond} \label{cond:observation-representability}
A parameter value $\theta^* \in \Theta$ exists such that for every $n \in \N$ and every set %tuple 
of input values $x_1, \dots, x_n \in X$, the corresponding measured values $y_1, \dots, y_n \in \R^{\dimy}$ of the output quantity are given by the predictions $f(x_1,\theta^*), \dots, f(x_n,\theta^*)$ of the model $f(\cdot,\theta^*)$ up to independent normally distributed measurement errors $\epsilon_1, \dots, \epsilon_n$, that is, 
\begin{align} \label{eq:representability}
y_i = f(x_i,\theta^*) + \epsilon_i \qquad (i \in \{1,\dots,n\})
\end{align}
where $\epsilon_1, \dots, \epsilon_n$ are realizations of independent and normally distributed measurement errors $\eps_1, \dots, \eps_n$ having mean $0$ and a known covariance matrix $\varsigma \in \R^{\dimy \times \dimy}_{\mathrm{pd}}$.
\end{cond}

Conventionally, every tuple $\tilde{x} = (x_1, \dots, x_n) \in X^n$ of -- not necessarily distinct -- design points as above is called an exact design. Also, any value $\theta^*$ as above is called (the) true model parameter value. Since the true value $\theta^*$ is unknown, of course, one needs to come up with reliable estimates for it and a standard way of doing so is least-squares estimation, that is, for a given exact design $\tilde{x} = (x_1, \dots, x_n) \in X^n$ one measures the values $\tilde{y} = (y_1, \dots, y_n)$ of the output quantity and then computes a  corresponding least-squares estimate %$\lse_f(\tilde{x},\tilde{y})$.
\begin{align}
\lse_f(\tilde{x},\tilde{y}) \in \argmin_{\theta \in \Theta} \sum_{i=1}^n (f(x_i,\theta)-y_i)^\top \varsigma^{-1} (f(x_i,\theta)-y_i).
\end{align}
In view of Condition~\ref{cond:regularity}, %the existence of such a least-squares estimate is clear. 
it is clear that such a least-squares estimate exists and that it can be chosen to depend measurably on $\tilde{y}$  (Lemma~2.9 of \cite{PrPa}). Since for every design $\tilde{x}$ the corresponding measurement values $\tilde{y}$ are subjected to random measurement errors by Condition~\ref{cond:observation-representability}, the same is true for the computed least-squares estimate $\lse_f(\tilde{x},\tilde{y})$. Specifically, $\lse_f(\tilde{x},\tilde{y})$ is a realization of the random variable $\lse_f(\tilde{x},\tilde{y}(\tilde{x}))$, where %$\tilde{y}(\tilde{x})$ is the $n\dimy$-dimensional random variable defined by
\begin{align}
\tilde{y}(\tilde{x}) := \tilde{f}(\tilde{x},\theta^*) + \tilde{\eps}(\tilde{x})
\end{align}
is the %$n\dimy$-dimensional 
random variable defined by the predictions $\tilde{f}(\tilde{x},\theta^*) := (f(x_1,\theta^*), \dots, f(x_n,\theta^*))$ of the true model at $\tilde{x}$ plus normally distributed measurement errors $\tilde{\eps}(\tilde{x})$ with mean $0$ and covariance matrix $\tilde{\varsigma} := \diag(\varsigma, \dots, \varsigma) \in \Rpd^{n\dimy \times n\dimy}$. 
\smallskip

In optimal experimental design~\cite{KiWo59, Fe, Ki74, Si, Pu, AtDoTo, FeLe, PrPa}, the goal is to find a design $\tilde{x}$ such that the uncertainty $\Cov \lse_f(\tilde{x},\tilde{y}(\tilde{x}))$ of the corresponding least-squares estimate %incurred by the measurement uncertainties. 
becomes minimal. As is well-known, %\cite{NoWr}
for nonlinear models $f$ one usually has no closed-form expression of this covariance and one therefore has to resort to suitable approximations for it. In locally optimal experimental design, one uses the approximation
\begin{align} \label{eq:covariance-approximation}
\Cov \lse_f(\tilde{x},\tilde{y}(\tilde{x})) \approx \Cov \lse_{f^{\lin}_{\ol{\theta}}}(\tilde{x},\tilde{y}(\tilde{x})).
\end{align}  
In other words, one approximates the covariance of the least-squares estimate for $f$ by the covariance of the least-squares estimate for the linearization $f^{\lin}_{\ol{\theta}}$ of $f$ around a suitable reference parameter value $\ol{\theta}$, that is,
\begin{align} \label{eq:linearized-model}
f^{\lin}_{\ol{\theta}}(x,\theta) := f(x,\ol{\theta}) + D_\theta f(x,\ol{\theta})(\theta-\ol{\theta})
\qquad ((x,\theta) \in \X \times \R^{\dimp}).
\end{align}
And this linearized covariance, in turn, can be %neatly 
compactly expressed through the information matrix~\eqref{eq:information-matrix-intro} of the (approximate) design 
\begin{align} \label{eq:approximate-design-corresponding-to-exact-design}
\xi_{\tilde{x}} := \frac{1}{n} \sum_{i=1}^n \delta_{x_i}
\end{align}
corresponding to the exact design $\tilde{x} = (x_1, \dots, x_n)$.

\begin{lm}
Suppose that Conditions~\ref{cond:regularity} and~\ref{cond:observation-representability} are satisfied. If $\tilde{x} \in X^n$ is an arbitrary exact design of size $n \in \N$ and $\ol{\theta} \in \Theta$ is an arbitrary reference parameter value, then 
\begin{align} \label{eq:linearized-covariance-and-information-matrix}
\Cov \lse_{f^{\lin}_{\ol{\theta}}}(\tilde{x},\tilde{y}(\tilde{x})) = \frac{1}{n} M_f(\xi_{\tilde{x}}, \ol{\theta})^+,
\end{align}
where $\lse_{f^{\lin}_{\ol{\theta}}}(\tilde{x},\tilde{y})$ is the minimum-norm least-squares estimate of $f^{\lin}_{\ol{\theta}}$ for $\tilde{x}$ and $\tilde{y}$ 
and $M_f(\xi_{\tilde{x}}, \ol{\theta})^+$ is the (Moore-Penrose) pseudo-inverse of the information matrix of~\eqref{eq:approximate-design-corresponding-to-exact-design}. 
\end{lm}

In view of~\eqref{eq:covariance-approximation} and~\eqref{eq:linearized-covariance-and-information-matrix}, it becomes clear why the information matrix $M_f(\xi,\ol{\theta})$ %can be taken as a measure of information content
is a valid measure of (local) information content of the design $\xi$ about the true model $f(\cdot, \theta^*)$, thus justifying its central role in locally optimal experimental design.

\subsection{Some technical preliminaries}

In this section, we collect some basic definitions and facts that are important for the optimal design of experiments and, in particular, for our convergence results. 

\begin{lm} \label{lm:rint(pos-def-matrices)}
$\Rpd^{d \times d}$ is convex and open in $\R^{d \times d}$.
%and $\rint(\Rpd^{d \times d}) = \Rpd^{d \times d}$. 
Additionally, $\Rpsd^{d \times d}$ is convex and the closure of $\Rpd^{d \times d}$ in $\R^{d \times d}$. 
\end{lm}

\begin{proof}
Straightforward verification.
\end{proof}

%\Rpd^{d \times d} \ni M \mapsto M^{-1} is antitonic and convex

\subsubsection{Convexity and differentiability}

As usual, for a subset $S$ of some %finite-dimensional 
real vector space $V$, 
%the relative interior $\rint(S)$ is defined as the relative interior of $S$ w.r.t.~the relative topology of the affine hull of $S$. Also, 
the convex hull $\conv(S)$ of $S$ is the smallest convex set in $V$ that comprises $S$. 

\begin{lm}[Carath\'{e}odory] \label{lm:caratheodory}
If $S$ is an arbitrary subset of some finite-dimensional real vector space $V$, then every element of $\conv(S)$ can be written as a convex combination of at most $\dim V + 1$ points from $S$. Additionally, every element of $\conv(S) \cap \partial \conv(S)$ can be written as a convex combination of at most $\dim V$ points from $S$.
\end{lm} 

\begin{proof}
See Theorem~17.1 of~\cite{Rockafellar} for a proof of the first part of the lemma and Appendix~2 of~\cite{Si} for a proof of the second part of the lemma.
\end{proof}

%\begin{lm} \label{lm:closure-of-convex-hull}
%If $S$ is a compact subset of some finite-dimensional real vector space $V$, then $\conv(S)$ is compact as well.
%\end{lm}
%
%\begin{proof}
%See Theorem~17.2 of~\cite{Rockafellar}.
%\end{proof}

We recall that for an extended real-valued function $\Psi: S \to \R \cup \{\infty\}$ on some subset $S$ of a %finite-dimensional 
real vector space $V$, the domain $\dom \Psi$ is defined as %$\dom\Psi := \{x \in V: \Psi(x) < \infty\}$ 
\begin{align}
\dom\Psi := \{x \in S: \Psi(x) < \infty\},
\end{align}
that is, the set of points where $\Psi$ takes finite values. We also recall that a function $\Psi: C \to \R \cup \{\infty\}$ on some convex subset $C$ of $V$ is convex iff
\begin{align}
\Psi(\alpha x + (1-\alpha)y) \le \alpha \Psi(x) + (1-\alpha) \Psi(y)
\end{align}
for all $x,y \in C$ and all $\alpha \in [0,1]$. It is trivial to verify that the extended real-valued function $\Psi: C \to \R \cup \{\infty\}$ %on some convex set $C$ 
is convex if and only if $\dom\Psi$ is convex and the real-valued function $\Psi|_{\dom\Psi}$ is convex.
Similarly, $\Psi: C \to \R \cup \{\infty\}$ on some convex subset $C$ of $V$ is called strictly convex iff
\begin{align} \label{eq:strict-convexity-def}
\Psi(\alpha x + (1-\alpha)y) < \alpha \Psi(x) + (1-\alpha) \Psi(y)
\end{align}
for all $x,y \in \dom\Psi$ with $x \ne y$ and all $\alpha \in (0,1)$ (Section~26 of~\cite{Rockafellar}). In other words, $\Psi: C \to \R \cup \{\infty\}$ is strictly convex if and only if its real-valued restriction $\Psi|_{\dom\Psi}$ is strictly convex. We call $\Psi: C \to \R \cup \{\infty\}$ strictly mid-point convex iff~\eqref{eq:strict-convexity-def} holds true with $\alpha := 1/2$ for all $x,y \in \dom\Psi$ with $x \ne y$. 

\begin{lm} \label{lm:strict-convexity}
Suppose that $\Psi: C \to \R \cup \{\infty\}$ is convex, where $C$ is a convex subset of some %finite-dimensional 
real vector space. Then $\Psi$ is strictly convex if and only if it is strictly mid-point convex. %(in the sense that~\eqref{eq:strict-convexity-def} holds true for all $x,y \in \dom\Psi$ with $x \ne y$ and for $\alpha := 1/2$). 
\end{lm}

\begin{proof}
Simple verification -- to get the idea, draw a picture using the geometric interpretation of (strict) convexity.
\end{proof}

%\begin{lm} \label{lm:automatic-cont-and-directional-db}
%If $\Psi: V \to \R \cup \{\infty\}$ is convex on some finite-dimensional real vector space $V$, then $\Psi|_{\rint(\dom \Psi)}$ is continuous and $\Psi|_{\dom\Psi}$ is directionally differentiable. 
%\end{lm}
%
%\begin{proof}
%See Theorem~10.1 and Theorem~23.1 of~\cite{Rockafellar}.
%\end{proof}

%Directional differentiability (Definition~3.1 of~\cite{Jahn})
As usual, a real-valued function $\Psi: S \to \R$ on some subset $S$ of a real vector space $V$ is called directionally differentiable at $x \in S$ in the direction $e \in V$ iff $t \mapsto \Psi(x+te)$ is right-differentiable at $t = 0$. Spelled out, this means that there exists a $t^* \in (0,\infty)$ such that $x+te \in S$ for all $t \in [0,t^*]$ and the limit %$\lim_{t\searrow 0} \frac{\Psi(x+te)-\Psi(x)}{t}$ 
\begin{align} \label{eq:directional-derivative-def}
\partial_e \Psi(x) := \ddt \Psi(x+te)\big|_{t=0+} := \lim_{t\searrow 0} \frac{\Psi(x+te)-\Psi(x)}{t}
\end{align}
exists in $\R$. 
%If $\Psi$ is directionally differentiable at $x$ in the direction $e$, then ...
In this case, the limit~\eqref{eq:directional-derivative-def} is called the directional derivative of $\Psi$ at $x$ in the direction $e$ (Section~3.1 of~\cite{Jahn}). 
%
%Differentiability of a function $\Psi: C \to \R$ on a not necessarily open subset $S$ of $V$.
Apart from directional derivatives, we will also need derivatives of real-valued functions defined on arbitrary subsets, instead of just open subsets of a normed vector space. So, let $\Psi: S \to \R$ be a real-valued function on some arbitrary subset $S$ of a real normed vector space $V$ with norm $\norm{\cdot}$. We then call $\Psi$ (continuously) differentiable iff it can be extended to a real-valued function $\tilde{\Psi}: \tilde{S} \to \R$ on some open subset $\tilde{S}$ of $V$ that is (continuously) differentiable in the usual (Fr\'{e}chet) sense (Section~VII.2 of~\cite{AmannEscher} or Section 3.2 of~\cite{Jahn}, for instance). In this case, we will write
\begin{align} \label{eq:derivative-def}
D\Psi(x)e := D\tilde{\Psi}(x)e
\end{align}
for the derivative of $\Psi$ at $x \in S$ in the direction $e \in V$. In general, this derivative~\eqref{eq:derivative-def} depends on the chosen extension $\tilde{\Psi}$, of course. If, however, for a given $x \in S$ and $e \in V$ there exists a $t^* \in (0,\infty)$ such that $x+te \in S$ for all $t \in [0,t^*]$, then the derivative~\eqref{eq:derivative-def} is independent of the chosen differentiable extension $\tilde{\Psi}$  (because then $D\tilde{\Psi}(x)e = \partial_e\Psi(x)$ for every differentiable extension). 

\begin{lm} \label{lm:directional-derivatives}
If $M \in  \Rpd^{d \times d}$ and $p \in (0,\infty)$, then for every direction $E \in \R^{d\times d}$ one has the following directional derivatives:
\begin{gather}
\ddt \log \det(M+tE)^{-1} \big|_{t=0+} = -\tr(M^{-1}E) 
\label{eq:D-criterion-derivative} \\
\ddt \tr((M+tE)^{-p}) \big|_{t=0+} = -p \tr(M^{-p-1}E).
\label{eq:A-criterion-derivative}
\end{gather}
Additionally, the map $\Rpd^{d \times d} \ni M \mapsto M^{-p}$ is locally Lipschitz continuous for every $p \in (0,\infty)$. 
\end{lm}

\begin{proof}
See \cite{MagnusNeudecker} (Theorem~8.2), for instance, %or Theorem B.18 of~\cite{Ucinski}
for a proof of~\eqref{eq:D-criterion-derivative}. We now prove~\eqref{eq:A-criterion-derivative} because we could not find a proof for non-integer $p$ in the literature. 
So, let $p \in (0,\infty)$ and $M \in \Rpd^{d \times d}$ and $E \in \R^{d\times d}$ be fixed and write $M_t := M+tE$ for brevity. Also, let $f: \C \setminus (-\infty,0] \to \C$ be the holomorphic function defined by $f(z) := z^{-p} := \e^{-p \log(z)}$, where $\log: \C \setminus (-\infty,0] \to \C$ denotes the principal arc of the complex logarithm. In particular,
\begin{align}  \label{eq:directional-derivative-A-criterion,1}
f'(z) = -pz^{-p-1} \qquad (z \in \C \setminus (-\infty,0]).
\end{align}
Choose now $t_0 > 0$ and $\delta > 0$ so small that $M_t + \delta  \in \Rpd^{d \times d}$ for every $t \in [0,t_0]$ (Lemma~\ref{lm:rint(pos-def-matrices)}). It then follows that the spectrum $\sigma(M_t)$ of $M_t$ is contained in $ [\delta,\infty) \subset \C \setminus (-\infty,0]$ for all $t \in [0,t_0]$ and that there exists a %cycle 
closed path $\gamma$ in $\C \setminus (-\infty,0]$ that has winding number $1$ around $\sigma(M_t)$ and winding number $0$ around the complement $\C \setminus \sigma(M_t)$. %such that $\sigma(M_t)$ lies inside of $\gamma$ and 
So, by the holomorphic functional calculus~\cite{DunfordSchwartz} (Section~VII.3), \cite{Rudin} (Section 10.23)  and~\eqref{eq:directional-derivative-A-criterion,1} that
\begin{align}
(M+tE)^{-p} = f(M_t) 
&= \frac{1}{2\pi \imunit} \int_{\gamma}f(z)(z-M_t)^{-1}\d z 
\label{eq:directional-derivative-A-criterion,2}\\
\frac{1}{2\pi \imunit} \int_{\gamma}f(z)(z-M_t)^{-2}\d z 
&= \frac{1}{2\pi \imunit} \int_{\gamma}f'(z)(z-M_t)^{-1}\d z = f'(M_t) \notag\\
&= -p(M+tE)^{-p-1}
\label{eq:directional-derivative-A-criterion,3}
\end{align}
for all $t \in [0,t_0]$, where the first equality in~\eqref{eq:directional-derivative-A-criterion,3} follows by Cauchy's theorem. It further follows from~\eqref{eq:directional-derivative-A-criterion,2} by taking the derivative inside the path integral (dominated convergence theorem) that
\begin{align} \label{eq:directional-derivative-A-criterion,4}
\ddt (M+tE)^{-p} \big|_{t=0+} 
&= \frac{1}{2\pi \imunit} \int_{\gamma}f(z)\ddt (z-M_t)^{-1} \big|_{t=0+} \d z, \notag \\
%&= \frac{1}{2\pi \imunit} \int_{\gamma}f(z)\ddt (z-M_t)^{-1}M'(t)(z-M_t^{-1}\d z
&= \frac{1}{2\pi \imunit} \int_{\gamma}f(z)(z-M)^{-1} E (z-M)^{-1}\d z
\end{align}
So, by taking the trace in~\eqref{eq:directional-derivative-A-criterion,4} and using its invariance under cyclic permutations in conjunction with~\eqref{eq:directional-derivative-A-criterion,3} at $t = 0$, we see that
\begin{align}
\ddt \tr((M+tE)^{-p}) \big|_{t=0+} &= \frac{1}{2\pi \imunit} \int_{\gamma}f(z)\tr\big((z-M)^{-1} E (z-M)^{-1}\big)\d z \notag\\
&= -p\tr(M^{-p-1}E),
\end{align}
as desired. 
It remains to establish the local Lipschitz continuity of $\Rpd^{d \times d} \ni M \mapsto M^{-p}$. So, let $\mathcal{M}$ be a compact subset of $\Rpd^{d \times d}$. It then follows that %$\mathcal{M} \subset \{M \in \Rpd^{d \times d}: c \le M \le C\}$ 
\begin{align} \label{eq:M^-p-Lipschitz,1}
\mathcal{M} \subset \{M \in \Rpd^{d \times d}: c \le M \le C\}
\end{align}
for some positive constants $c, C \in (0,\infty)$. It further follows by the holomorphic functional calculus that
\begin{align} \label{eq:M^-p-Lipschitz,2}
M^{-p} - N^{-p} = \frac{1}{2\pi \imunit} \int_{\gamma} z^{-p} \big( (z-M)^{-1} - (z-N)^{-1} \big) \d z
\qquad (M,N \in \mathcal{M})
\end{align}
for the tank-shaped closed path $\gamma$ in $\C \setminus (-\infty,0]$ that connects the four points $c+\imunit c$, $C+\imunit c$, $C - \imunit c$, $c -\imunit c$ by two horizontal straight lines and, respectively, by two vertically oriented semi-circles. %that goes from $c+\imunit c$ to $C+\imunit c$ and from $C - \imunit c$ to $c -\imunit c$ in straight lines and from $C+\imunit c$ to $C - \imunit c$ and from $c-\imunit c$ to $c + \imunit c$ in semi-circles. 
Clearly, the length of that path is
\begin{align} \label{eq:M^-p-Lipschitz,3}
l(\gamma) = 2(C-c) + 2\pi c/2 = 2(C-c) + \pi c
\end{align}
and, moreover, the operator norm of the inverse of the matrices $z-A$ for $z$ on the path $\gamma$ and $A \in \mathcal{M}$ can be estimated as 
\begin{align} \label{eq:M^-p-Lipschitz,4}
\norm{(z-A)^{-1}} = \frac{1}{\dist(z,\sigma(A))} \le \frac{1}{\dist(z,[c,C])} \le \frac{2}{c}
\qquad (z \in \ran(\gamma) \text{ and } A \in \mathcal{M})
%\norm{(z-M)^{-1} - (z-N)^{-1}} \le \norm{(z-M)^{-1}} \norm{(z-N)^{-1}} \norm{N-M} = \frac{1}{\dist(z,\sigma(M))} \frac{1}{\dist(z,\sigma(N))} \norm{N-M}
\end{align}
because the matrices $A \in \mathcal{M}$ are normal matrices with spectrum $\sigma(A) \subset [c,C]$ by~\eqref{eq:M^-p-Lipschitz,1}. Since $(z-M)^{-1} - (z-N)^{-1} = (z-M)^{-1}(N-M)(z-N)^{-1}$ we see from~\eqref{eq:M^-p-Lipschitz,2} with the help of~\eqref{eq:M^-p-Lipschitz,3} and~\eqref{eq:M^-p-Lipschitz,4} that
\begin{align} \label{eq:M^-p-Lipschitz,5}
\norm{M^{-p} - N^{-p}} 
%&\le \frac{2(C-c) + \pi c}{2\pi} \max_{z\in\ran(\gamma)} |z^{-p}| \cdot (2/c)^2 \norm{N-M} \notag\\
\le \frac{2(C-c) + \pi c}{2\pi} (2/c)^{p+2} \norm{N-M}
\qquad (M,N \in \mathcal{M}).
\end{align}
And therefore, $\mathcal{M} \ni M \mapsto M^{-p}$ is Lipschitz continuous, as desired. 
\end{proof}

We finally turn to the notions of $L$-smoothness and $\mu$-strong convexity~\cite{Vi83}, %\cite{Vi83, GaHa15}, 
which are essential for our convergence rate results. Suppose $\Psi: C \to \R$ is a real-valued function on some convex subset $C$ of a normed real vector space $V$ with norm $\norm{\cdot}$. Also, let $L \in [0,\infty)$ and $\mu \in [0,\infty)$ be given numbers. $\Psi$ is then called $L$-smooth w.r.t.~$\norm{\cdot}$ %for a given $L \in [0,\infty)$ 
iff $\Psi$ is differentiable in the sense defined above and 
\begin{align} \label{eq:L-smoothness-def}
\Psi(y)-\Psi(x)-D\Psi(x)(y-x) \le \frac{L}{2}\norm{y-x}^2 \qquad (x,y \in C).
\end{align} 
Similarly, $\Psi$ is called $\mu$-strongly convex w.r.t.~$\norm{\cdot}$ %for a given $\mu \in [0,\infty)$ 
iff $\Psi$ is differentiable in the sense defined above and
\begin{align} \label{eq:mu-strong-convexity-def}
\Psi(y)-\Psi(x)-D\Psi(x)(y-x) \ge \frac{\mu}{2}\norm{y-x}^2 \qquad (x,y \in C).
\end{align}
It should be noted that by the assumed convexity of $C$, the derivatives $D\Psi(x)(y-x) = \partial_{y-x}\Psi(x)$ appearing in the defining relations~\eqref{eq:L-smoothness-def} and~\eqref{eq:mu-strong-convexity-def} are actually independent of the chosen differentiable extension $\tilde{\Psi}$ of $\Psi$. See the remarks following~\eqref{eq:derivative-def}. We close with a simple sufficient condition for $L$-smoothness. 

\begin{lm} \label{lm:L-smoothness-sufficient-condition}
Suppose that $\Psi: \Rpsd^{d \times d} \to \R \cup \{\infty\}$ is a mapping such that $\Psi|_{\dom\Psi}$ is  differentiable. Suppose further that $\mathcal{M} \subset \dom\Psi$ is a convex subset and the derivative mapping $\mathcal{M} \ni M \mapsto D\Psi(M)$ restricted to $\mathcal{M}$ is Lipschitz continuous, that is, there exists an $L \in [0,\infty)$ such that
\begin{align} \label{eq:derivative-lipschitz-on-M}
\big| \big( D\Psi(M') - D\Psi(M) \big)E \big| \le L |M'-M| |E|
\qquad (M,M' \in \mathcal{M} \text{ and } E \in \R^{d\times d}).
\end{align}
Then $\Psi|_{\mathcal{M}}$ is $L$-smooth w.r.t.~$|\cdot|$. 
\end{lm}

\begin{proof}
Since $\Psi|_{\dom\Psi}$ is differentiable and $\mathcal{M}$ is a convex subset of $\dom\Psi$ by assumption, it follows by~\eqref{eq:derivative-lipschitz-on-M} that $[0,1] \ni t \mapsto \Psi(M+t(N-M))$ is continuously differentiable for all $M,N \in \mathcal{M}$. So, by the mean value theorem and~\eqref{eq:derivative-lipschitz-on-M}, we conclude 
\begin{align}
\Psi(N) - \Psi(M) - D\Psi(M)(N-M) 
&\le \int_0^1 \big|\big(D\Psi(M+t(N-M)) - D\Psi(M)\big)(N-M)\big| \d t \notag\\
&\le L \int_0^1 t \d t \cdot |N-M|^2 = \frac{L}{2} |N-M|^2
\end{align}
for all $M,N \in \mathcal{M}$. And therefore, $\Psi|_{\mathcal{M}}$ is $L$-smooth w.r.t.~$|\cdot|$, as desired.
\end{proof}

\subsubsection{Support and convergence of probability measures}

As usual, a point $x$ of some metric space $X$ is called a support point of the probability measure $\xi \in \Xi(X)$ iff $\xi(B_{\delta}(x)) > 0$ for every $\delta > 0$ (where $B_{\delta}(x)$ denotes the open $\delta$-ball around $x$ in $X$, of course). And accordingly, the set of all support points of a measure $\xi \in \Xi(X)$ is called its support, in short: %$\supp \xi := \{x \in X: x \text{ is a support point of } \xi\}$.
\begin{align}
\supp \xi := \{x \in X: x \text{ is a support point of } \xi\}.
\end{align}

\begin{lm} \label{lm:support}
Suppose that $X$ is a compact metric space and let $\xi \in \Xi(X)$. Then $\supp \xi$ is the largest closed subset $C$ of $X$ with $\xi(X \setminus C) = 0$. In particular,
\begin{align} \label{eq:support}
\xi(A) = \xi(A \cap \supp \xi) \qquad (A \in \mathcal{B}_X).
\end{align} 
Additionally, if $\supp \xi$ is finite, then $\xi(\{x\}) > 0$ for every $x \in \supp \xi$ and 
\begin{align} \label{eq:repr-of-measures-with-finite-supp}
\xi = \sum_{x\in\supp \xi} \xi(\{x\}) \delta_x.
\end{align}
\end{lm}

\begin{proof}
As a compact metric space, $X$ is in particular separable. %(Theorem~IX.1.8 of~\cite{AmannEscher})
Consequently, the first part of the lemma (up until~\eqref{eq:support}) follows from Theorem~II.2.1 of~\cite{Parthasarathy} (note that a measure in~\cite{Parthasarathy} is always meant to be a probability measure, see the footnote in Section~II.1 of~\cite{Parthasarathy}). 
Suppose now that $\supp \xi$ is finite. Then, for every $x \in \supp \xi$, there is a $\delta > 0$ such that $\{x\} = B_{\delta}(x) \cap \supp \xi$ and therefore, by~\eqref{eq:support} and the definition of support points, 
\begin{align*}
\xi(\{x\}) = \xi\big(B_{\delta}(x) \cap \supp \xi \big) =  \xi\big(B_{\delta}(x) \big) > 0.
\end{align*}
Additionally, it follows by~\eqref{eq:support} that
\begin{align*}
\xi(A) = \xi(A \cap \supp \xi) = \sum_{x\in\supp\xi} \xi(A \cap \{x\}) = \sum_{x\in\supp\xi} \xi(\{x\}) \delta_x(A) \qquad (A \in \mathcal{B}_X)
\end{align*}
and therefore~\eqref{eq:repr-of-measures-with-finite-supp} holds true. 
\end{proof}

We recall that a sequence $(\xi_n)$ of probability measures in $\Xi(X)$ with some metric space $X$ is said to converge weakly (or in distribution) to a probability measure $\xi \in \Xi(X)$ iff 
\begin{align}
\int_X m(x) \d\xi_n(x) \longrightarrow \int_X m(x)\d\xi(x) \qquad (n\to\infty)
\end{align}  
for every $m \in C_{\mathrm{b}}(X,\R) := \{\text{bounded continuous functions } X \to \R\}$ (Section~II.6 of~\cite{Parthasarathy}). 

\begin{lm}[Prohorov] \label{lm:prohorov}
If $X$ is a compact metric space, then $\Xi(X)$ endowed with the weak topology %(topology of distributional convergence) 
is a metrizable compact space. In particular, $\Xi(X)$ is sequentially compact. 
\end{lm}

\begin{proof}
Invoke Theorem~II.6.4 of \cite{Parthasarathy} and recall that for metric spaces, compactness and sequential compactness are equivalent (Theorem~III.3.4 of \cite{AmannEscher}). %See also Korollar 13.30 (Klenke)
\end{proof}

\subsection{Common design criteria}
%\subsection{Standard design criteria}

%\subsection{Standard properties of design criteria}

In this section, we recall the %most popular 
design criteria $\Psi$ %from~\eqref{eq:doe-intro}
that are most commonly used in practice and record their monotonicity, convexity, continuity, and differentiability properties. %standard properties like antitonicity, convexity, continuity and directional differentiability.
In particular, we establish the strong convexity of a large class of practically relevant design criteria. 

%antitonicity and homogeneity 
%convexity, strict convexity, 
%lower semicontinuity, 
%strict directional differentiability 

\subsubsection{Simple design criteria}

%D-criterion, A-criterion, E-criterion, $\Psi_{\gamma}$-criterion,
%weighted A-criterion ($\tr(WM^{-1})$)
%definitions and standard properties (in particular, antitonicity, (strict) convexity, lower semi-continuity, strict directional differentiability and strict directional derivatives (sensitivity functions)

We begin by discussing a scale of design criteria $\Psi_p$ that comprises the D- and the A-criterion as special cases. Specifically, $\Psi_p$ is defined by
%\begin{align}
%\Psi_0(M) := \begin{cases}
%\log \det(M^{-1}), \qquad M \in \Rpd^{d \times d} \\
%\infty, \qquad M \in \Rpsd^{d \times d} \setminus \Rpd^{d \times d}
%\end{cases}
%\quad \text{and} \quad
%\Psi_p(M) := \begin{cases}
%(\tr(M^{-p}))^{1/p}, \qquad M \in \Rpd^{d \times d} \\
%\infty, \qquad M \in \Rpsd^{d \times d} \setminus \Rpd^{d \times d}
%\end{cases},
%\end{align}
\begin{align} \label{eq:Psi_p-def}
\Psi_0(M) := \log \det(M^{-1})
\qquad \text{and} \qquad
\Psi_p(M) := (\tr(M^{-p}))^{1/p}
\qquad (p \in (0,\infty))
\end{align}
for $M \in \Rpd^{d \times d}$ and by
\begin{align} \label{eq:Psi_p-def-2}
\Psi_p(M) := \infty \qquad (p \in [0,\infty))
\end{align}
for $M \in \Rpsd^{d \times d} \setminus \Rpd^{d \times d}$. Commonly, $\Psi_0$ and $\Psi_1$ are called the (logarithmic) D-criterion and the A-criterion, respectively. With the help of a spectral decomposition of a given $M \in \Rpd^{d \times d}$, we immediately see that
\begin{align} \label{eq:Psi_p-eigenvalue-representation}
\Psi_0(M) = -\sum_{i=1}^d \log(\lambda_i(M))
\qquad \text{and} \qquad
\Psi_p(M) = \bigg( \sum_{i=1}^d \lambda_i(M)^{-p} \bigg)^{1/p},
\end{align}
where $\lambda_1(M), \dots, \lambda_d(M)$ are the eigenvalues of $M$ counted according to their multiplicities and ordered increasingly. It is also easy to verify that $(1/d) \Psi_0(M) = \lim_{p\searrow 0} \log\big((1/d)^{1/p} \Psi_p(M)\big)$ 
%\begin{align}
%(1/d) \Psi_0(M) = \lim_{p\searrow 0} \log\big((1/d)^{1/p} \Psi_p(M)\big)
%\end{align}
(just calculate the derivative of $[0,\infty) \ni p \mapsto \log\big((1/d) \sum_{i=1}^d \lambda_i(M^{-1})^p$ at $p =0$).

\begin{lm} \label{lm:upper-bound-on-Psi(M)-implies-lower-bound-on-M}
Suppose that $p \in [0,\infty)$ and $R \in (0,\infty)$. If $M \in \R^{d\times d}_{\mathrm{psd}}$ with $\Psi_p(M) \le R$, then %$M \ge\mu_{p,R}(\norm{M})$, 
\begin{align}
M \ge\mu_{p,R}(\norm{M}),
\end{align}
where $\mu_{0,R}(\norm{M}) := \e^{-R} \norm{M}^{-(d-1)}$ and $\mu_{p,R}(\norm{M}) := 1/R$ for $p \in (0,\infty)$.
\end{lm}

\begin{proof}
Suppose first that $p = 0$ and that $M \in \Rpsd^{d \times d}$ with $\Psi_0(M) \le R < \infty$. It then follows by~\eqref{eq:Psi_p-def}-\eqref{eq:Psi_p-eigenvalue-representation} that $M \in \Rpd^{d \times d}$ and 
\begin{align}
\log\lambda_i(M) \ge -R -\sum_{j\ne i} \log\lambda_j(M) \ge -R -(d-1)\log\norm{M}
\qquad (i\in\{1,\dots,d\}). 
\end{align}
%because $\lambda_j(M) \le \norm{M}$ by the spectral theorem.
And therefore we obtain the claimed estimate in the case $p = 0$ as follows:
\begin{align}
M \ge \min\{\lambda_i(M): i \in \{1,\dots,d\}\} \ge \e^{-R} \norm{M}^{-(d-1)} =\mu_{0,R}(\norm{M}),
\end{align}
Suppose now that $p \in (0,\infty)$ and that $M \in \Rpsd^{d \times d}$ with $\Psi_p(M) \le R < \infty$. It then follows by~\eqref{eq:Psi_p-def}-\eqref{eq:Psi_p-eigenvalue-representation} that $M \in \Rpd^{d \times d}$ and
\begin{align}
1/\lambda_i(M) \le \bigg( \sum_{i=1}^d \lambda_i(M)^{-p} \bigg)^{1/p} \le R
\qquad (i\in\{1,\dots,d\}). 
\end{align}
And therefore we obtain 
\begin{align}
M \ge \min\{\lambda_i(M): i \in \{1,\dots,d\}\} \ge 1/R =\mu_{p,R}(\norm{M}),
\end{align}
which is the claimed estimate in the case $p \in (0,\infty)$. 
\end{proof}

\begin{prop} \label{prop:standard-properties-for-Psi_p}
Suppose that $p \in [0,\infty)$. Then $\Psi_p$ is antitonic, strictly convex, and lower semicontinuous with $\dom \Psi_p = \R^{d\times d}_{\mathrm{pd}}$. Additionally, the restriction $\Psi_p|_{\R^{d\times d}_{\mathrm{pd}}}$ is %continuous and directionally differentiable with directional derivative given by 
continuously differentiable with derivative given by
\begin{gather}
D\Psi_0(M)E = -\tr(M^{-1}E)
\label{eq:directional-derivative-p=0}\\
D\Psi_p(M)E = -\big(\tr(M^{-p})\big)^{1/p-1} \tr(M^{-p-1}E) \qquad (p\in(0,\infty))
\label{eq:directional-derivative-p>0}
\end{gather}
for $M \in \R^{d\times d}_{\mathrm{pd}}$ and $E \in \R^{d\times d}$. In particular, the derivative map $\Rpd^{d \times d} \ni M \mapsto D\Psi_p(M)$ is locally Lipschitz continuous.
\end{prop}

\begin{proof}
It is clear by the definition~\eqref{eq:Psi_p-def}-\eqref{eq:Psi_p-def-2} that $\dom \Psi_p = \R^{d\times d}_{\mathrm{pd}}$ and that $\Psi_p|_{\dom \Psi_p}$ is continuous (Lemma \ref{lm:directional-derivatives}) for $p \in [0,\infty)$. With this continuity property and Lemma~\ref{lm:upper-bound-on-Psi(M)-implies-lower-bound-on-M}, in turn, the lower semicontinuity of $\Psi_p$ for $p \in [0,\infty)$ easily follows. %(Indeed, let $M_n, M \in \R^{d\times d}_{\mathrm{psd}}$ be such that $M_n \longrightarrow M$ and $\liminf_{n\to\infty} \Psi_p(M_n) < \infty$. Lemma~\ref{lm:upper-bound-on-Psi(M)-implies-lower-bound-on-M} then shows that $M \in \R^{d\times d}_{\mathrm{pd}}$ and therefore $\Psi_p(M) = \lim_{n\to\infty} \Psi_p(M_n) = \liminf_{n\to\infty} \Psi_p(M_n)$.)
Also, the antitonicity of $\Psi_p$ for $p \in [0,\infty)$ immediately follows by~\eqref{eq:Psi_p-eigenvalue-representation} and the well-known isotonicity of the eigenvalue maps $M \mapsto \lambda_i(M)$ for $i \in \{1,\dots,d\}$ (Theorem~11.9 of~\cite{MagnusNeudecker}). 
In order to see the directional differentiability of $\Psi_p|_{\dom\Psi_p}$ and the formulas~\eqref{eq:directional-derivative-p=0} and~\eqref{eq:directional-derivative-p>0}, we have only to apply Lemma~\ref{lm:directional-derivatives}. It is clear from \eqref{eq:directional-derivative-p=0} and~\eqref{eq:directional-derivative-p>0} that $$\R^{d\times d} \ni E \mapsto \partial_E \Psi_p(M) \in \R$$ is a bounded linear map $D\Psi_p(M)$ for every $M \in  \R^{d\times d}_{\mathrm{pd}}$ and that $\R^{d\times d}_{\mathrm{pd}} \ni M \mapsto D\Psi_p(M)$, in turn, is continuous (Lemma~\ref{lm:directional-derivatives}). Consequently, $\Psi_p|_{\dom\Psi_p}$ is continuously (Fr\'{e}chet) differentiable %(in the Fr\'{e}chet sense) 
by a standard differentiability criterion (Proposition~4.8 of~\cite{Zeidler}, for instance). Additionally, $\Rpd^{d \times d} \ni M \mapsto D\Psi_p(M)$ is locally Lipschitz continuous by the formulas~\eqref{eq:directional-derivative-p=0} and~\eqref{eq:directional-derivative-p>0} of Lemma~\ref{lm:directional-derivatives} and the last part of Lemma~\ref{lm:directional-derivatives}. 
%
%Antitonicity and lower semicontinuity of the $\Psi_p$ for all $p \in [0,\infty)$ easily follow by~\eqref{eq:Psi_p-eigenvalue-representation}. Also, $\dom \Psi_p = \R^{d\times d}_{\mathrm{pd}}$ is obvious by the definition of the $\Psi_p$. 
%
It remains to establish the strict convexity of $\Psi_p$ for all $p \in [0,\infty)$. Strict convexity of $\Psi_0$ follows by Theorem~11.25 in~\cite{MagnusNeudecker}. Convexity of $\Psi_p$ in the case $p \in [1,\infty)$ follows from~\eqref{eq:Psi_p-eigenvalue-representation} by noting first that
\begin{align} \label{eq:standard-properties-for-Psi_p,1}
\lambda_i\big( (\alpha M + (1-\alpha) N)^{-1} \big) 
\le 
\lambda_i\big( \alpha M^{-1} + (1-\alpha) N^{-1} \big) 
\qquad (i \in \{1,\dots,d\})
\end{align}
for arbitrary $M, N \in \R^{d\times d}_{\mathrm{pd}}$ and all $\alpha \in [0,1]$ by virtue of the convexity of matrix inversion on $\R^{d\times d}_{\mathrm{pd}}$ (Corollary~V.2.6 of~\cite{Bhatia}) and by applying then the triangle inequality for the $p$-Schatten norms (Theorem~11.26 of~\cite{MagnusNeudecker} which is valid for $p \in [1,\infty)$) to the right-hand side of~\eqref{eq:standard-properties-for-Psi_p,1}. Strict convexity of $\Psi_p$ for $p \in [1,\infty)$ can be seen as follows: for $M \ne N$ and $\alpha = 1/2$ one has strict inequality in~\eqref{eq:standard-properties-for-Psi_p,1} for at least one $i \in  \{1,\dots,d\}$ (this follows by the explicit formula from Exercise~V.1.15 of~\cite{Bhatia}) and thus the same arguments as for mere convexity above yield strict mid-point convexity which, in turn, proves strict convexity (Lemma~\ref{lm:strict-convexity}). Strict convexity of $\Psi_p$ in the case $p \in (0,1)$ is stated in the first and second paragraph on page 864 of~\cite{Ki74} -- a proof of this is indicated on page~863 (Case 2) of~\cite{Ki74}. 
%
%In view of the convexity of the $\Psi_p$ just established, the continuity and directional differentiability of $\Psi_p|_{\R^{d\times d}_{\mathrm{pd}}}$ now follow immediately by Lemma~\ref{lm:automatic-cont-and-directional-db} in conjunction with Lemma~\ref{lm:rint(pos-def-matrices)}.  
%
%Also, the stated formulas for the directional derivatives immediately follow by Lemma~\ref{lm:directional-derivatives}. 
\end{proof}

In some cases, one is interested in further generalizations of the design criteria $\Psi_p$, namely in the design criteria $\Psi_{p,Q}: \Rpsd^{d \times d} \to \R \cup \{\infty\}$ defined by
\begin{align} \label{eq:Psi_p,Q-def}
\Psi_{0,Q}(M) := \log \det(Q^\top M^- Q) 
\qquad \text{and} \qquad
\Psi_{p,Q}(M) := (\tr(Q^\top M^- Q)^p)^{1/p}
%\qquad (p \in (0,\infty))
\end{align}
for $p \in (0,\infty)$ and for all $M \in \Rpsd^{d \times d}$ with $\ran(Q) \subset \ran(M)$ and defined by 
\begin{align}
\Psi_{p,Q}(M) := \infty \qquad (p \in [0,\infty))
\end{align}
else. See~\cite{Ki74} or~\cite{PrPa} (Section~5.1.2). In the above definitions, $Q \in \R^{d \times s}$ with some $s \le d$ and $M^-$ denotes any generalized inverse of $M$ (Section 9.1 of~\cite{Harville}), while $\ran(A)$ denotes the range of a matrix $A$. % or Exercise 2 in Section~2.9 of~\cite{MagnusNeudecker}
Such design criteria arise, for instance, if one only wants to estimate a linear function $Q^\top \theta$ of the model parameter (as opposed to the whole model parameter vector $\theta$).

\begin{cor} \label{cor:standard-properties-for-Psi_p,Q}
Suppose that $p \in [0,\infty)$ and $Q \in \R^{d \times s}$ with full column rank $\rk Q = s \le d$. Then $\Psi_{p,Q}$ is antitonic, convex, and lower semicontinuous. Additionally, the restriction $\Psi_{p,Q}|_{\Rpd^{d \times d}}$ is continuously differentiable and the derivative map $$\Rpd^{d \times d} \ni M \mapsto D\Psi_{p,Q}(M)$$ is locally Lipschitz continuous.
\end{cor}

\begin{proof}
Since $Q$ has full column rank $s$ by assumption, it follows by Theorem~3.15 of~\cite{Pu} that $\Psi_{p,Q}$ can be expressed in terms of information matrices $C_Q(M)$ (Section~3.2 of~\cite{Pu}). Specifically,
\begin{align} \label{eq:Psi-p,Q-alternative-representation}
\Psi_{p,Q}(M) = \Psi_p(C_Q(M)) \qquad (M \in \Rpsd^{d \times d}).
\end{align}
And from this representation, in turn, the antitonicity, convexity and lower semicontinuity of $\Psi_{p,Q}$ can be %easily 
concluded using the corresponding properties of $\Psi_p$ (Proposition~\ref{prop:standard-properties-for-Psi_p}) and of the information matrix mapping $C_Q$ (Theorem~3.13 of~\cite{Pu}). (In order to obtain the lower semicontinuity of $\Psi_{p,Q}$, assume $M_n, M \in \Rpsd^{d \times d}$ with $M_n \longrightarrow M$. Setting $M_n' := M_n + \norm{M-M_n}$, we observe that $M_n' \longrightarrow M$ as $n \to \infty$ and that $M_n' \ge M$ %because $\norm{M_n - M} \ge M_n - M$
and hence $C_Q(M_n') \ge C_Q(M)$ for all $n \in \N$ by Theorem~3.13 of~\cite{Pu}. Applying Theorem~3.13.b of~\cite{Pu} to $(M_n')$, we then arrive at the desired lower semicontinuity estimate, using the lower semicontinuity of $\Psi_p$ in conjunction with the antitonicity of $\Psi_p$ and the estimate $C_Q(M_n') \ge C_Q(M_n)$.) %$\Psi_{p,Q}(M) = \Psi_p(C_Q(M)) \le \liminf_{n\to\infty} \Psi_p(C_Q(M_n')) \le \liminf_{n\to\infty} \Psi_p(C_Q(M_n)) = \liminf_{n\to\infty} \Psi_{p,Q}(M_n)$
Additionally, from~\eqref{eq:Psi-p,Q-alternative-representation}, we easily obtain the continuous differentiability of $\Psi_{p,Q}|_{\Rpd^{d \times d}}$ and the local Lipschitz continuity of the derivative map. We have only to use that
\begin{align} \label{eq:C_Q-alternative-representation}
C_Q(M) = \big( Q^\top M^{-1} Q \big)^{-1} \in \Rpd^{d \times d} \qquad (M \in \Rpd^{d \times d})
\end{align}  
(Section~3.3 and Theorem~3.15 of~\cite{Pu}) and combine this with the continuous differentiability of $\Psi_p|_{\Rpd^{d \times d}}$ and the local Lipschitz continuity of the derivative map (Proposition~\ref{prop:standard-properties-for-Psi_p}). 
\end{proof}

Sometimes it is more convenient to work with the modified design criteria $\tilde{\Psi}_p$ and $\tilde{\Psi}_{0,Q}$ defined by
%$\tilde{\Psi}_0(M) := \Psi_0(M)$ and $\tilde{\Psi}_p(M) := \Psi_p(M)^p$
\begin{align} \label{eq:tilde-Psi_p-def}
\tilde{\Psi}_0(M) := \Psi_0(M)
\qquad \text{and} \qquad
\tilde{\Psi}_p(M) := (\Psi_p(M))^p
\qquad (p \in (0,\infty))
\end{align}
and, respectively, by $\tilde{\Psi}_{0,Q}(M) := \Psi_{0,Q}(M)$ and $\tilde{\Psi}_{p,Q}(M) := (\Psi_{p,Q}(M))^p$. See~\cite{YaBiTa13}, for instance. Since $[0,\infty) \ni t \mapsto t^p$ is convex and monotonically increasing for $p \in [1,\infty)$, it follows by Proposition~\ref{prop:standard-properties-for-Psi_p} and Corollary~\ref{cor:standard-properties-for-Psi_p,Q} that for $p \in [1,\infty)$ the modified criteria $\tilde{\Psi}_p$ and $\tilde{\Psi}_{p,Q}$ are antitonic, lower semicontinuous and convex as well. As we will show now, for $p \in \N_0$ the modified criteria $\tilde{\Psi}_p$ are even strongly convex.

\begin{cor} \label{cor:strongly-convex-design-criteria}
Suppose $C \in (0,\infty)$ and let $\mathcal{M} := \{A \in \Rpd^{d \times d}: A \le C\}$. Then $\tilde{\Psi}_p|_{\mathcal{M}}$ for every $p \in \N_0$ is $\mu$-strongly convex w.r.t.~$|\cdot|$, where $\mu := \max\{1,p\}(p+1)/C^{p+2}$. In particular, $\Psi_0|_{\mathcal{M}}$ and $\Psi_1|_{\mathcal{M}}$ are strongly convex.
\end{cor}

\begin{proof}
Suppose $p \in \N_0$ and write $\Psi := \tilde{\Psi}_p$ and $\eps_p := \max\{1,p\}$ for brevity. Also, let $M, N \in \mathcal{M}$ and write $E := N-M$ and $M_t := M + tE$ for $t \in [0,1]$. Since $\Psi|_{\dom \Psi}$ is continuously differentiable (Proposition~\ref{prop:standard-properties-for-Psi_p}) and $\mathcal{M}$ is a convex subset of $\dom\Psi = \Rpd^{d \times d}$, it follows by the mean value theorem that
\begin{align} \label{eq:strongly-convex-design-criteria,1}
\Psi(N) &- \Psi(M) - D\Psi(M)(N-M)
= \int_0^1 \big( D\Psi(M_t)-D\Psi(M) \big) E \d t \notag\\
&= \eps_p \int_0^1 \tr\big( (M^{-p-1} - M_t^{-p-1})E \big) \d t
= \eps_p \sum_{q=1}^{p+1} \int_0^1 t \cdot \tr\big( M^{-q} E M_t^{-(p+2-q)} E \big) \d t,
\end{align}  
where %for the second equality 
we used that $D\Psi(A)E = -\eps_p \tr(A^{-p-1}E)$ for $A \in \Rpd^{d \times d}$  (Lemma~\ref{lm:directional-derivatives}) and %for the third equality we used
that $$M^{-p-1} - M_t^{-p-1} = t \sum_{q=1}^{p+1} M^{-q} E M_t^{-(p+2-q)}$$ for all $t \in [0,1]$ (induction over $p \in \N_0$). Since $M, M_t \in \mathcal{M}$, it further follows that
\begin{align} \label{eq:strongly-convex-design-criteria,2}
M_t^{-k} \ge C^{-k} 
\qquad \text{and} \qquad
M^{-k} \ge C^{-k}
\qquad (k \in \N \text{ and } t \in [0,1]).
\end{align}
With the help of cyclic permutations under the trace and of~\eqref{eq:strongly-convex-design-criteria,2} we conclude that
\begin{align} \label{eq:strongly-convex-design-criteria,3}
\tr\big( M^{-q} E M_t^{-(p+2-q)} E \big) 
&= \tr\big( (E M^{-q/2})^\top M_t^{-(p+2-q)} E M^{-q/2} \big) \notag\\
&\ge C^{-(p+2-q)}  \cdot \tr\big( (E M^{-q/2})^\top E M^{-q/2} \big)
= C^{-(p+2-q)}  \cdot \tr\big( E M^{-q} E \big) \notag\\
&\ge C^{-(p+2-q)} C^{-q} \cdot \tr(E^2) = C^{-(p+2)} |E|^2 
\end{align}
for every $t \in [0,1]$ and $q \in \{1,\dots,p+1\}$. Inserting now~\eqref{eq:strongly-convex-design-criteria,3} into~\eqref{eq:strongly-convex-design-criteria,1}, we immediately obtain the claimed $\eps_p (p+1)/C^{p+2}$-strong convexity of $\tilde{\Psi}_p|_{\mathcal{M}}$ w.r.t.~$|\cdot|$.
\end{proof}

%Similar properties hold for the weighted A-criterion (Corollary~\ref{cor:weighted-A-criterion}) and for the E-criterion (Corollary~\ref{cor:E-criterion}).
%
%\begin{lm} \label{lm:standard-properties-for-Psi_u}
%Suppose that $U$ is a compact subset of $\R^d$ and, for every $u \in U$, let $\Psi_u(M) := u^{\top} M^{-1} u$ for $M \in \R^{d\times d}_{\mathrm{pd}}$ and $\Psi_u(M) := \infty$ for $M \in \Rpsd^{d \times d} \setminus \Rpd^{d \times d}$. Then $\Psi_u$ is antitonic, convex, and the restriction $\Psi_u|_{ \R^{d\times d}_{\mathrm{pd}}}$ is directionally differentiable uniformly w.r.t.~$u \in U$ with directional derivative given by
%\begin{align}
%\partial_E \Psi_u(M) = -u^{\top}M^{-1} E M^{-1}u
%\end{align}
%for $M \in \R^{d\times d}_{\mathrm{pd}}$ and $E \in \R^{d\times d}$. 
%\end{lm}
%
%\begin{proof}
%Convexity of $\Psi_u$ is an immediate consequence of the convexity of matrix inversion on $\R^{d\times d}_{\mathrm{pd}}$ (Corollary~V.2.6 of~\cite{Bhatia}). All other assertions are straightforward to verify.
%\end{proof}
%
%It is easy to see that, in general, $\Psi_u$ is not strictly convex and not lower semicontinuous. See the design criterion from~\eqref{eq:counterex,2} below, for instance.

\subsubsection{Composite design criteria}

%Standard properties carry over to averages and to suprema of simple design criteria, under certain assumptions (lower semicontinuity carries over only under the quite restrictive assumption that $\dom \Psi(\cdot,u) = \R^{d\times d}_{\mathrm{pd}}) for all $u \in U$)

We now discuss three common ways of constructing new design criteria from more basic ones, namely (i) incorporating information from a previous stage, (ii) taking averages, and (iii) taking suprema. In the case of (i), one speaks of two-stage design criteria. Important special cases of (ii) and (iii) are the weighted A-criterion~\eqref{eq:def-weighted-A-criterion} and the E-criterion~\eqref{eq:def-E-criterion}. 

%\subparagraph{Two-stage design criteria}

\begin{prop} \label{prop:two-stage-design-criteria}
Suppose that $\Psi: \Rpsd^{d\times d} \to \R \cup \{\infty\}$ is any mapping (design criterion) and let $M_0 \in \Rpsd^{d\times d}$ and $\alpha \in [0,1)$. Also, let $\Psi^{(\alpha)}: \Rpsd^{d\times d} \to \R \cup \{\infty\}$ be the corresponding two-stage design criterion, that is,
\begin{align} \label{eq:two-stage-design-criterion}
\Psi^{(\alpha)}(M) := \Psi\big(\alpha M^0 + (1-\alpha) M\big)
\qquad (M \in \Rpsd^{d \times d}).
\end{align}
Convexity, lower semicontinuity and antitonicity then carry over from $\Psi$ to $\Psi^{(\alpha)}$. Additionally, if $\Psi|_{\dom\Psi}$ is differentiable, then $\Psi^{(\alpha)}|_{\dom\Psi^{(\alpha)}}$ is differentiable as well with derivative given by
\begin{align} \label{eq:two-stage-design-criteria-derivative}
D\Psi^{(\alpha)}(M)E = (1-\alpha) D\Psi\big(\alpha M^0 + (1-\alpha) M\big) E
\qquad (M \in \dom\Psi^{(\alpha)} \text{ and } E \in \R^{d \times d}). 
\end{align}
\end{prop}

\begin{proof}
Straightforward verifications using our definition of differentiability and of derivatives of functions on arbitrary (not necessarily open) subsets of the vector space $\R^{d \times d}$. See the remarks around~\eqref{eq:derivative-def}.
\end{proof}

%\subparagraph{Averages of design criteria}

\begin{prop} \label{prop:sum-of-design-criteria}
Suppose that $\Psi_u: \Rpsd^{d \times d} \to \R \cup \{\infty\}$ is convex with $\dom \Psi_u = \Rpd^{d \times d}$ for every $u \in U$, where $U$ is a compact metric space. Suppose further that $\Rpd^{d \times d} \times U \ni (M,u) \mapsto \Psi_u(M) \in \R$ is continuous and that $\ol{\Psi}: \Rpsd^{d \times d} \to \R \cup \{\infty\}$ is an average of the $\Psi_u$, that is, 
\begin{align}
\ol{\Psi}(M) := \int_U \Psi_u(M) \d \zeta(u) \qquad (M \in \Rpsd^{d \times d})
\end{align}
with some finite measure $\zeta$ on $\mathcal{B}_U$. Then $\ol{\Psi}|_{\Rpd^{d \times d}}$ is convex and continuous and $\dom \ol{\Psi} = \Rpd^{d \times d}$. Additionally, the following assertions hold true:
\begin{itemize}
\item[(i)] If $\Psi_u$ is antitonic for every $u \in U$, then $\ol{\Psi}$ is antitonic as well.
\item[(ii)] If $\Psi_u$ is lower semicontinuous and antitonic for every $u \in U$, then $\ol{\Psi}$ is lower semicontinuous as well.
\item[(iii)] If $\Psi_u|_{\Rpd^{d \times d}}$ is directionally differentiable uniformly w.r.t.~$u \in U$, then $\ol{\Psi}|_{\Rpd^{d \times d}}$ is directionally differentiable as well with
\begin{align}
\partial_E \ol{\Psi}(M) = \int_U \partial_E \Psi_u(M) \d \zeta(u) 
\qquad (M \in \Rpd^{d \times d} \text{ and } E \in \R^{d \times d}).
\end{align}
\end{itemize}
\end{prop}

\begin{proof}
Simple verification using %Lemma~\ref{lm:automatic-cont-and-directional-db} above along with 
Fatou's lemma in the version for a sequence of measurable, not necessarily non-negative, functions with an integrable minorant. (At first glance, the antitonicity assumption in~(ii) might seem superfluous. We impose it because then for every convergent sequence $(M_n)$ in $\Rpsd^{d \times d}$, we have %$\Psi_u(M_n) \ge \Psi_u(C)$ for all $n \in \N$ and $u \in U$, 
\begin{align*}
\Psi_u(M_n) \ge \Psi_u(C) \qquad (n\in \N \text{ and } u \in U),
\end{align*}
where $C \in \Rpd^{d \times d}$ is an arbitrary positive definite upper bound of the $M_n$ w.r.t.~the standard partial order on $\Rpsd^{d \times d}$. And therefore $u \mapsto \Psi_u(C)$ is a continuous, hence $\zeta$-integrable, minorant of the functions $u \mapsto \Psi_u(M_n)$, whence Fatou's lemma in the mentioned version can be applied.) 
\end{proof}

%Weighted A-criterion as special case
As a simple application of the above result, one can derive standard properties of the weighted A-criterion defined by 
\begin{align} \label{eq:def-weighted-A-criterion}
%\Psi_W(M) := \tr(WM^{-1}) \quad (M \in \Rpd^{d \times d}) \quad \text{and} \quad \Psi_W(M) := \infty \quad (M \in \Rpsd^{d \times d} \setminus \Rpd^{d \times d})
\Psi_W(M) := 
\begin{cases}
\tr(WM^{-1}), \qquad M \in \Rpd^{d \times d} \\
\infty, \qquad M \in \Rpsd^{d \times d} \setminus \Rpd^{d \times d}
\end{cases},
\end{align} 
where $W \in \Rpsd^{d \times d}$ is some fixed positive semidefinite matrix.

\begin{prop} \label{prop:supremum-of-design-criteria}
Suppose that $\Psi_u: \Rpsd^{d \times d} \to \R \cup \{\infty\}$ is convex with $\dom \Psi_u = \Rpd^{d \times d}$ for every $u \in U$, where $U$ is a compact metric space. Suppose further that $\Rpd^{d \times d} \times U \ni (M,u) \mapsto \Psi_u(M) \in \R$ is continuous and that $\ol{\Psi}: \Rpsd^{d \times d} \to \R \cup \{\infty\}$ is the supremum of the $\Psi_u$, that is, 
\begin{align}
\ol{\Psi}(M) := \sup_{u\in U} \Psi_u(M) \qquad (M \in \Rpsd^{d \times d}).
\end{align}
Then $\ol{\Psi}|_{\Rpd^{d \times d}}$ is convex and continuous and $\dom \ol{\Psi} = \Rpd^{d \times d}$. Additionally, the following assertions hold true:
\begin{itemize}
\item[(i)] If $\Psi_u$ is antitonic for every $u \in U$, then $\ol{\Psi}$ is antitonic as well.
\item[(ii)] If $\Psi_u$ is lower semicontinuous for every $u \in U$, then $\ol{\Psi}$ is lower semicontinuous as well.
\item[(iii)] If $\Psi_u|_{\Rpd^{d \times d}}$ is directionally differentiable uniformly w.r.t.~$u \in U$, then $\ol{\Psi}|_{\Rpd^{d \times d}}$ is directionally differentiable as well with
\begin{align}
\partial_E \ol{\Psi}(M) = \sup_{u \in U(M)} \partial_E \Psi_u(M) 
\qquad (M \in \Rpd^{d \times d} \text{ and } E \in \R^{d \times d}),
\end{align}
where $U(M) := \argmax_{u\in U} \Psi_u(M) = \{u \in U: \Psi_u(M) = \ol{\Psi}(M)\}$.
\end{itemize}
\end{prop}

\begin{proof}
Simple verification using %Lemma~\ref{lm:automatic-cont-and-directional-db} above along with 
Theorem~3.2 from~\cite{Pshenichnyi} in the slightly generalized version where the normed space $B$ is replaced by an open subset of a normed space. (It is easy to see that the proofs in~\cite{Pshenichnyi} remain valid in this more general situation.)
\end{proof}

%E-criterion as a special case
As a simple application of the above result, one can derive standard properties of the E-criterion defined by
\begin{align} \label{eq:def-E-criterion}
\Psi_{\infty}(M) := 
\begin{cases}
\lambda_{\mathrm{max}}(M^{-1}), \qquad M \in \Rpd^{d \times d} \\
\infty, \qquad M \in \Rpsd^{d \times d} \setminus \Rpd^{d \times d}
\end{cases},
\end{align}
where $\lambda_{\mathrm{max}}(M^{-1}) = 1/\lambda_{\mathrm{min}}(M)$ denotes the largest eigenvalue of $M^{-1}$. In view of~\eqref{eq:Psi_p-eigenvalue-representation} it is clear that $\Psi_{\infty}(M) = \lim_{p \to \infty} \Psi_p(M)$ for every $M \in \Rpsd^{d \times d}$, explaining the notation $\Psi_{\infty}$. 

\subsection{Solvability of optimal design problems} \label{sec:existence}

In this section, we address the solvability of optimal design problems %~\eqref{eq:doe-intro} 
or, more precisely, of %continuous
optimal design problems of the form
\begin{align} \label{eq:doe}
\min_{\xi \in \Xi(X)} \Psi(M(\xi)),
\end{align} 
where $\Psi$ is an appropriate design criterion, $X$ is the design space, $\Xi(X)$ is the set of probability measures on $\mathcal{B}_X$, and $M(\xi)$ is the information matrix of the design $\xi$. %A solution to the optimization problem~\eqref{eq:doe} is called a \emph{$\Psi \circ M$-optimal design}. 
A design $\xi^* \in \Xi(X)$ is called an \emph{$\eps$-optimal design for $\Psi \circ M$} iff it is an $\eps$-approximate solution to~\eqref{eq:doe}, that is,  
\begin{align}
\Psi(M(\xi^*)) \le \inf_{\xi \in \Xi(X)} \Psi(M(\xi)) + \eps.
\end{align}
In particular, a design $\xi^*$ is called an \emph{optimal design for $\Psi \circ M$} iff it is a solution to~\eqref{eq:doe} or, in other words, iff it is an $\eps$-optimal design with $\eps = 0$. 
In all results to come, we will need the following basic assumptions on the design space $X$ and the information matrices $M(\xi)$. %In view of~\eqref{eq:information-matrix-intro}, these assumptions are very mild. 

\begin{cond} \label{cond:X-and-M}
$X$ is a compact metric space, $m \in C(X,\Rpsd^{d \times d})$, and $M: \Xi(X) \to \Rpsd^{d \times d}$ is the corresponding information matrix map defined by
\begin{align}
M(\xi) = \int_X m(x) \d\xi(x) \qquad (\xi \in \Xi(X)). 
\end{align}
\end{cond}

\begin{lm} \label{lm:designs-with-finite-supp}
Suppose that Condition~\ref{cond:X-and-M} is satisfied. Then for every $M \in M(\Xi(X))$, there exists a design $\xi_0 \in \Xi(X)$ with at most $d(d+1)/2 + 1$ support points such that $M(\xi_0) = M$. Additionally, for every boundary point $M \in \partial M(\Xi(X))$, there even exists a design $\xi_0 \in \Xi(X)$ with at most $d(d+1)/2$ support points such that $M(\xi_0) = M$.
\end{lm}

\begin{proof}
Consider the set $S := \{m(x): x \in X\}$. Since $X$ is compact and $m \in C(X,\Rpsd^{d \times d})$, $S$ is a compact subset of the vector space $V := \R^{d \times d}_{\mathrm{sym}}$. Additionally, $m$ is easily seen to be a uniform limit of $\mathcal{B}_X$-measurable functions $m_k$ with finitely many values in $S$. (Indeed, the $m_k$ can be constructed in a similar way as in Corollary~X.1.13 of~\cite{AmannEscher}.)  Consequently, for every $\xi \in \Xi(X)$, we have
\begin{align} \label{eq:designs-with-finite-supp,1}
M(\xi) = \int_X m(x) \d\xi(x) = \lim_{k\to\infty} \int_X m_k(x) \d\xi(x) \in \ol{\conv}(S) = \conv(S),
\end{align}
%(Lemma~\ref{lm:closure-of-convex-hull}). 
where the last equality follows by Theorem~17.2 of~\cite{Rockafellar}. 
And conversely, for every convex combiniation $\sum_{i=1}^n w_i m(x_i) \in \conv(S)$, we have
\begin{align} \label{eq:designs-with-finite-supp,2}
\sum_{i=1}^n w_i m(x_i) %= \int_X m(x) \d\xi(x) 
= M(\xi) \in M(\Xi(X))
\qquad \text{for} \qquad \xi := \sum_{i=1}^n w_i \delta_{x_i} \in \Xi(X).
\end{align}
So, combining~\eqref{eq:designs-with-finite-supp,1} and~\eqref{eq:designs-with-finite-supp,2}, we see that
\begin{align} \label{eq:designs-with-finite-supp,3}
M(\Xi(X)) = \{M(\xi): \xi \in \Xi(X)\} = \conv(S) \subset V.
\end{align}
And therefore the assertions follow by virtue of Lemma~\ref{lm:caratheodory} from~\eqref{eq:designs-with-finite-supp,2} and~\eqref{eq:designs-with-finite-supp,3}, taking into account that $\dim V = d(d+1)/2$ and that $\conv(S) \cap \partial \conv(S) = \partial \conv(S)$ by the last equality in~\eqref{eq:designs-with-finite-supp,1}. 
\end{proof}

%If design criterion $\Psi$ convex and lower semicontinuous and $\Psi(M)$ finite for at least one $M$, then there exists a $\Psi$-optimal information matrix $M^*$. If $\Psi$ even strictly convex, then there exists a unique (exactly one) $\Psi$-optimal information matrix. 

%If design criterion $\Psi$ convex and lower semicontinuous and if $\Xifin(X) \ne \emptyset$ (that is, $\Psi(M(\xi))$ finite for at least one $\xi \in \Xi(X)$), then there exists a $\Psi$-optimal design (with at most $d(d+1)/2 + 1$ support points).  

With the above preliminaries, one easily obtains the following existence result for optimal designs. 
%In this result, we use the notion of strict antitonicity of $\Psi|_{\dom\Psi}$ w.r.t.~the identity, by which we mean that for every $M \in \dom\Psi$ the function $$[0,\infty) \ni \alpha \mapsto \Psi(M+\alpha I) \in \R \cup \{\infty\}$$ is strictly antitonic (with $I$ being the identity matrix). In view of the representations~\eqref{eq:Psi_p-eigenvalue-representation} and~\eqref{eq:def-E-criterion} of $\Psi_p(M)$ in terms of eigenvalues for $p \in [0,\infty) \cup \{\infty\}$, it is clear that this strict antitonicity assumption is satisfied for the D-, the generalized A-, and the E-criterion. %$\Psi_p$ with $p \in \{0\} \cup [1,\infty) \cup \{\infty\}$ (that is, the D-, the generalized A-, and the E-criterion).  
%And, of course, the lower semicontinuity assumption is satisfied as well for these standard criteria (Proposition~\ref{prop:standard-properties-for-Psi_p} and Corollary~\ref{cor:E-criterion}). %the aforementioned criteria.

\begin{thm} \label{thm:ex-of-optimal-designs}
Suppose that Condition~\ref{cond:X-and-M} is satisfied and that $\Psi: \Rpsd^{d \times d} \to \R \cup \{\infty\}$ is lower semicontinuous. Then there exists an optimal design for $\Psi \circ M$ with at most $d(d+1)/2 + 1$ support points. If, in addition, $\Psi$ is antitonic, then there even exists an optimal design for $\Psi \circ M$  with at most $d(d+1)/2$ support points.
%If, in addition, $\Psi|_{\dom\Psi}$ is strictly antitonic w.r.t.~the identity and if %$\Xifin(X) := \dom \Psi \circ M = \{\xi \in \Xi(X): \Psi(M(\xi)) < \infty\}$ 
%\begin{align}
%\Xifin(X) := \dom \Psi \circ M = \{\xi \in \Xi(X): \Psi(M(\xi)) < \infty\}
%\end{align}
%is non-empty, then there even exists a $\Psi \circ M$-optimal design with at most $d(d+1)/2$ support points.
\end{thm}

\begin{proof}
As a first step, we show that there exists an optimal design for $\Psi \circ M$ at all. %whatsoever.
Indeed, let $(\xi^n)$ be a minimizing sequence for $\Psi \circ M$, that is, $\xi^n \in \Xi(X)$ and
\begin{align} \label{eq:ex-of-opt-designs,0}
\Psi(M(\xi^n)) \longrightarrow \inf_{\xi \in \Xi(X)} \Psi(M(\xi)) \qquad (n\to\infty). 
\end{align}
Since $\Xi(X)$ is sequentially compact (Lemma~\ref{lm:prohorov}), there exists a subsequence $(\xi^{n_k})$ and a $\xi^*$ such that
\begin{align} \label{eq:ex-of-opt-designs,1}
\xi^* \in \Xi(X) 
\qquad \text{and} \qquad
\xi^{n_k} \longrightarrow \xi^* \qquad (k\to\infty). 
\end{align}
Since $m \in C(X,\R^{d \times d}_{\mathrm{psd}}) = C_{\mathrm{b}}(X,\Rpsd^{d \times d})$, the relation~\eqref{eq:ex-of-opt-designs,1} also implies
\begin{align} \label{eq:ex-of-opt-designs,2}
M(\xi^{n_k}) = \int_X m(x) \d\xi^{n_k}(x) \longrightarrow \int_X m(x) \d\xi^*(x) = M(\xi^*) \qquad (k\to\infty). 
\end{align}
Since, moreover, $\Psi$ is lower semicontinuous, it follows from~\eqref{eq:ex-of-opt-designs,2} that
\begin{align} \label{eq:ex-of-opt-designs,3}
\Psi(M(\xi^*)) \le \liminf_{k\to\infty} \Psi(M(\xi^{n_k})). 
\end{align}
Inserting now~\eqref{eq:ex-of-opt-designs,0} and~(\ref{eq:ex-of-opt-designs,1}.a) into~\eqref{eq:ex-of-opt-designs,3}, we finally obtain
\begin{align}
\inf_{\xi \in \Xi(X)} \Psi(M(\xi)) \le \Psi(M(\xi^*)) \le \liminf_{k\to\infty} \Psi(M(\xi^{n_k})) = \inf_{\xi \in \Xi(X)}  \Psi(M(\xi)) 
\end{align}
and therefore $\xi^*$ is an optimal design for $\Psi \circ M$, as desired. 
\smallskip

As a second step, we show that there also exists an optimal design $\xi^*$ for $\Psi \circ M$ with at most $d(d+1)/2 + 1$ support points. 
Indeed, let $\xi^*$ be an arbitrary an optimal design for $\Psi \circ M$ (which exists by the first step). Lemma~\ref{lm:designs-with-finite-supp} then yields a design $\xi^*_0$ with at most $d(d+1)/2 + 1$ support points such that 
\begin{align} \label{eq:ex-of-opt-designs,4}
M(\xi^*_0) = M(\xi^*)
\end{align}
In view of~\eqref{eq:ex-of-opt-designs,4} and the optimality of $\xi^*$, the finitely supported design $\xi^*_0$ is optimal for $\Psi\circ M$ as well, as desired.
\smallskip

As a third and last step, we show that under the additional antitonicity assumption there even exists an optimal design $\xi^*$ with at most $d(d+1)/2$ support points. 
Indeed, let the additional antitonicity assumption be satisfied and let $\xi^*$ be an arbitrary optimal design for $\Psi \circ M$ (which exists by the first step). Since $M(\xi^*) \in M(\Xi(X))$ and $M(\Xi(X))$ is compact, there exists an $\alpha^* \in [0,\infty)$ %and, hence, also a $\xi^{**} \in \Xi(X)$ such that
\begin{align} \label{eq:ex-of-opt-designs,5}
M(\xi^*) + \alpha^* I \in \partial M(\Xi(X)) \subset M(\Xi(X)),
\end{align}
where $I$ denotes the identity matrix, and therefore there also exists a $\xi^{**} \in \Xi(X)$ such that 
\begin{align} \label{eq:ex-of-opt-designs,6}
M(\xi^{**}) = M(\xi^*) + \alpha^* I \in \partial M(\Xi(X)).
\end{align}
In view of the antitonicity of $\Psi$ and the optimality of $\xi^*$, we conclude from~\eqref{eq:ex-of-opt-designs,6} that $\xi^{**}$ is an optimal design for $\Psi \circ M$ with $M(\xi^{**}) \in \partial M(\Xi(X))$. Applying now Lemma~\ref{lm:designs-with-finite-supp} to $M(\xi^{**})$, we see that there also exists an optimal design  $\xi^*_0$ with at most $d(d+1)/2$ support points, as desired.
\end{proof}

We point out that the lower semicontinuity assumption of the above theorem cannot be dropped. Indeed, as we will show in the next example, the %(naively defined) 
optimal design problem for the weighted A-criterion~\eqref{eq:def-weighted-A-criterion} (with non-positive-definite weight matrix $W$) has no solution, in general. In particular, the existence statement of Theorem~2.2 from~\cite{FeLe} is false as it stands.

\begin{ex}
Consider a discrete metric space $X := \{x_1, x_2\}$ consisting of two elements and let $m: X \to \R^{2\times 2}_{\mathrm{psd}}$ be defined by
\begin{align} \label{eq:counterex,1}
m(x_1) := \diag(1,0) \qquad \text{and} \qquad m(x_2) := \diag(0,1). 
\end{align}
Also, let $\Psi: \R^{2\times 2}_{\mathrm{psd}} \to \R \cup \{\infty\}$ be defined by %$\Psi = \Psi_u$ with $u := (1, 0)^{\top}$
\begin{align} \label{eq:counterex,2}
\Psi(M) := e_1^{\top} M^{-1} e_1 \quad (M \in \R^{2\times 2}_{\mathrm{pd}}) 
\quad \text{and} \quad
\Psi(M) := \infty \quad (M \in \R^{2\times 2}_{\mathrm{psd}} \setminus \R^{2\times 2}_{\mathrm{pd}}),
\end{align}
where $e_1 := (1,0)^{\top} \in \R^2$. In other words, %$\Psi$ is equal to $\Psi_{e_1}$ (Lemma~\ref{lm:standard-properties-for-Psi_u}) or, put differently, 
$\Psi$ is the weighted A-criterion~\eqref{eq:def-weighted-A-criterion} with weight matrix $W$ being the orthogonal projection on $e_1$. 
With these definitions, it is easy to see that all assumptions of the existence theorem (Theorem~2.2) of~\cite{FeLe} are satisfied, but there exists no optimal design for $\Psi \circ M$ (because $\Psi$ fails to satisfy the lower semicontinuity assumption of our existence theorem above). 
Indeed, the assumptions from~\cite{FeLe} are satisfied because $X$ is a compact metric space, $m \in C(X,\R^{2\times 2}_{\mathrm{psd}})$, and $\Psi$ is antitonic, convex with $\dom \Psi = \R^{2\times 2}_{\mathrm{pd}}$ and because 
%by Lemma~\ref{lm:standard-properties-for-Psi_u}, 
$\Psi|_{\R^{2\times 2}_{\mathrm{pd}}}$ is continuous and directionally differentiable with
\begin{align} \label{eq:counterex,3}
\lim_{t\searrow 0} \frac{\Psi\big(M+t(M(\eta)-M)\big)-\Psi(M)}{t} = \partial_{M(\eta)-M}\Psi(M) = \int_X \psi(M,x) \d\eta(x) 
\end{align}  
for every $M \in \R^{2\times 2}_{\mathrm{pd}}$ and $\eta \in \Xi(X)$, where $\psi(M,x) := e_1^{\top} M^{-1}e_1 - e_1^{\top} M^{-1}m(x)M^{-1}e_1$.
%\begin{align*}
%\psi(M,x) := e_1^{\top} M^{-1}e_1 - e_1^{\top} M^{-1}m(x)M^{-1}e_1.
%\end{align*}
Additionally,
\begin{align} \label{eq:counterex,4}
M(\Xi(X)) = \{M(\xi): \xi \in \Xi(X)\} = \{\diag(a, 1-a): a \in [0,1]\}
\end{align}
because $\Xi(X) = \{a\delta_{x_1} + (1-a)\delta_{x_2}: a \in [0,1]\}$ (Lemma~\ref{lm:support}). So, by~\eqref{eq:counterex,2} and~\eqref{eq:counterex,4},
\begin{align}
\{\Psi(M(\xi)): \xi \in \Xi(X)\} = \{1/a: a \in (0,1)\} \cup \{\infty\}
\end{align}
and therefore
\begin{align}
\inf_{\xi\in \Xi(X)} \Psi(M(\xi)) = 1 < \Psi(M(\xi^*)) \qquad (\xi^* \in \Xi(X)).
\end{align}
Consequently, there exists no optimal design for $\Psi \circ M$, as claimed. 
\end{ex}

\subsection{Characterization of optimal designs}
\label{sec:characterization}

In this section, we record well-known characterizations of optimal and approximately optimal designs in terms of the so-called sensitivity function $\psi$ of $\Psi|_{\dom\Psi}$ \cite{KiWo60, FeLe, PrPa}. In other words, we derive necessary and sufficient conditions for a design $\xi^*$ to be $\eps$-optimal for $\Psi \circ M$, in terms of the sensitivity function of $\Psi|_{\dom\Psi}$. As before, we assume that $X$ is a compact metric space, $m \in C(X,\Rpsd^{d \times d})$, %$M:\Xi(X) \to \Rpsd^{d \times d}$ 
\begin{align} \label{eq:M(xi)-def}
M:\Xi(X) \to \Rpsd^{d \times d}, \qquad
M(\xi) := \int_X m(x)\d\xi(x)
\end{align}
is the corresponding information matrix map (Condition~\ref{cond:X-and-M}) and that $\Psi: \Rpsd^{d \times d} \to \R \cup \{\infty\}$ is the considered design criterion. A function $\psi: \dom\Psi \times X \to \R$ is then called a \emph{sensitivity function of $\Psi|_{\dom\Psi}$ w.r.t.~$M(\Xi(X))$} iff for every $M \in \dom\Psi$ and every $\eta \in \Xi(X)$, the function $\psi(M,\cdot)$ is $\eta$-integrable and 
\begin{align} \label{eq:sensi-fct-def}
\frac{\Psi\big(M+t(M(\eta)-M)\big)-\Psi(M)}{t} \longrightarrow \int_X \psi(M,x) \d\eta(x) \qquad (t\searrow 0). 
\end{align}
%(In the terminology of~\cite{FeLe}, a sensitivity function is the negative of a sensitivity function in the above sense. We find it more natural to skip this sign flip and therefore slightly deviate from the terminology of~\cite{FeLe}.) 
Additionally, we call the restriction $\Psi|_{\dom\Psi}$ of the design criterion \emph{directionally differentiable w.r.t.~$M(\Xi(X))$} iff there exists a sensitivity function $\psi$ for $\Psi|_{\dom\Psi}$ w.r.t.~$M(\Xi(X))$. 
\smallskip

It is clear by the definition~\eqref{eq:sensi-fct-def} that if for a given design criterion $\Psi$ there exists any sensitivity function $\psi$ w.r.t.~$M(\Xi(X))$, then it is necessarily unique. (Simply consider~\eqref{eq:sensi-fct-def} with all point measures $\eta := \delta_x$ for $x \in X$.) It is also clear by the definition~\eqref{eq:sensi-fct-def} that if $\Psi|_{\dom\Psi}$ is directionally differentiable w.r.t.~$M(\Xi(X))$, then $\dom\Psi$ is \emph{directionally open w.r.t.~$M(\Xi(X))$} in the following sense: for every $M \in \dom\Psi$ and every $\eta \in \Xi(X)$, there exists a $t^* \in (0,1]$ such that
\begin{align} \label{eq:directional-openness-definition}
M + t(M(\eta)-M) \in \dom\Psi \qquad (t \in [0,t^*]). 
\end{align} 
And finally, it is clear by the definition~\eqref{eq:sensi-fct-def} that if $\Psi|_{\dom\Psi}$ is directionally differentiable w.r.t.~$M(\Xi(X))$, then for every $\xi \in \dom \Psi \circ M$ and every $\eta \in \Xi(X)$ the function $(\Psi \circ M)|_{\dom \Psi \circ M}$ is directionally differentiable at $\xi$ in the direction $\eta - \xi$ in the usual sense~\eqref{eq:directional-derivative-def} %~\cite{Jahn} (Section 3.1) 
and the directional derivative $\partial_{\eta-\xi}(\Psi \circ M)(\xi)$ takes the specific integral form
\begin{align}
\partial_{\eta-\xi}(\Psi \circ M)(\xi) = \int_X \psi(M(\xi),x)\d\eta(x).
\end{align}
As the next lemma shows, directional differentiability w.r.t.~$M(\Xi(X))$ follows from differentiability and thus is a %fairly 
mild condition.

\begin{lm} \label{lm:sensi-fct-in-terms-of-derivative}
Suppose that Condition~\ref{cond:X-and-M} is satisfied and that $\Psi: \Rpsd^{d \times d} \to \R \cup \{\infty\}$ is a design criterion such that $\dom \Psi$ is directionally open w.r.t.~$M(\Xi(X))$ and $\Psi|_{\dom \Psi}$ is differentiable. Then $\Psi|_{\dom\Psi}$ is directionally differentiable w.r.t.~$M(\Xi(X))$ with sensitivity function $\psi$ given by
\begin{align}
\psi(M,x) = D\Psi(M)(m(x)-M) \qquad (M \in \dom\Psi \text{ and } x \in X). 
\end{align} 
\end{lm} 

\begin{proof}
Suppose $M \in \dom \Psi$ and $\eta \in \Xi(X)$. Since $\dom\Psi$ is directionally open w.r.t.~$M(\Xi(X))$, there exists a $t^* \in (0,1]$ such that $M + t(M(\eta)-M) \in \dom\Psi$ for every $t \in [0,t^*]$. And therefore, by the differentiability of $\Psi|_{\dom\Psi}$ and by~\eqref{eq:M(xi)-def}, 
\begin{align}
\frac{\Psi(M +t(M(\eta)-M))-\Psi(M)}{t} 
&\longrightarrow D\Psi(M)\big(M(\eta)-M\big) \notag\\
&= \int_X D\Psi(M)\big(m(x)-M\big)\d\eta(x) \qquad (t\searrow 0).
\end{align}
Consequently, $\dom\Psi \times X \ni (M,x) \mapsto \psi(M,x) := D\Psi(M)(m(x)-M)$ is a sensitivity function for $\Psi|_{\dom\Psi}$, as desired. 
\end{proof}

\begin{ex} \label{ex:sensitivity-fcts-for-Psi_p}
Suppose that Condition~\ref{cond:X-and-M} is satisfied. Also, let $M^0 \in \Rpsd^{d \times d}$ and $\alpha \in [0,1)$ and $p \in [0,\infty)$ and let $\Psi_p^{(\alpha)}$ be the two-stage design criterion defined by
\begin{align}
\Psi_p^{(\alpha)}(M) := \Psi_p(\alpha M^0 + (1-\alpha) M) \qquad (M \in \Rpsd^{d \times d}). 
\end{align} 
Since $\dom\Psi_p = \Rpd^{d \times d}$ is an open subset of $\R^{d \times d}$, it is straightforward to verify that 
$\dom\Psi_p^{(\alpha)}$ is directionally open w.r.t.~$M(\Xi(X))$. Since, moreover, %$\Psi_p^{(\alpha)}|_{\dom\Psi_p^{(\alpha)}}$ 
$\Psi_p^{(\alpha)}$ restricted to its domain is differentiable (Proposition~\ref{prop:standard-properties-for-Psi_p} and Proposition~\ref{prop:two-stage-design-criteria}), it follows by Lemma~\ref{lm:sensi-fct-in-terms-of-derivative} that the two-stage design criterion %$\Psi_p^{(\alpha)}|_{\dom\Psi_p^{(\alpha)}}$ 
is directionally differentiable w.r.t.~$M(\Xi(X))$ with sensitivity function $\psi_p^{(\alpha)}: \dom\Psi_p^{(\alpha)} \times X \to \R$ given by
\begin{align}
\psi_p^{(\alpha)}(M,x) = (1-\alpha) D\Psi_p(M^{(\alpha)})(m(x)-M)
\end{align}
where $M^{(\alpha)} := \alpha M^0 + (1-\alpha) M$. In view of~\eqref{eq:directional-derivative-p=0} and~\eqref{eq:directional-derivative-p>0}, this implies that
\begin{gather} 
\psi_0^{(\alpha)}(M,x) = (1-\alpha) \tr\big( (M^{(\alpha)})^{-1} (M-m(x)) \big) \qquad (p=0) \label{eq:sensitivity-fcts-for-Psi_0}\\
\psi_p^{(\alpha)}(M,x) = (1-\alpha) \big(\tr((M^{(\alpha)})^{-p})\big)^{1/p-1} \tr((M^{(\alpha)})^{-p-1} (M-m(x)) \big)
\qquad (p \in (0,\infty)). \label{eq:sensitivity-fcts-for-Psi_p}
\end{gather} 
%It follows from these formulas that the sensitivity function $\psi_p: \dom\Psi_p \times X \to \R$ is continuous for every $p \in [0,\infty)$. In fact, it even follows that $\psi_p: \dom\Psi_p \times X \to \R$ for every $p \in [0,\infty)$ is locally Lipschitz continuous w.r.t.~$M \in \dom\Psi_p$ uniformly w.r.t.~$x \in X$. 
It follows from these formulas that the sensitivity function $\psi_p^{(\alpha)}$ %: \dom\Psi_p^{(\alpha)} \times X \to \R$ %for every $p \in [0,\infty)$ 
is locally Lipschitz continuous w.r.t.~$M \in \dom\Psi_p^{(\alpha)}$ uniformly w.r.t.~$x \in X$. Indeed, this  follows from~\eqref{eq:sensitivity-fcts-for-Psi_0} and~\eqref{eq:sensitivity-fcts-for-Psi_p} by the local Lipschitz continuity of the maps $\Rpd^{d \times d} \ni M \mapsto M^{-q_1}$ (Lemma~\ref{lm:directional-derivatives}) and $(0,\infty) \ni t \mapsto t^{q_2}$ for $q_1 \in [0,\infty)$ and $q_2 \in \R$ and by the %well-known 
simple estimate %$\tr(AB) \le \norm{A} \tr(B)$ for $A, B \in \Rpsd^{d\times d}$.
\begin{align}
\tr(AB) = \tr(B^{1/2} A B^{1/2}) \le \norm{A} \tr(B) 
\qquad (A, B \in \Rpsd^{d\times d}).
\end{align}
\end{ex}

%\subsection{Simple design criteria}

\begin{thm} \label{thm:equivalence-thm}
Suppose that Condition~\ref{cond:X-and-M} is satisfied and that $\Psi: \Rpsd^{d \times d} \to \R \cup \{\infty\}$ is convex. Suppose further that $\Psi|_{\dom\Psi}$ is directionally differentiable w.r.t.~$M(\Xi(X))$ and that 
\begin{align} \label{eq:Xi_fin(X)-def}
\Xifin(X) := \dom \Psi \circ M = \{\xi \in \Xi(X): \Psi(M(\xi)) < \infty\}
\end{align}
is non-empty. Then for every $\eps \in [0,\infty)$ the following two assertions are equivalent:
\begin{itemize}
\item[(i)] $\xi^*$ is an $\eps$-optimal design for $\Psi \circ M$
\item[(ii)] $\xi^* \in \Xifin(X)$ and $\psi(M(\xi^*),x) \ge -\eps$ for every $x \in X$.
\end{itemize}
Additionally, for every optimal design $\xi^*$ for $\Psi \circ M$, one has $\psi(M(\xi^*),x) = 0$ for $\xi^*$-a.e.~$x \in X$. In particular, for every optimal design $\xi^*$ with finite support, one has $\psi(M(\xi^*),x) = 0$ for every $x \in \supp \xi^*$. 
\end{thm}

\begin{proof}
Suppose first that assertion~(i) is satisfied. Since $\Xifin(X) \ne \emptyset$, we immediately obtain $\Psi(M(\xi^*)) = \min_{\xi\in\Xi(X)}\Psi(M(\xi)) < \infty$ and thus %$\xi^* \in \Xifin(X)$. 
\begin{align} \label{eq:equivalence-thm,1}
\xi^* \in \dom \Psi\circ M = \Xifin(X).
\end{align}
Since $\xi^*$ is an $\eps$-optimal design for $\Psi \circ M$ and $(1-t)\xi^* + t \delta_x \in \Xi(X)$ for every $t \in [0,1]$ and $x \in X$, it follows that
\begin{align} \label{eq:equivalence-thm,2}
\Psi\big( M((1-t)\xi^* + t \delta_x)) \big) - \Psi(M(\xi^*)) \ge -\eps \qquad (t \in [0,1] \text{ and } x \in X)
\end{align}
Since, moreover, $\psi$ is a sensitivity function of $\Psi|_{\dom\Psi}$, it further follows by~\eqref{eq:equivalence-thm,2} that
\begin{align} \label{eq:equivalence-thm,3}
\psi(M(\xi^*),x) &= \int_X \psi(M(\xi^*),x') \d\delta_x(x') \notag\\
&= \lim_{t\searrow 0} \frac{\Psi\big( M((1-t)\xi^* + t \delta_x)) \big) - \Psi(M(\xi^*))}{t}
\ge -\eps
\qquad (x \in X).
\end{align}
In view of~\eqref{eq:equivalence-thm,1} and~\eqref{eq:equivalence-thm,3}, assertion~(ii) is now clear.
Suppose now, conversely, that assertion~(ii) is satisfied. Since $\Psi$ is convex, it follows that
\begin{align} \label{eq:equivalence-thm,4}
\frac{\Psi\big( (1-t)M(\xi^*) + t M(\eta)) \big) - \Psi(M(\xi^*))}{t} \le \Psi(M(\eta))-\Psi(M(\xi^*))
%\qquad (\eta \in \Xi(X)).
\end{align}
for every $\eta \in \Xi(X)$ and every $t \in (0,1]$. 
Since, moreover, $\psi$ is a sensitivity function for $\Psi|_{\dom\Psi}$ and since $M(\xi^*) \in \dom\Psi$ and $\psi(M(\xi^*),x) \ge -\eps$ for all $x \in X$ by the assumed assertion~(ii), it further follows by~\eqref{eq:equivalence-thm,4} that
\begin{align} \label{eq:equivalence-thm,5}
-\eps \le \int_X \psi(M(\xi^*),x) \d\eta(x) 
&= \lim_{t\searrow 0} \frac{\Psi\big( (1-t)M(\xi^*) + t M(\eta)) \big) - \Psi(M(\xi^*))}{t} \notag \\
&\le \Psi(M(\eta))-\Psi(M(\xi^*)) 
\qquad (\eta \in \Xi(X)).
\end{align}
Consequently, $\xi^*$ is an $\eps$-optimal design for $\Psi \circ M$, that is, assertion~(i) is satisfied. 
We have thus established the equivalence of assertions~(i) and (ii) and it remains to prove the additional statement of the theorem. So, let $\xi^*$ be any optimal design for $\Psi\circ M$. It then follows by the equivalence just established that
\begin{align} \label{eq:equivalence-thm,6}
M(\xi^*) \in \dom\Psi  
\qquad \text{and} \qquad 
\psi(M(\xi^*),x) \ge 0 \qquad (x \in X).
\end{align}
Consider now the sets $Z_n := \{ x \in X: \psi(M(\xi^*),x) \ge 1/n \}$ for $n \in \N$. Since $\psi$ is a sensitivity function for $\Psi|_{\dom\Psi}$, we see from~\eqref{eq:equivalence-thm,6} that $Z_n \in \mathcal{B}_X$ and that
\begin{align}
0 \le (1/n) \cdot \xi^*(Z_n) \le \int_{Z_n} \psi(M(\xi^*),x) \d\xi^*(x) 
\le \int_X \psi(M(\xi^*),x) \d\xi^*(x) 
= 0 \qquad (n\in\N),
\end{align}
where for the last inequality we used~(\ref{eq:equivalence-thm,6}.b) and for the last equality we used~\eqref{eq:sensi-fct-def}. Consequently, $Z_n$ is a $\xi^*$-nullset for every $n \in \N$ %$\xi^*(Z_n) = 0$ for all $n \in \N$. 
and therefore
\begin{align}
\{x \in X: \psi(M(\xi^*),x) \ne 0\} = \{x \in X: \psi(M(\xi^*),x) > 0\} = \bigcup_{n\in\N} Z_n
\end{align}
is a $\xi^*$-nullset as well, which proves the first part of the additional statement. %or, in other words, $\psi(M(\xi^*),x) = 0$ for $\xi^*$-a.e.~$x \in X$, as desired.
In order to prove also the second part of the additional statement, let $\xi^*$ be a an optimal design for $\Psi\circ M$ with finite support. We then have, on the one hand, that
\begin{align}
Z := \{x \in \supp\xi^*:  \psi(M(\xi^*),x) \ne 0\} \subset \supp\xi^*
\end{align}
is a $\xi^*$-nullset by the statement just proven and, on the other hand, that $\xi^*(\{x\}) > 0$ for all $x \in \supp\xi^*$ by Lemma~\ref{lm:support}. Consequently, $Z$ must be the empty set, %because otherwise we would obtain $0 = \xi^*(Z) = \sum_{x\in Z} \xi^*(\{x\}) > 0$, a contradiction)
which proves the final statement. 
\end{proof}

\section{Adaptive discretization algorithms to compute optimal designs}
\label{sec:algorithms}

In this section, we introduce our adaptive discretization algorithms for the computation of optimal and approximately optimal designs, which refine and improve the algorithm from~\cite{YaBiTa13}. After introducing the algorithms, we establish our various termination, convergence and convergence rate results on them. In very rough terms, our algorithms work as follows. In each iteration of the algorithms, a discretized version
\begin{align} \label{eq:doe-discretized}
\min_{\xi\in\Xi(X^k)} \Psi(M(\xi))
\end{align}
of~\eqref{eq:doe} is solved up to some optimality tolerance $\ol{\delta}_k$. In this discretized optimal design problem~\eqref{eq:doe-discretized}, the discretization $X^k$ is a finite subset of $X$ that is adaptively updated in the following sense: the new discretization $X^{k+1}$ depends on how much the solution $\xi^k$ computed for~\eqref{eq:doe-discretized} violates the necessary and sufficient $\eps$-optimality condition
\begin{align} \label{eq:Viol(xi^k)}
\min_{x\in X} \psi(M(\xi^k),x) \ge -\eps
\end{align}
for $\xi^k$ (Theorem~\ref{thm:equivalence-thm}). In contrast to the results from the previous sections, we now also have to impose a continuity assumption on the sensitivity functions $\psi$. As was pointed out above (Example~\ref{ex:sensitivity-fcts-for-Psi_p}), this additional continuity assumption is satisfied for all standard design criteria $\Psi_p$ with $p \in [0,\infty)$. 

\begin{cond} \label{cond:shared-assumption-1-algos-with-and-without-exchange}
Condition~\ref{cond:X-and-M} is satisfied and $\Psi: \Rpsd^{d \times d} \to \R \cup \{\infty\}$ is convex and lower semicontinuous with 
\begin{align} \label{eq:Xifin(X)-non-empty}
\Xifin(X) := \dom \Psi\circ M \ne \emptyset.
\end{align}
\end{cond}

\begin{cond} \label{cond:shared-assumption-2-algos-with-and-without-exchange}
$\Psi|_{\dom\Psi}$ is directionally differentiable w.r.t.~$M(\Xi(X))$ with a continuous sensitivity function $\psi$. 
\end{cond}

\begin{cond} \label{cond:shared-assumption-3-algos-with-and-without-exchange}
$\Psi|_{\dom\Psi}$ is continuously differentiable and $\dom\Psi$ is directionally open w.r.t. $M(\Xi(X))$. 
\end{cond}

\begin{lm} \label{lm:Psi-continuous}
\begin{itemize}
\item[(i)] If Condition~\ref{cond:shared-assumption-1-algos-with-and-without-exchange} is satisfied, then there exists a finite subset $X^0$ of $X$ with
\begin{align} \label{eq:finite-criterion-on-Xi(X^0)}
\Xifin(X^0) := \{\xi \in \Xi(X^0): \Psi(M(\xi)) < \infty\} \ne \emptyset
\end{align}
\item[(ii)] If Conditions~\ref{cond:shared-assumption-1-algos-with-and-without-exchange} and~\ref{cond:shared-assumption-2-algos-with-and-without-exchange} are satisfied, then the map $(\Psi \circ M)|_{\Xifin(X)}$ is continuous. 
\end{itemize}
\end{lm}

%\begin{lm} \label{lm:Psi-continuous}
%Suppose that Conditions~\ref{cond:shared-assumption-1-algos-with-and-without-exchange} and~\ref{cond:shared-assumption-2-algos-with-and-without-exchange} are satisfied. 
%%Suppose that Condition~\ref{cond:X-and-M} is satisfied and that $\Psi: \Rpsd^{d \times d} \to \R \cup \{\infty\}$ is convex. Suppose further that $\Psi|_{\dom\Psi}$ is directionally differentiable w.r.t.~$M(\Xi(X))$ with a continuous sensitivity function $\psi$ and that $\Xifin(X) = \dom \Psi\circ M$ is non-empty. 
%Then $(\Psi \circ M)|_{\Xifin(X)}$ is continuous.
%\end{lm}

\begin{proof}
In order see assertion (i), choose an optimal design $\xi^0$ for $\Psi \circ M$ with finite support (Theorem~\ref{thm:ex-of-optimal-designs}). It is then clear by~\eqref{eq:Xifin(X)-non-empty} that $X^0 := \supp \xi^0$ is a finite set satisfying~\eqref{eq:finite-criterion-on-Xi(X^0)}.
In order to see assertion (ii), let $\xi^n, \xi \in \Xifin(X)$ with $\xi^n \longrightarrow \xi$ as $n \to \infty$. 
It then follows 
%by the assumed convexity of $\Psi_i$ that
%\begin{gather*} 
%\alpha \big( \Psi_i(\xi)-\Psi_i(\xi^n)\big) \ge \Psi_i(\xi^n + \alpha(\xi-\xi^n)) - \Psi_i(\xi^n)\\
%\alpha\big( \Psi_i(\xi^n)-\Psi_i(\xi)\big) \ge \Psi_i(\xi + \alpha(\xi^n-\xi)) - \Psi_i(\xi)
%\end{gather*}
%for all $\alpha \in [0,1]$ and therefore
by the assumed convexity of $\Psi$, using the same arguments as for~\eqref{eq:equivalence-thm,4}, that 
\begin{align*}  %\label{eq:Psi-continuous-1}
\Psi(M(\xi^n)) - \Psi(M(\xi)) \le -\int_X \psi(M(\xi^n),x) \d\xi(x) \\
\Psi(M(\xi)) - \Psi(M(\xi^n)) \le -\int_X \psi(M(\xi),x) \d\xi^n(x) 
%\qquad (n \in \N)
\end{align*}
for all $n \in \N$. Consequently,
\begin{align} \label{eq:Psi-continuous-2}
|\Psi(M(\xi^n)) - \Psi(M(\xi))| \le  \max\bigg\{ -\int_X \psi(M(\xi^n),x) \d\xi(x), -\int_X \psi(M(\xi),x) \d\xi^n(x) \bigg\} 
\end{align}
for all $n \in \N$. Since by assumption $\psi(M(\cdot),x)$ is continuous for every $x \in X$ and $\psi(M(\xi),\cdot)$ is a bounded continuous function, we conclude with the dominated convergence theorem that the right-hand side of~\eqref{eq:Psi-continuous-2} converges to
\begin{align} \label{eq:Psi-continuous-3}
\int_X \psi(M(\xi),x) \d\xi(x) = 0
\end{align}
as $n \to \infty$. So, the same is true for the left-hand side of~\eqref{eq:Psi-continuous-2}, whence $\Psi(M(\xi^n)) \longrightarrow \Psi(M(\xi))$ as $n \to \infty$, as desired.
\end{proof}

\subsection{Adaptive discretization algorithms without exchange} %in updating the discretizations
%\subsection{An algorithm without exchange} %in updating the discretizations

We begin with two algorithms without exchange, where the discretization sets $X^k$ get larger and larger from iteration to iteration:
\begin{align} \label{eq:discretization-sets-get-larger-and-larger}
X^k \subset X^{k+1}
\end{align}
for all iteration indices $k$. 

\begin{algo} \label{algo:no-exchange}
Input: a finite subset $X^0$ of $X$ and optimality tolerances $\eps, \ol{\delta}_k, \ul{\delta}_k \in [0,\infty)$. %With these inputs, proceed in the following steps.
Initialize $k=0$ and then proceed in the following steps. 
\begin{itemize}
\item[1.] Compute a $\ol{\delta}_k$-approximate solution $\xi^k \in \Xi(X^k)$ of the discretized optimal design problem
\begin{align} \label{eq:DOE(X^k)-simple}
\min_{\xi\in\Xi(X^k)} \Psi(M(\xi))
\end{align}
\item[2.] Compute a $\ul{\delta}_k$-approximate solution $x^k \in X$ of the strongest optimality violator problem
\begin{align} \label{eq:Viol(xi^k)-simple}
\min_{x\in X} \psi(M(\xi^k),x).
\end{align}
If $\psi(M(\xi^k),x^k) < -\eps + \ul{\delta}_k$, then set $X^{k+1} := X^k \cup \{x^k\}$ 
%\begin{align}
%X^{k+1} := X^k \cup \{x^k\}
%\end{align}
and return to Step~1 with $k$ replaced by $k+1$. If $\psi(M(\xi^k),x^k) \ge -\eps + \ul{\delta}_k$, then terminate. 
\end{itemize}
\end{algo}

\begin{algo} \label{algo:no-exchange-arbitrary-violators}
Input: a finite subset $X^0$ of $X$ and optimality tolerances $\eps, \ol{\delta}_k \in [0,\infty)$. %With these inputs, proceed in the following steps.
Initialize $k=0$ and then proceed in the following steps. 
\begin{itemize}
\item[1.] Compute a $\ol{\delta}_k$-approximate solution $\xi^k \in \Xi(X^k)$ of the discretized optimal design problem
\begin{align} \label{eq:DOE(X^k)-simple-arbitrary-violators}
\min_{\xi\in\Xi(X^k)} \Psi(M(\xi))
\end{align}
\item[2.] Search for a violator of the $\eps$-optimality condition for $\xi^k$, that is, for a point $x^k$ in the set
\begin{align} \label{eq:optimality-violator-set,no-exchange}
\{x \in X: \psi(M(\xi^k),x) < -\eps \}.
\end{align}
If such a point $x^k$ is found, then set $X^{k+1} := X^k \cup \{x^k\}$ 
%\begin{align}
%X^{k+1} := X^k \cup \{x^k\}
%\end{align}
and return to Step~1 with $k$ replaced by $k+1$. If no such point can possibly be found (that is, the set~\eqref{eq:optimality-violator-set,no-exchange} is empty), then terminate. 
\end{itemize}
\end{algo}

We first convince ourselves that for a suitable initial discretization $X^0$ %the above algorithm is well-defined
all the optimization problems encountered in the course of Algorithm~\ref{algo:no-exchange} or~\ref{algo:no-exchange-arbitrary-violators} are solvable and that the sequence $(\Psi(M(\xi^k))$ of optimal values is bounded above. 

\begin{lm} \label{lm:simple-algo-well-defined}
%Suppose that Condition~\ref{cond:X-and-M} is satisfied and that $\Psi: \Rpsd^{d \times d} \to \R \cup \{\infty\}$ is convex and lower semicontinuous. Suppose further that $\Psi|_{\dom\Psi}$ is directionally differentiable w.r.t.~$M(\Xi(X))$ with a continuous sensitivity function $\psi$ and that $\Xifin(X) = \dom \Psi\circ M$ is non-empty. 
Suppose that Conditions~\ref{cond:shared-assumption-1-algos-with-and-without-exchange} and~\ref{cond:shared-assumption-2-algos-with-and-without-exchange} are satisfied and that $X^0$ is a finite subset of $X$ with~\eqref{eq:finite-criterion-on-Xi(X^0)} and $\eps, \ol{\delta}_k, \ul{\delta}_k \in [0,\infty)$ are arbitrary optimality tolerances.   
%Then there exists a finite subset $X^0$ of $X$ with
%\begin{align} \label{eq:finite-criterion-on-Xi(X^0),simple-algo}
%\Xifin(X^0) := \{\xi \in \Xi(X^0): \Psi(M(\xi)) < \infty\} \ne \emptyset. 
%\end{align}
%Additionally, 
Then all optimization problems in Algorithm~\ref{algo:no-exchange} and~\ref{algo:no-exchange-arbitrary-violators} %with the above inputs %with such an $X^0$ and arbitrary optimality tolerances $\eps, \ol{\delta}_k, \ul{\delta}_k \in [0,\infty)$ %as its inputs 
are solvable and, for the (possibly finite) sequence $(\xi^k)_{k\in K}$ of iterates generated by Algorithm~\ref{algo:no-exchange} or~\ref{algo:no-exchange-arbitrary-violators}, %with these inputs, 
one has
\begin{align} \label{eq:simple-algo-well-defined}
\Psi(M(\xi^{l})) \le \Psi(M(\xi^k)) + \ol{\delta}_l
\qquad (k, l \in K \text{ with } k < l).
\end{align} 
In particular, $\xi^k \in \Xifin(X)$ and $\Psi(M(\xi^k)) \le \Psi(M(\xi^0)) + \ol{\delta}_k$ for all $k \in K$. %$M(\xi^k) \in \dom\Psi$ for all $k \in K$.
\end{lm}

\begin{proof}
It is straightforward to see by induction over $K$ that for every iteration index $k \in K$ the following statements are true: (i) the discretized design problem~\eqref{eq:DOE(X^k)-simple} or~\eqref{eq:DOE(X^k)-simple-arbitrary-violators}, respectively, has a solution (so that $\xi^k$ is well-defined also when $\ol{\delta}_k = 0$), (ii) if $k\ne 0$, then $\Psi(M(\xi^{k})) \le \Psi(M(\xi^{k-1})) + \ol{\delta}_{k}$, (iii) $\xi^k \in \Xifin(X)$, and (iv) the strongest violator problem~\eqref{eq:Viol(xi^k)-simple} has a solution. 
Consequently, it only remains to prove~\eqref{eq:simple-algo-well-defined}, but this is straightforward as well. Indeed, by the algorithms' progressively refining updating rule, we have for every $k < l$ that $X^k \subset X^l$ and therefore $\xi^k \in \Xi(X^l)$, which implies~\eqref{eq:simple-algo-well-defined} by the $\ol{\delta}_l$-approximate solution property of $\xi^l$. 
\end{proof}

\begin{lm} \label{lm:termination-lemma-without-exchange}
%Suppose that Condition~\ref{cond:X-and-M} is satisfied and that $\Psi: \Rpsd^{d \times d} \to \R \cup \{\infty\}$ is  convex and lower semicontinuous. Suppose further that $\Psi|_{\dom\Psi}$ is directionally differentiable w.r.t.~$M(\Xi(X))$ with a continuous sensitivity function $\psi$ and that $\Xifin(X) = \dom \Psi\circ M$ is non-empty. 
Suppose that Conditions~\ref{cond:shared-assumption-1-algos-with-and-without-exchange} and~\ref{cond:shared-assumption-2-algos-with-and-without-exchange} are satisfied. Suppose further that $X^0$ is a finite subset of $X$ satisfying~\eqref{eq:finite-criterion-on-Xi(X^0)} and that $(\xi^k)_{k\in K}$ is generated by Algorithm~\ref{algo:no-exchange} or~\ref{algo:no-exchange-arbitrary-violators} with such an $X^0$ and arbitrary optimality tolerances $\eps, \ol{\delta}_k, \ul{\delta}_k \in [0,\infty)$ such that $(\ol{\delta}_k)$ is bounded.
\begin{itemize}
\item[(i)] If $(\xi^k)$ is terminating, then it terminates at an $\eps$-optimal design for $\Psi \circ M$.
\item[(ii)] If $(\xi^k)$ is non-terminating, then every accumuluation point $\xi^*$ of $(\xi^k)$ belongs to $\Xifin(X)$.
\end{itemize}
\end{lm}

\begin{proof}
(i) Suppose that $(\xi^k)_{k\in K}$ is terminating, that is, $K = \{0,\dots, k^*\}$ for some terminal index $k^* \in \N_0$. It then follows by Lemma~\ref{lm:simple-algo-well-defined} and by the termination rules of  Algorithm~\ref{algo:no-exchange} and~\ref{algo:no-exchange-arbitrary-violators}, respectively, that 
\begin{align} \label{eq:termination-lm-without-exchange,1}
\xi^{k^*} \in \Xifin(X) \qquad \text{and} \qquad \min_{x\in X} \psi(M(\xi^{k^*}),x) \ge -\eps.
\end{align}
And therefore, $\xi^{k^*}$ is an $\eps$-optimal design for $\Psi\circ M$ by virtue of Theorem~\ref{thm:equivalence-thm}, as desired.
\smallskip

(ii) Suppose now that $(\xi^k)_{k\in K}$ is non-terminating, that is, $K = \N_0$. Also, let $\xi^*$ be any accumulation point of $(\xi^k)$ and let $(\xi^{k_l})$ be any subsequence with $\xi^{k_l} \longrightarrow \xi^*$ as $l\to\infty$. Since $\Psi$ is lower semicontinuous, it follows by the final boundedness statement of Lemma~\ref{lm:simple-algo-well-defined} and the assumed boundedness of $(\ol{\delta}_k)$ that 
$\Psi(M(\xi^*)) < \infty$
%\begin{align}
%\Psi(M(\xi^*)) \le \Psi(M(\xi^0)) + \sup_{k\in\N_0} \ol{\delta}_k < \infty
%\end{align}
and therefore $\xi^* \in \Xifin(X)$, as desired.
\end{proof}

With these preparations at hand, we can now show that Algorithm~\ref{algo:no-exchange} or~\ref{algo:no-exchange-arbitrary-violators} applied with a tolerance $\eps > 0$ is guaranteed to terminate, namely at an $\eps$-optimal design for $\Psi \circ M$. %Additionally, if the mere continuity assumption for the sensitivity function is sharpened to a uniform local Lipschitz continuity assumption, then the number of iterations until termination can also be estimated, namely in terms of packing numbers~\cite{Ti93}. 
Additionally, if instead of the mere continuity assumption, the sensitivity function is assumed to be locally  Lipschitz continuous uniformly w.r.t.~$x \in X$, then the number of iterations until termination can also be bounded above, namely in terms of packing numbers~\cite{Ti93}. See~Corollaries~\ref{cor:numbers-of-iterations-until-termination-under-sublinear-assumptions} and~\ref{cor:numbers-of-iterations-until-termination-under-sublinear-assumptions} for further estimates on the numbers of iterations until termination.  As was shown above (Example~\ref{ex:sensitivity-fcts-for-Psi_p}), this additional uniform local Lipschitz continuity assumption is satisfied for the sensitivity function $\psi_p$ of every standard design criterion $\Psi_p$ with $p \in [0,\infty)$. Also, to obtain a concrete uniform Lipschitz constant $L$ of $\psi_p(\cdot,x)|_{M(\Xi_{R_0}(X))}$ as specified after~\eqref{eq:Xi_R0(X)-def}, one can combine the estimate~\eqref{eq:M^-p-Lipschitz,5} from the proof of Lemma~\ref{lm:directional-derivatives} with Lemma~\ref{lm:upper-bound-on-Psi(M)-implies-lower-bound-on-M}. Specifically, one can use the following simple consequence of Lemma~\ref{lm:upper-bound-on-Psi(M)-implies-lower-bound-on-M}: %In fact, the latter lemma implies that
\begin{align}
M(\Xi_{R_0}(X)) = \{M(\xi): \xi \in \Xi(X) \text{ with } \Psi_p(M(\xi)) \le R_0\}
\subset \{M \in \Rpd^{d \times d}: c \le M \le C\}
\end{align}
with $c :=\mu_{p,R_0}(C) > 0$ and $C := \sup_{\xi \in \Xi_{R_0}(X)} \norm{M(\xi)} \le \max_{x\in X} \norm{m(x)}< \infty$.

\begin{thm} \label{thm:finite-termination-simple-algo}
Suppose that Conditions~\ref{cond:shared-assumption-1-algos-with-and-without-exchange} and~\ref{cond:shared-assumption-2-algos-with-and-without-exchange} are satisfied. Suppose further that $X^0$ is a finite subset of $X$ satisfying~\eqref{eq:finite-criterion-on-Xi(X^0)} and that $(\xi^k)_{k\in K}$ is generated by Algorithm~\ref{algo:no-exchange} or~\ref{algo:no-exchange-arbitrary-violators} with optimality tolerances $\eps > 0$ and $\ol{\delta}_k, \ul{\delta}_k \in [0,\infty)$ such that
\begin{align} \label{eq:opt-tolerances-assumption-without-exchange}
\eps > \ol{\delta} + \ul{\delta},
%\limsup_{k\to\infty}\ol{\delta}_k + \limsup_{k\to\infty}\ul{\delta}_k.
\end{align}
where $\ol{\delta} := \limsup_{k\to\infty}\ol{\delta}_k$ and $\ul{\delta} := \limsup_{k\to\infty}\ul{\delta}_k$. 
Then $(\xi^k)$ terminates at an $\eps$-optimal design for $\Psi \circ M$. If, in addition, the sensitivity function $\psi$ is locally Lipschitz continuous w.r.t.~$M \in \dom \Psi$ uniformly w.r.t.~$x \in X$ and if instead of~\eqref{eq:opt-tolerances-assumption-without-exchange} one even has $\eps > \sup_{k\in\N_0} \ol{\delta}_k + \sup_{k\in\N_0} \ul{\delta}_k$, then the number of iterations until termination can be estimated above as follows:
\begin{align} \label{eq:finite-termination-simple-algo-upper-bound-on-termination-index}
|K| \le \operatorname{pac}_{\eps^*/L,\norm{\cdot}}(M(\Xi_{R_0}(X))).
\end{align}
In this estimate, $\eps^* := \eps - \ol{\delta}^* - \ul{\delta}^*$ with $\ol{\delta}^* := \sup_{k\in\N_0} \ol{\delta}_k$ and $\ul{\delta}^* := \sup_{k\in\N_0}\ul{\delta}_k$, whereas
\begin{align} \label{eq:Xi_R0(X)-def}
\Xi_{R_0}(X) := \{ \xi \in \Xi(X): \Psi(M(\xi)) \le R_0\}
\qquad \text{with} \qquad
R_0 := \Psi(M(\xi^0)) + \ol{\delta}^*
\end{align}
and $L$ is a uniform Lipschitz constant of the restricted functions $\psi(\cdot,x)|_{M(\Xi_{R_0}(X))}$. 
%whereas $L$ is a uniform Lipschitz constant of the restricted functions $\psi(\cdot,x)|_{M(\Xi_{R_0}(X))}$ with $\Xi_{R_0}(X) := \{ \xi \in \Xi(X): \Psi(M(\xi)) \le R_0\}$ and $R_0 := \Psi(M(\xi^0)) + \ol{\delta}^*$. 
\end{thm}

\begin{proof}
As a first step, we show that $(\xi^k)$ terminates at an $\eps$-optimal design for $\Psi \circ M$. 
Assume that $(\xi^k)_{k\in K}$ does not terminate, that is, $K = \N_0$. It then follows by the algorithms' termination conditions %of Algorithm~\ref{algo:no-exchange} and~\ref{algo:no-exchange-arbitrary-violators} 
that
\begin{align} \label{eq:termination-without-exchange-1}
\psi(M(\xi^k),x^k) < -\eps + \ul{\delta}_k \qquad (k \in \N_0).
\end{align}
(In fact, in the case where Algorithm~\ref{algo:no-exchange-arbitrary-violators} is used, this even holds with $\ul{\delta}_k = 0$.)
Since $\xi^l$ is a $\ol{\delta}_l$-optimal design on $X^l$ and since $X^l \supset X^{k+1} \ni x^k$ for all $l > k$ by the algorithms' definitions, it further follows that
\begin{align} \label{eq:termination-without-exchange-2}
\psi(M(\xi^l),x^k) \ge - \ol{\delta}_l \qquad (k,l \in \N_0 \text{ with } k < l)
\end{align}
by virtue of Theorem~\ref{thm:equivalence-thm} with $X$ replaced by $X^l$. We now choose convergent subsequences $(\xi^{k_j})$ and $(x^{k_j})$ with limits denoted by $\xi^*$ and $x^*$ (Lemma~\ref{lm:prohorov}). It then follows that $\xi^* \in \Xifin(X)$ (Lemma~\ref{lm:termination-lemma-without-exchange}) and from~\eqref{eq:termination-without-exchange-1} and~\eqref{eq:termination-without-exchange-2} it further follows by the continuity of $\psi$ 
that
\begin{gather}
\psi(M(\xi^*),x^*) = \lim_{j\to\infty} \psi(M(\xi^{k_j}),x^{k_j}) \le -\eps + \ul{\delta} \label{eq:termination-without-exchange-3}\\
\psi(M(\xi^*),x^*) = \lim_{j\to\infty} \psi(M(\xi^{k_{j+1}}),x^{k_j}) \ge -\limsup_{j\to\infty} \ol{\delta}_{k_{j+1}} \ge -\ol{\delta}. \label{eq:termination-without-exchange-4}
\end{gather}
%In the first equality and inequality of~\eqref{eq:termination-with-exchange-4}, we used the trivial fact that $\xi^* = \lim_{j\to\infty}\xi^{k_j} = \lim_{j\to\infty}\xi^{k_{j+1}}$ and that $k_{j+1} > k_j$. 
Combining now the inequalities~\eqref{eq:termination-without-exchange-3} and~\eqref{eq:termination-without-exchange-4}, we conclude that $-\eps + \ul{\delta} \ge -\ol{\delta}$. Contradiction to~\eqref{eq:opt-tolerances-assumption-without-exchange}! So, our assumption that $(\xi^k)$ does not terminate is false. And therefore $(\xi^k)$ terminates at an $\eps$-optimal design for $\Psi \circ M$ by Lemma~\ref{lm:termination-lemma-without-exchange}~(i), as desired. 
\smallskip

As a second step, we show that under the additional Lipschitz continuity assumption on the sensitivity function and under the sharpened optimality-tolerance assumption $\eps > \ol{\delta}^* + \ul{\delta}^*$, the number of iterations until termination can be estimated above as in~\eqref{eq:finite-termination-simple-algo-upper-bound-on-termination-index}. 
So, let the sensitivity function $\psi$ be locally Lipschitz continuous w.r.t.~$M \in \dom \Psi$ uniformly w.r.t.~$x \in X$, that is, for every compact subset $\mathcal{M}$ of $\dom\Psi$ the restricted functions $\psi(\cdot,x)|_\mathcal{M}$ are Lipschitz continuous uniformly w.r.t.~$x\in X$. %with an $x$-independent Lipschitz constant. 
Since $M(\Xi_{R_0}(X))$ is a compact subset of $\dom\Psi$ by the lower semicontinuity of $\Psi$, it follows that %the functions $\psi(\cdot,x)|_{M(\Xi_{R_0}(X))}$ are Lipschitz continuous uniformly w.r.t.~$x \in X$
there exists an $x$-independent Lipschitz constant $L \in (0,\infty)$ such that
\begin{align} \label{eq:termination-without-exchange-5}
\big| \psi(M(\eta),x)-\psi(M(\xi),x) \big| \le L \norm{M(\eta)-M(\xi)}
\qquad (\xi, \eta \in \Xi_{R_0}(X) \text{ and } x \in X). 
\end{align}
It further follows by the final boundedness statement of Lemma~\ref{lm:simple-algo-well-defined} that 
\begin{align} \label{eq:termination-without-exchange-6}
\xi^k \in \Xi_{R_0}(X) \qquad (k \in K).
\end{align}
Combining now~\eqref{eq:termination-without-exchange-1}, \eqref{eq:termination-without-exchange-2}, \eqref{eq:termination-without-exchange-5} and \eqref{eq:termination-without-exchange-6}, we see that
\begin{align} \label{eq:termination-without-exchange-7}
0 < \eps^* 
\le \eps - \ul{\delta}_k - \ol{\delta}_l 
&< \psi(M(\xi^l),x^k)-\psi(M(\xi^k),x^k) \notag\\
&\le L\norm{M(\xi^l)-M(\xi^k)}
\qquad (k,l \in K \text{ with } k < l).
\end{align}
In other words, \eqref{eq:termination-without-exchange-6} and \eqref{eq:termination-without-exchange-7} say that $\{M(\xi^k): k \in K\}$ is an $(\eps^*/L)$-packing of the set $M(\Xi_{R_0}(X))$ w.r.t.~$\norm{\cdot}$. And therefore, we have
\begin{align}
|K| = \big| \{M(\xi^k): k \in K\} \big| \le \operatorname{pac}_{\eps^*/L,\norm{\cdot}}(M(\Xi_{R_0}(X)))
\end{align}
by the definition of the packing number~\cite{Ti93}, as desired. 
\end{proof}

%As was shown above (Example~\ref{ex:sensitivity-fcts-for-Psi_p}), the additional uniform local Lipschitz continuity assumption is satisfied for the sensitivity function $\psi_p$ of every standard design criterion $\Psi_p$ with $p \in [0,\infty)$. In order to obtain a concrete Lipschitz constant $L$ of $\psi_p|_{M(\Xi_{R_0}(X))}$ as required in the above result, one can combine the estimate~\eqref{eq:M^-p-Lipschitz,5} from the proof of Lemma~\ref{lm:directional-derivatives} with Lemma~\ref{lm:upper-bound-on-Psi(M)-implies-lower-bound-on-M}. Specifically, one can use the following simple consequence of Lemma~\ref{lm:upper-bound-on-Psi(M)-implies-lower-bound-on-M}: %In fact, the latter lemma implies that
%\begin{align}
%M(\Xi_{R_0}(X)) = \{M(\xi): \xi \in \Xi(X) \text{ with } \Psi_p(M(\xi)) \le R_0\}
%\subset \{M \in \Rpd^{d \times d}: c \le M \le C\}
%\end{align}
%with $c :=\mu_{p,R_0}(C) > 0$ and $C := \sup_{\xi \in \Xi_{R_0}(X)} \norm{M(\xi)} \le \max_{x\in X} \norm{m(x)}< \infty$.
%
%In a similar manner, 
With essentially the same arguments, we can also prove that if Algorithm~\ref{algo:no-exchange} is applied with tolerance $\eps = 0$, then the sequence of iterates %generated by it %either terminates or 
accumulates at an optimal design for $\Psi\circ M$. See Corollary~\ref{cor:convergence} below.

\subsection{Adaptive discretization algorithms with exchange} %in updating the discretizations
%\subsection{An algorithm with exchange} %in updating the discretizations

We now modify the above algorithms without exchange 
to algorithms with exchange in updating the discretizations. 
Instead of taking into account all points from the discretization $X^k$ when passing to the next discretization $X^{k+1}$, we now take into account only the support points of $\xi^k$. In particular, the discretization sets are no longer guaranteed to increase from iteration to iteration as they did for the algorithms without exchange by~\eqref{eq:discretization-sets-get-larger-and-larger}.

\begin{algo} \label{algo:exchange}
Input: a finite subset $X^0$ of $X$ and optimality tolerances $\eps, \ol{\delta}_k, \ul{\delta}_k \in [0,\infty)$. %With these inputs, proceed in the following steps.
Initialize $k=0$ and then proceed in the following steps. 
\begin{itemize}
\item[1.] Compute a $\ol{\delta}_k$-approximate solution $\xi^k \in \Xi(X^k)$ of the discretized optimal design problem
\begin{align} \label{eq:DOE(X^k)-exchange}
\min_{\xi\in\Xi(X^k)} \Psi(M(\xi))
\end{align}
\item[2.] Compute a $\ul{\delta}_k$-approximate solution $x^k \in X$ of the strongest optimality violator problem
\begin{align} \label{eq:Viol(xi^k)-exchange}
\min_{x\in X} \psi(M(\xi^k),x).
\end{align}
If $\psi(M(\xi^k),x^k) < -\eps + \ul{\delta}_k$, then set $X^{k+1} := \supp \xi^k \cup \{x^k\}$ 
%\begin{align}
%X^{k+1} := \{x \in X^k: \psi(M(\xi^k),x) = 0\} \cup \{x^k\}
%\end{align}
and return to Step~1 with $k$ replaced by $k+1$. If $\psi(M(\xi^k),x^k) \ge -\eps + \ul{\delta}_k$, then terminate. 
\end{itemize}
\end{algo}

\begin{algo} \label{algo:exchange-arbitrary-violators}
Input: a finite subset $X^0$ of $X$ and optimality tolerances $\eps, \ol{\delta}_k \in [0,\infty)$. %With these inputs, proceed in the following steps.
Initialize $k=0$ and then proceed in the following steps. 
\begin{itemize}
\item[1.] Compute a $\ol{\delta}_k$-approximate solution $\xi^k \in \Xi(X^k)$ of the discretized optimal design problem
\begin{align} \label{eq:DOE(X^k)-exchange-arbitrary-violators}
\min_{\xi\in\Xi(X^k)} \Psi(M(\xi))
\end{align}
\item[2.] Search for a violator of the $\eps$-optimality condition for $\xi^k$, that is, for a point $x^k$ in the set
\begin{align} \label{eq:optimality-violator-set,exchange}
\{x \in X: \psi(M(\xi^k),x) < -\eps \}.
\end{align}
If such a point $x^k$ is found, then set $X^{k+1} := X^k \cup \{x^k\}$ 
%\begin{align}
%X^{k+1} := X^k \cup \{x^k\}
%\end{align}
and return to Step~1 with $k$ replaced by $k+1$. If no such point can possibly be found (that is, the set~\eqref{eq:optimality-violator-set,exchange} is empty), then terminate. 
\end{itemize}
\end{algo}

We first convince ourselves that for a suitable initial discretization $X^0$ %the above algorithm is well-defined
all the optimization problems encountered in the course of Algorithm~\ref{algo:exchange} or~\ref{algo:exchange-arbitrary-violators} are solvable and that the sequence $(\Psi(M(\xi^k))$ of optimal values is essentially monotonically decreasing, provided that the sequence $(\ol{\delta}_k)$ of optimality tolerances for~\eqref{eq:DOE(X^k)-exchange} is summable. 

\begin{lm} \label{lm:exchange-algo-well-defined}
Suppose that Conditions~\ref{cond:shared-assumption-1-algos-with-and-without-exchange} and~\ref{cond:shared-assumption-2-algos-with-and-without-exchange} are satisfied and that $X^0$ is a finite subset of $X$ with~\eqref{eq:finite-criterion-on-Xi(X^0)} and $\eps, \ol{\delta}_k, \ul{\delta}_k \in [0,\infty)$ are arbitrary optimality tolerances.
%Then there exists a finite subset $X^0$ of $X$ with
%\begin{align} \label{eq:finite-criterion-on-Xi(X^0)}
%\Xifin(X^0) := \{\xi \in \Xi(X^0): \Psi(M(\xi)) < \infty\} \ne \emptyset. 
%\end{align}
%Additionally, 
Then all optimization problems in Algorithm~\ref{algo:exchange} or~\ref{algo:exchange-arbitrary-violators} %with the above inputs %with such an $X^0$ and arbitrary optimality tolerances $\eps, \ol{\delta}_k, \ul{\delta}_k \in [0,\infty)$ %as its inputs 
are solvable and, for the (possibly finite) sequence $(\xi^k)_{k\in K}$ of iterates generated by Algorithm~\ref{algo:exchange} or~\ref{algo:exchange-arbitrary-violators} with these inputs, one has
\begin{align} \label{eq:exchange-algo-well-defined}
\Psi(M(\xi^{k+1})) \le \Psi\big( \alpha M(\xi^k) + (1-\alpha) M(\xi^{k+1}) \big) + \ol{\delta}_{k+1}
\qquad (k \in K-1 \text{ and } \alpha \in [0,1]).
\end{align} 
In particular, $\xi^k \in \Xifin(X)$ for all $k \in K$ and, if $(\ol{\delta}_k)$ is summable, then there is a null sequence $(\nu_k)$ in $\R$ such that $(\Psi(M(\xi^k)) + \nu_k)_{k\in K}$ is monotonically decreasing. %$M(\xi^k) \in \dom\Psi$ for all $k \in K$.
\end{lm}

\begin{proof}
It is straightforward to see by induction over $K$ that for every iteration index $k \in K$ the following statements are true: (i) the discretized design problem~\eqref{eq:DOE(X^k)-exchange} or~\eqref{eq:DOE(X^k)-exchange-arbitrary-violators}, respectively, has a solution (so that $\xi^k$ is well-defined also when $\ol{\delta}_k = 0$), (ii) if $k\ne 0$, then $\Psi(M(\xi^{k})) \le \Psi(\alpha M(\xi^{k-1}) + (1-\alpha)M(\xi^k)) + \ol{\delta}_{k}$ for all $\alpha \in [0,1]$, (iii) $\xi^k \in \Xifin(X)$, and (iv) the strongest violator problem~\eqref{eq:Viol(xi^k)-simple} has a solution. In order to see~(ii), notice that by the algorithms' updating rule one has
\begin{align}
\supp(\alpha \xi^{k-1} + (1-\alpha) \xi^k) \subset \supp \xi^{k-1} \cup \supp \xi^k \subset X^k
\qquad (\alpha \in [0,1])
\end{align}
and therefore $\alpha \xi^{k-1} + (1-\alpha) \xi^k \in \Xi(X^k)$ for all $\alpha \in [0,1]$. Combining this with the $\ol{\delta}_k$-approximate solution property of $\xi^k$, one immediately obtains (ii). 
It only remains to prove the monotonic decreasingness of $(\Psi(M(\xi^k)) + \nu_k)_{k\in K}$ for a suitable null sequence $(\nu_k)$, but this is straightforward as well. Indeed, define $\nu_k := \sum_{l=k}^\infty \ol{\delta}_{l+1}$. Since $(\ol{\delta}_k)$ is summable by assumption, $(\nu_k)$ is a null sequence in $\R$ and, by~\eqref{eq:exchange-algo-well-defined} with $\alpha = 1$, it also follows that $(\Psi(M(\xi^k)) + \nu_k)_{k\in K}$ is monotonically decreasing, as desired.
\end{proof}

In contrast to the termination lemma for Algorithm~\ref{algo:no-exchange} and~\ref{algo:no-exchange-arbitrary-violators}, we now have to additionally assume that the design criterion $\Psi|_{\dom\Psi}$ be strictly convex. In view of Proposition~\ref{prop:standard-properties-for-Psi_p}, this additional assumption is satisfied for all standard design criteria $\Psi_p$ with $p \in [0,\infty)$. 

\begin{lm} \label{lm:termination-lemma-with-exchange}
Suppose that Conditions~\ref{cond:shared-assumption-1-algos-with-and-without-exchange} and~\ref{cond:shared-assumption-2-algos-with-and-without-exchange} are satisfied and that $\Psi$ is even strictly convex. Suppose further that $X^0$ is a finite subset of $X$ satisfying~\eqref{eq:finite-criterion-on-Xi(X^0)} and that $(\xi^k)_{k\in K}$ is generated by Algorithm~\ref{algo:exchange} or~\ref{algo:exchange-arbitrary-violators} %with such an $X^0$ and 
with optimality tolerances $\eps, \ol{\delta}_k, \ul{\delta}_k \in [0,\infty)$ such that $(\ol{\delta}_k)$ is summable.
\begin{itemize}
\item[(i)] If $(\xi^k)$ is terminating, then it terminates at an $\eps$-optimal design for $\Psi \circ M$.
\item[(ii)] If $(\xi^k)$ is non-terminating, then every accumuluation point $\xi^*$ of $(\xi^k)$ belongs to $\Xifin(X)$ and for every subsequence $(\xi^{k_l})$ with $\xi^{k_l} \longrightarrow \xi^*$ there exists yet another subsequence $(k_{l_j})$ such that
\begin{align} \label{eq:termination-lm-with-exchange}
M(\xi^{k_{l_j}+1}) \longrightarrow M(\xi^*) \qquad (j \to \infty).
\end{align}
\end{itemize}
\end{lm}

\begin{proof}
(i) Suppose that $(\xi^k)_{k\in K}$ is terminating, that is, $K = \{0,\dots, k^*\}$ for some terminal index $k^* \in \N_0$. It then follows by Lemma~\ref{lm:exchange-algo-well-defined} and by the termination rules of  Algorithm~\ref{algo:exchange} and~\ref{algo:exchange-arbitrary-violators}, respectively, that 
\begin{align} \label{eq:termination-lm-with-exchange,1}
\xi^{k^*} \in \Xifin(X) \qquad \text{and} \qquad \min_{x\in X} \psi(M(\xi^{k^*}),x) \ge -\eps.
\end{align}
And therefore, $\xi^{k^*}$ is an $\eps$-optimal design for $\Psi\circ M$ by virtue of Theorem~\ref{thm:equivalence-thm}, as desired.
\smallskip

(ii) Suppose now that $(\xi^k)_{k\in K}$ is non-terminating, that is, $K = \N_0$. Also, let $\xi^*$ be any accumulation point of $(\xi^k)$ and let $(\xi^{k_l})$ be any subsequence with $\xi^{k_l} \longrightarrow \xi^*$ as $l\to\infty$. Since $\Xi(X)$ is sequentially compact (Lemma~\ref{lm:prohorov}), there exists yet another subsequence $(k_{l_j})$ and a $\xi^{**} \in \Xi(X)$ such that $\xi^{k_{l_j}+1} \longrightarrow \xi^{**}$ as $j\to\infty$. Consequently, 
\begin{align} \label{eq:termination-lm-with-exchange,2}
M(\xi^{k_l}) \longrightarrow M(\xi^*) \qquad (l\to\infty)
\quad \text{and} \quad
M(\xi^{k_{l_j}+1}) \longrightarrow M(\xi^{**}) \qquad (j\to\infty).
\end{align}
Since $(\Psi(M(\xi^k)+\nu_k)$ is monotonically decreasing for a null sequence $(\nu_k)$ (Lemma~\ref{lm:exchange-algo-well-defined}), the subsequences are monotonically decreasing as well and
\begin{align} \label{eq:termination-lm-with-exchange,3}
\lim_{l\to\infty} \Psi(M(\xi^{k_l})) 
&= \lim_{l\to\infty} \big( \Psi(M(\xi^{k_l})) + \nu_{k_l}\big) 
= \inf_{k\in\N_0} (\Psi(M(\xi^k) + \nu_k) \notag\\
&= \lim_{j\to\infty} \big(\Psi(M(\xi^{k_{l_j}+1})) + \nu_{k_{l_j}+1}\big)
= \lim_{j\to\infty} \Psi(M(\xi^{k_{l_j}+1})).
\end{align}
Since, moreover, $\Psi$ is lower semicontinuous by assumption, \eqref{eq:termination-lm-with-exchange,2}, \eqref{eq:termination-lm-with-exchange,3} %and $\xi^0 \in \Xifin(X)$ (Lemma~\ref{lm:termination-lemma-with-exchange})
further imply
\begin{align}
\Psi(M(\xi^*)), \Psi(M(\xi^{**})) \le \inf_{k\in\N_0} (\Psi(M(\xi^k) + \nu_k) \le \Psi(M(\xi^0)) + \nu_0 < \infty
\end{align}
and therefore
\begin{align} \label{eq:termination-lm-with-exchange,4}
\xi^*, \xi^{**}, \xi^k \in \Xifin(X) \qquad (k \in \N_0).
\end{align}
In view of the continuity of $(\Psi \circ M)|_{\Xifin(X)}$ (Lemma~\ref{lm:Psi-continuous}), it follows from~\eqref{eq:termination-lm-with-exchange,2} and~\eqref{eq:termination-lm-with-exchange,4} and by \eqref{eq:exchange-algo-well-defined} with $\alpha := 1/2$ that, on the one hand,
\begin{align} \label{eq:termination-lm-with-exchange,5}
\Psi(M(\xi^{**})) &= \lim_{j\to\infty} \Psi(M(\xi^{k_{l_j}+1})) \le \lim_{j\to\infty} \big( \Psi\big( M( \xi^{k_{l_j}}/2 + \xi^{k_{l_j}+1}/2) \big) + \ol{\delta}_{k_{l_j}+1} \big) \notag\\
&= \Psi(M(\xi^*/2 + \xi^{**}/2)) = \Psi(M(\xi^*)/2 + M(\xi^{**})/2).
\end{align}
In view of the continuity of $(\Psi \circ M)|_{\Xifin(X)}$ (Lemma~\ref{lm:Psi-continuous}), it further follows from~\eqref{eq:termination-lm-with-exchange,2} and~\eqref{eq:termination-lm-with-exchange,4} and by~\eqref{eq:termination-lm-with-exchange,3} that, on the other hand,
\begin{align} \label{eq:termination-lm-with-exchange,6}
\Psi(M(\xi^*)) = \lim_{l\to\infty} \Psi(M(\xi^{k_l}))= \lim_{j\to\infty} \Psi(M(\xi^{k_{l_j}+1}))= \Psi(M(\xi^{**})).
\end{align}
Since $\Psi$ is strictly convex by assumption, we conclude from~\eqref{eq:termination-lm-with-exchange,5} and~\eqref{eq:termination-lm-with-exchange,6} that
\begin{align} \label{eq:termination-lm-with-exchange,7}
M(\xi^{**}) = M(\xi^*).
\end{align}
Combining now~(\ref{eq:termination-lm-with-exchange,2}.b) and~\eqref{eq:termination-lm-with-exchange,7}, we finally obtain the desired convergence~\eqref{eq:termination-lm-with-exchange} and thus the lemma is proved. 
\end{proof}

With these preparations at hand, we can now show that Algorithm~\ref{algo:exchange} or~\ref{algo:exchange-arbitrary-violators} applied with a tolerance $\eps > 0$ is guaranteed to terminate, namely at an $\eps$-optimal design for $\Psi \circ M$. In contrast to Algorithm~\ref{algo:no-exchange} and~\ref{algo:no-exchange-arbitrary-violators}, a  bound on the numbers of iterations until termination is not so easy to establish anymore. See Corollaries~\ref{cor:numbers-of-iterations-until-termination-under-sublinear-assumptions} and~\ref{cor:numbers-of-iterations-until-termination-under-linear-assumptions}, however, which establish such bounds at least for Algorithm~\ref{algo:exchange}. 
%In contrast to Algorithm~\ref{algo:no-exchange} and~\ref{algo:no-exchange-arbitrary-violators}, however, an explicit bound on the numbers of iterations until termination is not so easy to establish anymore. In fact, the proof of that iteration number bound essentially uses the fact~\eqref{eq:discretization-sets-get-larger-and-larger} that the discretization sets of Algorithm~\ref{algo:no-exchange} and~\ref{algo:no-exchange-arbitrary-violators} get larger and larger from iteration to iteration. (Specifically, this is used in~\eqref{eq:termination-without-exchange-2} and~\eqref{eq:termination-without-exchange-7}.)

\begin{thm} \label{thm:finite-termination-exchange-algo}
Suppose that Conditions~\ref{cond:shared-assumption-1-algos-with-and-without-exchange} and~\ref{cond:shared-assumption-2-algos-with-and-without-exchange} are satisfied and that $\Psi$ is even strictly convex. Suppose further that $X^0$ is a finite subset of $X$ satisfying~\eqref{eq:finite-criterion-on-Xi(X^0)} and that $(\xi^k)_{k\in K}$ is generated by Algorithm~\ref{algo:exchange} or~\ref{algo:exchange-arbitrary-violators} %with such an $X^0$ and 
with optimality tolerances $\eps > 0$ and $\ol{\delta}_k, \ul{\delta}_k \in [0,\infty)$ such that $(\ol{\delta}_k)$ is summable and
\begin{align} \label{eq:opt-tolerances-assumption-with-exchange}
\eps > \ul{\delta},
\end{align}
where $\ul{\delta} := \limsup_{k\to\infty} \ul{\delta}_k$. Then $(\xi^k)$ terminates at an $\eps$-optimal design for $\Psi \circ M$. 
\end{thm}

\begin{proof}
Assume that $(\xi^k)_{k\in K}$ does not terminate, that is, $K = \N_0$. It then follows by the algorithms' termination conditions %of Algorithm~\ref{algo:exchange} and~\ref{algo:exchange-arbitrary-violators} 
that
\begin{align} \label{eq:termination-with-exchange-1}
\psi(M(\xi^k),x^k) < -\eps + \ul{\delta}_k \qquad (k \in \N_0).
\end{align}
(In fact, in the case where Algorithm~\ref{algo:exchange-arbitrary-violators} is used, this even holds with $\ul{\delta}_k = 0$.)
Since $\xi^{k+1}$ is a $\ol{\delta}_{k+1}$-optimal design on $X^{k+1}$ and since $X^{k+1} = \supp \xi^k \cup \{x^k\} \ni x^k$ by the algorithms' definitions, it further follows that
\begin{align} \label{eq:termination-with-exchange-2}
\psi(M(\xi^{k+1}),x^k) \ge - \ol{\delta}_{k+1} \qquad (k \in \N_0)
\end{align}
by virtue of Theorem~\ref{thm:equivalence-thm} with $X$ replaced by $X^{k+1}$. We now choose convergent subsequences $(\xi^{k_l})$ and $(x^{k_l})$ with limits denoted by $\xi^*$ and $x^*$ (Lemma~\ref{lm:prohorov}) and another subsequence $(k_{l_j})$ such that~\eqref{eq:termination-lm-with-exchange} is satisfied (Lemma~\ref{lm:termination-lemma-with-exchange}). It then follows that $\xi^* \in \Xifin(X)$ (Lemma~\ref{lm:termination-lemma-with-exchange}) and from~\eqref{eq:termination-with-exchange-1} and~\eqref{eq:termination-with-exchange-2} it further follows by the continuity of $\psi$ %and by~\eqref{eq:termination-lm-with-exchange}
that
\begin{gather}
\psi(M(\xi^*),x^*) = \lim_{l\to\infty} \psi(M(\xi^{k_l}),x^{k_l}) \le -\eps + \ul{\delta} \label{eq:termination-with-exchange-3}\\
\psi(M(\xi^*),x^*) = \lim_{j\to\infty} \psi(M(\xi^{k_{l_j}+1}),x^{k_{l_j}}) \ge -\lim_{j\to\infty} \ol{\delta}_{k_{l_j}+1} = 0. \label{eq:termination-with-exchange-4}
\end{gather}
In the last equality of~\eqref{eq:termination-with-exchange-4}, we used the assumed summability of $(\ol{\delta}_k)$. 
Combining now the inequalities~\eqref{eq:termination-with-exchange-3} and~\eqref{eq:termination-with-exchange-4}, we conclude that $-\eps + \ul{\delta} \ge 0$. Contradiction to~\eqref{eq:opt-tolerances-assumption-with-exchange}! So, our assumption that $(\xi^k)$ does not terminate is false. And therefore $(\xi^k)$ terminates at an $\eps$-optimal design for $\Psi \circ M$ by Lemma~\ref{lm:termination-lemma-with-exchange}~(i), as desired. 
\end{proof}

With essentially the same arguments, we can also prove that if Algorithm~\ref{algo:exchange} is applied with tolerance $\eps = 0$, then the sequence of iterates %generated by it %either terminates or 
accumulates at an optimal design for $\Psi\circ M$. See Corollary~\ref{cor:convergence} below.

\subsection{A basic convergence result}

We move on to establish a basic convergence result on our strict algorithms (Algorithms~\ref{algo:no-exchange} and~\ref{algo:exchange}). It says that when applied with optimality tolerance $\eps = 0$, the strict algorithms might not terminate anymore but still converge to an optimal design as $k \to \infty$.

\begin{cor} \label{cor:convergence}
Suppose that Conditions~\ref{cond:shared-assumption-1-algos-with-and-without-exchange} and~\ref{cond:shared-assumption-2-algos-with-and-without-exchange} are satisfied. Suppose further that $X^0$ is a finite subset of $X$ satisfying~\eqref{eq:finite-criterion-on-Xi(X^0)} and that $(\xi^k)_{k\in K}$ is generated by Algorithm~\ref{algo:no-exchange} or Algorithm~\ref{algo:exchange} %with such an $X^0$ and 
with optimality tolerances $\eps = 0$ and $\ol{\delta}_k, \ul{\delta}_k \in [0,\infty)$ such that $\lim_{k\to\infty} \ol{\delta}_k = 0 = \lim_{k\to\infty} \ul{\delta}_k$. In case Algorithm~\ref{algo:exchange} is used, additionally assume that $\Psi$ is even strictly convex and that $(\ol{\delta}_k)$ is even summable.
\begin{itemize}
\item[(i)] If $(\xi^k)$ is terminating, then it terminates at an optimal design for $\Psi\circ M$.
\item[(ii)] If $(\xi^k)$ is non-terminating, then each of its accumulation points is an optimal design for $\Psi\circ M$. 
%Additionally, for strictly convex $\Psi$, the sequence of information matrices $(M(\xi^k))$ converges to the unique solution of
%\begin{align} \label{eq:doe-in-M}
%\min_{M \in M(\Xi(X))} \Psi(M).
%\end{align}
\end{itemize} 
Additionally, for a strictly convex design criterion $\Psi$, the sequence of information matrices $(M(\xi^k))$ converges to the unique solution of
\begin{align} \label{eq:doe-in-M}
\min_{M \in M(\Xi(X))} \Psi(M).
\end{align}
\end{cor} 

\begin{proof}
Assertion~(i) is clear by Lemma~\ref{lm:termination-lemma-without-exchange}~(i) or Lemma~\ref{lm:termination-lemma-with-exchange}~(i), respectively, and we have only to prove assertion~(ii) and the additional convergence assertion concerning the information matrices. Suppose therefore that $(\xi^k)$ is non-terminating for the entire proof. 
\smallskip

As a first step, we show that every accumulation point $\xi^*$ of $(\xi^k)$ is an optimal design for $\Psi\circ M$. So, let $\xi^*$ be an accumulation point of $(\xi^k)$. 
With the same arguments as in the proof of the termination theorem for Algorithm~\ref{algo:no-exchange} or Algorithm~\ref{algo:exchange}, respectively, it then follows that $\xi^* \in \Xifin(X)$ and that there are subsequences $(\xi^{k_l})$ and $(x^{k_l})$ and an $x^* \in X$ such that
%\begin{gather}
%\xi^{k_l} \longrightarrow \xi^* 
%\qquad \text{and} \qquad
%x^{k_l} \longrightarrow x^* \\
%\psi(M(\xi^*),x^*) \ge 0.
%\end{gather}
\begin{align} \label{eq:convergence,1}
\xi^{k_l} \longrightarrow \xi^* 
\qquad \text{and} \qquad
x^{k_l} \longrightarrow x^*
\qquad (l\to\infty)
\end{align}
and such that the estimate~\eqref{eq:termination-without-exchange-4} or~\eqref{eq:termination-with-exchange-4} is satisfied, respectively (depending on whether Algorithm~\ref{algo:no-exchange} or~\ref{algo:exchange} is used). Since $x^k$ is a $\ul{\delta}_k$-approximate solution of~\eqref{eq:Viol(xi^k)-simple} or~\eqref{eq:Viol(xi^k)-exchange}, respectively, it further follows that
\begin{align} \label{eq:convergence,2}
\psi(M(\xi^k),x) \ge \psi(M(\xi^k),x^k) - \ul{\delta}_k 
\qquad (k \in \N_0 \text{ and } x \in X). 
\end{align}
Combining now~\eqref{eq:termination-without-exchange-4} or~\eqref{eq:termination-with-exchange-4} with  \eqref{eq:convergence,1} and~\eqref{eq:convergence,2}, we conclude by the continuity of $\psi$ and by the assumed convergences of $(\ol{\delta}_k)$ and $(\ul{\delta}_k)$ to $0$ that
\begin{align}
\psi(M(\xi^*),x) 
= \lim_{l\to\infty} \psi(M(\xi^{k_l}),x) \ge \lim_{l\to\infty} \big( \psi(M(\xi^{k_l}),x^{k_l}) - \ul{\delta}_{k_l} \big) = \psi(M(\xi^*),x^*) 
%\ge -\ol{\delta} = 0
\ge 0
\end{align}
for all $x \in X$. %where for the last equality we used that $\lim_{k\to\infty} \ol{\delta}_k = 0$ by assumption. 
And therefore, $\xi^*$ is an optimal design for $\Psi \circ M$ by virtue of Theorem~\ref{thm:equivalence-thm}, as desired.
\smallskip

As a second step, we show that for strictly convex $\Psi$ the sequence of information matrices $(M(\xi^k))$ converges to the unique solution $M^*$ of~\eqref{eq:doe-in-M}. So, let $\Psi$ be strictly convex. In particular, this implies that \eqref{eq:doe-in-M} indeed has a unique solution $M^*$. %, so that~\eqref{eq:doe-in-M} indeed has a unique solution $M^*$.
%and let $M^*$ be the solution of~\eqref{eq:doe-in-M} (which uniquely exists by Theorem~\ref{thm:ex-of-optimal-designs} and the strict convexity). %(where the unique solvability follows by Theorem~\ref{thm:ex-of-optimal-designs} and the assumed strict convexity of $\Psi$, of course). 
It then follows by the first step and Lemma~\ref{lm:prohorov} that every accumulation point of $(M(\xi^k))$ is a solution of~\eqref{eq:doe-in-M} %Suppose $M^{**}$ is any accumulation point of $(M(\xi^k))$. Then by the definition of accumulation points and Lemma~\ref{lm:prohorov} there exists a subsequence $(\xi^{k_l})$ and a $\xi^*$ with $M(\xi^{k_l}) \longrightarrow M^{**}$ and with $\xi^{k_l} \longrightarrow \xi^*$. Consequently, $M^{**} = M(\xi^*)$. Also, by the first step, $\xi^*$ is a $\Psi\circ M$-optimal design. Combining these two facts, we see that $M^{**}$ is a solution of~\eqref{eq:doe-in-M}, as desired.
and thus is equal to $M^*$. Consequently, $(M(\xi^k))$ has only one accumulation point, namely $M^*$. And therefore $(M(\xi^k))$ converges to $M^*$, as desired.
\end{proof}

\subsection{A general sublinear convergence rate result}

In this section, we improve the mere convergence result from the previous section to a convergence rate result with a sublinear convergence rate. %in the root sense (Section ?? of \cite{??})
In order to get this sublinear convergence rate, we have to strengthen the assumptions on the design criterion $\Psi$ a bit, namely essentially add an $L$-smoothness assumption on $\Psi$. As we will see, the thus strenghtened assumptions cover a large variety of situations (including those considered in~\cite{YaBiTa13}) and therefore pose almost no restriction from practical point of view. 
%Central idea behind this result: (recognize that) our algorithms are closely related to (are/can be seen as versions of) the fully corrective Frank-Wolfe algorithm and carry over convergence rate results for that algorithm from the literature (recognize that ... can be carried over)  

\begin{thm} \label{thm:sublin-convergence-rate}
Suppose that Conditions~\ref{cond:shared-assumption-1-algos-with-and-without-exchange} and~\ref{cond:shared-assumption-3-algos-with-and-without-exchange} are satisfied. Suppose further that $X^0$ is a finite subset of $X$ satisfying~\eqref{eq:finite-criterion-on-Xi(X^0)} and that $(\xi^k)_{k\in K}$ is generated by Algorithm~\ref{algo:no-exchange} or Algorithm~\ref{algo:exchange} with optimality tolerances $\eps, \ol{\delta}_k, \ul{\delta}_k \in [0,\infty)$ such that
\begin{align} \label{eq:opt-tolerances-of-subproblems}
\ol{\delta}_k \le \ol{c}/(k+1)^2 
\qquad \text{and} \qquad
\ul{\delta}_k \le \ul{c}/(k+2) 
\qquad (k \in \N_0)
\end{align}
for some constants $\ol{c}, \ul{c} \in [0,\infty)$. Suppose finally that 
\begin{align} \label{eq:M-subset-of-dom-Psi}
\mathcal{M} := \conv\big( \big\{ (1-\alpha)M(\xi^k) + \alpha m(x^k): k \in K \text{ and } \alpha \in [0,2/3] \big\} \big)
\subset \dom \Psi
\end{align}
and that $\Psi|_{\mathcal{M}}$ is $L$-smooth w.r.t.~$|\cdot|$ for some $L \in [0,\infty)$. %the derivative mapping $\mathcal{M} \ni M \mapsto D\Psi(M)$ restricted to $\mathcal{M}$ is Lipschitz continuous in the sense of~\eqref{eq:derivative-lipschitz-on-M}. 
We then have the following sublinear convergence rate estimate
\begin{align} \label{eq:sublin-convergence-rate-estimate}
\Psi(M(\xi^k)) - \Psi^* \le \frac{2}{k+2} \Big( (\Psi(M(\xi^1)) - \Psi^*) + L \big(\diam(M(\Xi(X))\big)^2 + \ol{c} + \ul{c} \Big)
\end{align}
for all $k \in K$ with $k \ge 2$, where $\Psi^* := \min_{\xi \in \Xi(X)} \Psi(M(\xi))$ and $\diam(M(\Xi(X)) := \sup\{|M-N|: M,N \in M(\Xi(X))\}$.
\end{thm} 

\begin{proof}
We proceed in three steps and in the entire proof we will use the following shorthand notation: 
\begin{align} \label{eq:sublin-convergence-rate-C-def}
C := \frac{L}{2} \big(\diam(M(\Xi(X))\big)^2.
\end{align}

As a first step, we show that 
\begin{align} \label{eq:sublin-convergence-rate-step1}
\Psi(M(\xi^{k+1})) \le \Psi(M(\xi^k)) + \alpha \psi(M(\xi^k),x^k) + C \alpha^2 + \ol{\delta}_{k+1}
\end{align}
for all $k \in K$ with $k+1 \in K$ and for all $\alpha \in [0,2/3]$. So, let $k, k+1 \in K$ and let $\alpha \in [0,2/3]$. It then follows by the updating rule both of Algorithm~\ref{algo:no-exchange} and of Algorithm~\ref{algo:exchange} that
\begin{align}
\supp\big( (1-\alpha) \xi^k + \alpha \delta_{x^k} \big) \subset \supp \xi^k \cup \{x^k\} \subset X^{k+1}
\end{align}
and therefore $(1-\alpha) \xi^k + \alpha \delta_{x^k} \in \Xi(X^{k+1})$. 
Since $\xi^{k+1}$ is a $\ol{\delta}_{k+1}$-optimal design on $X^{k+1}$ by the algorithms' definitions, it follows that
\begin{align}
\Psi(M(\xi^{k+1})) \le \min_{\xi \in \Xi(X^{k+1})} \Psi(M(\xi)) + \ol{\delta}_{k+1} 
\le \Psi\big( (1-\alpha)M(\xi^k) + \alpha m(x^k) \big) + \ol{\delta}_{k+1}.
\end{align}
Since, moreover, $(1-\alpha)M(\xi^k) + \alpha m(x^k) \in \mathcal{M}$ by the choice of $\alpha$, it further follows by the assumed $L$-smoothness of $\Psi|_{\mathcal{M}}$ %Lemma~\ref{lm:L-smoothness-sufficient-condition} 
and by Lemma~\ref{lm:sensi-fct-in-terms-of-derivative} 
that
\begin{align}
\Psi(M(\xi^{k+1})) 
&\le \Psi(M(\xi^k)) + \alpha D\Psi(M(\xi^k))(m(x^k)-M(\xi^k)) + \frac{L}{2} \big| m(x^k)-M(\xi^k) \big|^2 \alpha^2 + \ol{\delta}_{k+1} \notag\\
&\le \Psi(M(\xi^k)) + \alpha \psi(M(\xi^k),x^k) + C \alpha^2 + \ol{\delta}_{k+1},
\end{align}
which is precisely the desired estimate~\eqref{eq:sublin-convergence-rate-step1}.
\smallskip

As a second step, we show that 
\begin{align} \label{eq:sublin-convergence-rate-step2}
\Psi(M(\xi^{k+1})) - \Psi^* \le (1-\alpha) \big( \Psi(M(\xi^k)-\Psi^* \big) + C\alpha^2 + \ol{\delta}_{k+1} + \ul{\delta}_k \alpha
\end{align}
for all  all $k \in K$ with $k+1 \in K$ and for all $\alpha \in [0,2/3]$. So, let $k, k+1 \in K$ and let $\alpha \in [0,2/3]$. Also, let $\xi^* \in \Xi(X)$ be an optimal design for $\Psi \circ M$ on $X$, that is, $\Psi(M(\xi^*)) = \Psi^*$. %It then follows by the $\ul{\delta}_k$-approximate solution property of $x^k$
Since $x^k$ is a $\ul{\delta}_k$-approximate solution of the worst optimality violator problem for $\xi^k$ by the algorithms' definitions and since $\Psi$ is convex by assumption, it follows that
\begin{align} \label{eq:sublin-convergence-rate-step2,1}
\psi(M(\xi^k),x^k) 
&\le \min_{x\in X} \psi(M(\xi^k),x) + \ul{\delta}_k 
\le \int_X \psi(M(\xi^k),x) \d \xi^*(x) + \ul{\delta}_k \notag\\
&\le -\big(\Psi(M(\xi^k)) - \Psi^*\big) + \ul{\delta}_k.
\end{align}
Inserting this inequality~\eqref{eq:sublin-convergence-rate-step2,1} into the estimate~\eqref{eq:sublin-convergence-rate-step1} established in the first step, we immediately obtain the desired estimate~\eqref{eq:sublin-convergence-rate-step2}.
\smallskip

As a third step, we finally establish the claimed convergence rate estimate~\eqref{eq:sublin-convergence-rate-estimate}. In fact, we prove a slightly more general estimate, namely we prove that
\begin{align} \label{eq:sublin-convergence-rate-step3}
\Psi(M(\xi^k)) - \Psi^* \le \frac{2}{k+2} \Big( \frac{k_0+1}{2} \big(\Psi(M(\xi^{k_0})) - \Psi^*\big) + 2C + \ol{c} + \ul{c} \Big)
\end{align}
for all $k \in K_{k_0+1} := \{l \in K: l \ge k_0+1\}$ and arbitrary $k_0 \in K$ with $k_0 \ge 1$. In the special case $k_0 = 1$, this reduces to the claimed convergence rate estimate~\eqref{eq:sublin-convergence-rate-estimate}. %by virtue of~\eqref{eq:sublin-convergence-rate-C-def}. 
We prove~\eqref{eq:sublin-convergence-rate-step3} by induction over $K_{k_0+1}$ and adopt the shorthand notations
\begin{align} \label{eq:C'-def}
h_k := \Psi(M(\xi^k)) - \Psi^* 
\qquad \text{and} \qquad
C' := C + \frac{\ol{c} + \ul{c}}{2}.
\end{align}
In order to establish~\eqref{eq:sublin-convergence-rate-step3} for $k := k_0+1$, we set $\alpha := 2/(k_0+2) \in [0,2/3]$ in~\eqref{eq:sublin-convergence-rate-step2}. We thus obtain
\begin{align} \label{eq:sublin-convergence-rate-step3,induction-basis}
h_{k_0+1} 
&\le (1-\alpha) h_{k_0} + C \alpha^2 + \ol{\delta}_{k_0+1} + \ul{\delta}_{k_0} \alpha 
\le (1-\alpha) h_{k_0} + C' \alpha^2 \notag\\
&\le \frac{2}{k_0+3} \Big( \frac{k_0+1}{2} h_{k_0} + 2C'\Big).
\end{align}
%and this is precisely the claimed estimate~\eqref{eq:sublin-convergence-rate-step3} for $k := k_0+1$, so that the induction basis is finished. 
In the second inequality of~\eqref{eq:sublin-convergence-rate-step3,induction-basis} we used the assumptions~\eqref{eq:opt-tolerances-of-subproblems}, while in the last inequality of~\eqref{eq:sublin-convergence-rate-step3,induction-basis} we used that
%Indeed, the second inequality in~\eqref{eq:sublin-convergence-rate-step3,induction-basis} follows because
\begin{align}
1-\alpha &= \frac{k_0}{k_0+2} \le \frac{k_0+1}{k_0+3} = \frac{2}{k_0+3} \cdot \frac{k_0+1}{2}\\
\alpha^2 &\le \frac{3}{k_0+3} \cdot \frac{2}{k_0+2} \le \frac{2}{k_0+3} \cdot 2.
\end{align}
Since~\eqref{eq:sublin-convergence-rate-step3,induction-basis} is precisely the claimed estimate~\eqref{eq:sublin-convergence-rate-step3} for $k := k_0+1$, the induction basis is finished. 
We now move on to the induction step, that is, we assume that the claimed estimate~\eqref{eq:sublin-convergence-rate-step3} is true for some given $k \in K_{k_0+1}$ with $k+1 \in K_{k_0+1}$. Setting $\alpha := 2/(k+2) \in [0,2/3]$, we then conclude from~\eqref{eq:sublin-convergence-rate-step2} and the induction assumption that
\begin{align} \label{eq:sublin-convergence-rate-step3,induction-step}
h_{k+1} 
&\le (1-\alpha) h_k + C \alpha^2 + \ol{\delta}_{k+1} + \ul{\delta}_k \alpha
\le (1-\alpha) h_k + C' \alpha^2 \notag\\
&\le (1-\alpha) \frac{2}{k+2} \Big( \frac{k_0+1}{2} h_{k_0} + 2C' \Big) + C' \alpha^2 \notag\\
&= (1-\alpha) \frac{2}{k+2} \frac{k_0+1}{2} h_{k_0} + \Big((1-\alpha) \frac{4}{k+2} + \alpha^2 \Big) C'
\le \frac{2}{k+3} \Big( \frac{k_0+1}{2} h_{k_0} + 2C' \Big).
\end{align}
In the second inequality of~\eqref{eq:sublin-convergence-rate-step3,induction-step} we used the assumptions~\eqref{eq:opt-tolerances-of-subproblems}, while in the last inequality of~\eqref{eq:sublin-convergence-rate-step3,induction-step} we used that
\begin{align}
(1-\alpha) \frac{2}{k+2} 
&= \frac{k}{k+2} \cdot \frac{2}{k+2} \le \frac{k+1}{k+3} \cdot \frac{2}{k+2} \le \frac{2}{k+3} \\
(1-\alpha) \frac{4}{k+2} + \alpha^2 
&= \frac{k}{k+2} \cdot \frac{4}{k+2} + \frac{1}{k+2} \cdot \frac{4}{k+2}
= \frac{k+1}{k+2} \cdot \frac{4}{k+2} \le \frac{2}{k+3} \cdot 2.
\end{align}
Since~\eqref{eq:sublin-convergence-rate-step3,induction-step} is precisely the claimed estimate~\eqref{eq:sublin-convergence-rate-step3} with $k$ replaced by $k+1$, the induction step -- and thus the proof of~\eqref{eq:sublin-convergence-rate-step3} -- is finished. 
\end{proof}

\begin{cor} \label{cor:numbers-of-iterations-until-termination-under-sublinear-assumptions}
Suppose the assumptions of the above theorem are satisfied with a tolerance $\eps > \ul{\delta} := \sup_{k \ge 2} \ul{\delta}_k$.
%$\eps > \ul{\delta} := \sup_{k \in \N_2} \ul{\delta}_k$, where $\N_2 := \{n \in \N: n \ge 2\}$. 
Algorithms~\ref{algo:no-exchange} and~\ref{algo:exchange} then terminate at the very latest in iteration 
\begin{align} \label{eq:upper-bound-for-iteration-number}
k^* := \min \big\{ k \ge 2: \frac{2}{k+2} \big( \Psi(M(\xi^1)) - \ul{\Psi}^* + 2C'\big) \le \eps - \ul{\delta} \big\},
\end{align} 
where $\ul{\Psi}^*$ is an arbitrary lower bound for $\Psi^* := \min_{\xi \in \Xi(X)} \Psi(M(\xi))$ and $C'$ is defined as in~\eqref{eq:C'-def} above. 
\end{cor}

\begin{proof}
Assume, to the contrary, that Algorithm~\ref{algo:no-exchange} or~\ref{algo:exchange} runs at least until iteration $k^*+1$. It then follows that $k^* \in K$ with $k^* \ge 2$ and thus, by the sublinear convergence rate estimate~\eqref{eq:sublin-convergence-rate-estimate} and by the definition~\eqref{eq:upper-bound-for-iteration-number} of $k^*$, that
\begin{align}
\Psi(M(\xi^{k^*})) - \Psi^* \le \frac{2}{k^*+2} \big( \Psi(M(\xi^1)) - \ul{\Psi}^* + 2C'\big) \le \eps - \ul{\delta}.
\end{align}
Consequently, $\xi^{k^*}$ is an $(\eps - \ul{\delta})$-optimal design for $\Psi \circ M$. It follows (Theorem~\ref{thm:equivalence-thm}) that
\begin{align}
\psi(M(\xi^{k^*}),x^{k^*}) \ge \min_{x \in X} \psi(M(\xi^{k^*}),x) \ge -\eps + \ul{\delta} \ge -\eps + \ul{\delta}_{k^*}.
\end{align}
And this in turn means, by the termination criteria of Algorithm~\ref{algo:no-exchange} and~\ref{algo:exchange}, that the algorithm terminates at iteration $k^*$. Contradiction to our initial  assumption!
\end{proof}

\begin{ex} \label{ex:sublin-convergence-result,one-stage-design-criterion}
Suppose that $X$, $m$, $M$ are as in Condition~\ref{cond:X-and-M} and that the one-stage design criterion $\Psi: \Rpsd^{d \times d} \to \R \cup \{\infty\}$ is defined by
\begin{align}
\Psi(M) := \Psi_{p,Q}(M) \qquad (M \in \Rpsd^{d \times d})
\end{align}
for some $p \in [0,\infty)$ and an invertible matrix $Q \in \R^{d \times d}$, see~\eqref{eq:Psi_p,Q-def}. Suppose further that~\eqref{eq:finite-criterion-on-Xi(X^0)} is satisfied for some finite subset $X^0$ of $X$. If the sequence $(\xi^k)_{k\in K}$ is then generated by Algorithm~\ref{algo:no-exchange} or~\ref{algo:exchange} with tolarences $\eps, \ol{\delta}_k, \ul{\delta}_k \in [0,\infty)$ satisfying~\eqref{eq:opt-tolerances-of-subproblems}, then all assumptions of our sublinear convergence rate theorem are satisfied and therefore the error estimate~\eqref{eq:sublin-convergence-rate-estimate} holds true. 
(Indeed, by the definition~\eqref{eq:Psi_p,Q-def} and by the assumed invertibility of $Q$, it follows that 
\begin{align} \label{eq:example-sublin-convergence-result,one-stage-design-criterion,1}
\dom \Psi = \{M \in \Rpsd^{d \times d}: \ran(Q) \subset \ran(M) \text{ and } Q^\top M^- Q \in \Rpd^{d \times d}\} = \Rpd^{d \times d}.
\end{align}
And therefore, Conditions~\ref{cond:shared-assumption-1-algos-with-and-without-exchange} and~\ref{cond:shared-assumption-3-algos-with-and-without-exchange} are satisfied (Corollary~\ref{cor:standard-properties-for-Psi_p,Q}). Additionally, it follows by the final statements of Lemmas~\ref{lm:simple-algo-well-defined} and~\ref{lm:exchange-algo-well-defined} that $\Psi(M(\xi^k)) \le \Psi(M(\xi^0)) + c_0 < \infty$  for all $k \in K$. And from this, in turn, it follows by~\eqref{eq:example-sublin-convergence-result,one-stage-design-criterion,1} and~\eqref{eq:Psi_p,Q-def}, using analogous arguments as for Lemma~\ref{lm:upper-bound-on-Psi(M)-implies-lower-bound-on-M}, that 
\begin{align} \label{eq:example-sublin-convergence-result,one-stage-design-criterion,2}
M(\xi^k) \ge \mu_0 \qquad (k \in K)
\end{align}
for some $\mu_0 > 0$. So, by~\eqref{eq:example-sublin-convergence-result,one-stage-design-criterion,1} and~\eqref{eq:example-sublin-convergence-result,one-stage-design-criterion,2}, the assumption~\eqref{eq:M-subset-of-dom-Psi} is satisfied as well. And finally, the $L$-smoothness of $\Psi|_{\mathcal{M}}$ follows by Lemma~\ref{lm:L-smoothness-sufficient-condition} and Corollary~\ref{cor:standard-properties-for-Psi_p,Q}, using that $\mathcal{M}$ is a subset of the compact subset $\mathcal{C} := \{A \in \Rpsd^{d \times d}: \mu_0/3 \le A \le \max_{x \in X} \norm{m(x)} \}$ of $\Rpd^{d \times d}$.) 
\end{ex}

\begin{ex} \label{ex:sublin-convergence-result,two-stage-design-criterion}
Suppose that $X$, $m$, $M$ are as in Condition~\ref{cond:X-and-M} and that the two-stage design criterion $\Psi: \Rpsd^{d \times d} \to \R$ is defined by
\begin{align}
\Psi(M) := \Psi_{p,Q}(\alpha M^0 + (1-\alpha)M) \qquad (M \in \Rpsd^{d \times d})
\end{align}
for some $\alpha \in (0,1)$ and $M^0 \in \Rpd^{d \times d}$ and for some $p \in [0,\infty)$ and a matrix $Q \in \R^{d \times s}$ with $\rk Q = s \le d$, see~\eqref{eq:Psi_p,Q-def} and~\eqref{eq:two-stage-design-criterion}. If the sequence $(\xi^k)_{k\in K}$ is then generated by Algorithm~\ref{algo:no-exchange} or~\ref{algo:exchange} with tolarences $\eps, \ol{\delta}_k, \ul{\delta}_k \in [0,\infty)$ satisfying~\eqref{eq:opt-tolerances-of-subproblems}, then all assumptions of our sublinear convergence rate theorem are satisfied and therefore the error estimate~\eqref{eq:sublin-convergence-rate-estimate} holds true. 
(Indeed, by the definition~\eqref{eq:Psi_p,Q-def} and the assumed positive definiteness of $M^0$ as well as the assumed injectivity of $Q$, it follows that %by the relations~\eqref{eq:Psi-p,Q-alternative-representation} and~\eqref{eq:C_Q-alternative-representation} and the assumed positive definiteness of $M^0$, it follows that
\begin{align} \label{eq:example-sublin-convergence-result,two-stage-design-criterion,1}
\dom \Psi = \{M \in \Rpsd^{d \times d}: Q^\top (\alpha M^0 + (1-\alpha)M)^{-1} Q \in \Rpd^{d \times d}\} 
%\{M \in \Rpsd^{d \times d}: \Psi_p\big(Q^\top (\alpha M^0 + (1-\alpha)M)^{-1} Q)^{-1}\big) < \infty\}
= \Rpsd^{d \times d}.
\end{align}
And therefore, Conditions~\ref{cond:shared-assumption-1-algos-with-and-without-exchange} and~\ref{cond:shared-assumption-3-algos-with-and-without-exchange} are satisfied (Corollary~\ref{cor:standard-properties-for-Psi_p,Q}). In particular, the assumption~\eqref{eq:M-subset-of-dom-Psi} is satisfied trivially. And finally, the $L$-smoothness $\Psi|_{\mathcal{M}}$ follows by Lemma~\ref{lm:L-smoothness-sufficient-condition} and Corollary~\ref{cor:standard-properties-for-Psi_p,Q}, using that $\mathcal{C} := \{\alpha M^0 + (1-\alpha)M: M \in \mathcal{M} \}$ is a compact subset of $\Rpd^{d \times d}$.) 
\end{ex}

\subsection{A linear convergence rate result}

In this section, we further improve the sublinear convergence rate result from the previous section to a convergence rate result with an even linear convergence rate. %in the root and in the quotient sense (Section ?? of \cite{??})
In order to get this significantly better linear convergence rate, we have to strengthen the assumptions on the design criterion $\Psi$ and the design space $X$ a bit more, namely essentially add a $\mu$-strong convexity assumption on $\Psi$ and assume the finiteness of $X$. As we will see, the strong convexity assumptions is still satisfied for a large variety of practically relevant design criteria. %and therefore poses almost no restriction from practical point of view. 
Similarly, the finiteness assumption poses no serious restriction either because in practice set of possible %eligible 
experimental points is finite anyways. 
%Central idea behind this result: (recognize that) our algorithms are closely related to (are/can be seen as versions of) the fully corrective Frank-Wolfe algorithm and carry over convergence rate results for that algorithm from the literature (recognize that ... can be carried over)
\smallskip

We begin with a few preparations. Suppose that $\mathcal{A}$ is a finite subset of $\Rpsd^{d \times d}$ and that $M \in \Rpsd^{d \times d}$. %$M \in \mathcal{M} := \conv(\mathcal{A})$. 
A subset $\mathcal{S}$ of $\mathcal{A}$ is then called an active set for $M$ iff $M$ is a proper convex combination of the elements of $\mathcal{S}$, that is,
\begin{align}
M = \sum_{m \in \mathcal{S}} \lambda_m m \text{ for some } (\lambda_m)_{m\in\mathcal{S}} \in \Delta_{\mathcal{S}}^\circ, 
\end{align}
where $\Delta_{\mathcal{S}}^\circ := \{(\lambda_m)_{m\in\mathcal{S}} \in (0,1)^{|\mathcal{S}|}: \sum_{m \in \mathcal{S}} \lambda_m = 1\}$ is the interior of the $(|\mathcal{S}|-1)$-dimensional probability simplex $\Delta_{\mathcal{S}}$. 
We denote the set of subsets of $\mathcal{A}$ that are active for a given $M$ by %$\mathcal{S}_{\mathcal{A}}(M)$.
\begin{align}
\mathcal{S}_{\mathcal{A}}(M) := \{\mathcal{S} \subset \mathcal{A}: \mathcal{S} \text{ is an active set for } M \}. 
\end{align}  
Suppose in addition that $\Psi: \mathcal{M} \to \R$ is a differentiable map and that $M \in \mathcal{M}$ and $\mathcal{S} \subset \mathcal{A}$. We then need the sets
\begin{align} \label{eq:M_M-and-N_M,S-def}
\mathcal{M}_M := \mathcal{M}_{M,\mathcal{A}} := \argmin_{m\in \mathcal{A}} D\Psi(M)m
\qquad \text{and} \qquad
\mathcal{N}_{M,\mathcal{S}} := \argmax_{n\in \mathcal{S}} D\Psi(M)n
\end{align}
of minimizers and maximizers of $a \mapsto D\Psi(M)a$ over $\mathcal{A}$ or $\mathcal{S}$, respectively. Additionally, we also need the set
\begin{align}
\mathcal{N}_M := \argmin_{n \in \{n': n' \in \mathcal{N}_{M,\mathcal{S}} \text{ for some } \mathcal{S} \in \mathcal{S}_{\mathcal{A}}(M)\}} D\Psi(M)n
\end{align}
of those maximizers that yield the smallest maximum $\max_{n \in \mathcal{S}} D\Psi(M)n$ among all active sets $\mathcal{S}$ for $M$. In other words, a matrix $n_M$ belongs to $\mathcal{N}_M$ if and only if there exists an active set $\mathcal{S}^* \in \mathcal{S}_{\mathcal{A}}(M)$ for $M$ such that
\begin{align} \label{eq:N_M-characterization}
n_M \in \mathcal{N}_{M,\mathcal{S}^*}
\qquad \text{and} \qquad
D\Psi(M)n_M = \min_{\mathcal{S} \in \mathcal{S}_{\mathcal{A}}(M)} \max_{n \in \mathcal{S}} D\Psi(M)n.
\end{align}

\begin{lm} \label{lm:lin-convergence-rate-lm-1}
Suppose that $\mathcal{A}$ is a finite subset of $\Rpsd^{d \times d}$ and let $\mathcal{M} := \conv(\mathcal{A})$. Suppose further that $\Psi: \mathcal{M} \to \R$ is convex and differentiable. %Then the following assertions hold true.
\begin{itemize}
\item[(i)] If $M \in \mathcal{M}$ and $\mathcal{S} \in \mathcal{S}_{\mathcal{A}}(M)$, then 
\begin{align} \label{eq:lin-convergence-rate-lm-1(i)}
D\Psi(M)(n_{M,\mathcal{S}}-m_M) \ge D\Psi(M)(n_M-m_M) \ge \Psi(M)-\Psi^*
\end{align}
for arbitrary $m_M \in \mathcal{M}_M$, $n_M \in \mathcal{N}_M$ and $n_{M,\mathcal{S}} \in \mathcal{N}_{M,\mathcal{S}}$, where $\Psi^* := \min_{M \in \mathcal{M}} \Psi(M)$. Additionally, the quantity $D\Psi(M)(m_M-n_M)$ depends only on $M$ but not on the particular choices of $m_M \in \mathcal{M}_M$ and $n_M \in \mathcal{N}_M$.
\item[(ii)] If $D\Psi(M)(M'-M) < 0$ for some $M, M' \in \mathcal{M}$, then also $D\Psi(M)(m_M-n_M) < 0$ for all $m_M \in \mathcal{M}_M$ and $n_M \in \mathcal{N}_M$.
\end{itemize}
\end{lm}

\begin{proof}
(i) Suppose $M \in \mathcal{M}$ and $\mathcal{S} \in \mathcal{S}_{\mathcal{A}}(M)$ and let $m_M \in \mathcal{M}_M$, $n_M \in \mathcal{N}_M$ and $n_{M,\mathcal{S}} \in \mathcal{N}_{M,\mathcal{S}}$. Also, choose an $\mathcal{S}^* \in \mathcal{S}_{\mathcal{A}}(M)$ with~\eqref{eq:N_M-characterization} and an $M^* \in \mathcal{M}$ with $\Psi^* = \Psi(M^*)$ (which is possible, of course, by the continuity of $\Psi$ and the compactness of $\mathcal{M}$). 
It  follows by the definition~(\ref{eq:M_M-and-N_M,S-def}.a) of $\mathcal{M}_M$ and the convexity of $\Psi$ that
\begin{align} \label{eq:lin-convergence-rate-lm-1(i),1}
D\Psi(M)(m_M-M) \le D\Psi(M)(M'-M) \le \Psi(M')-\Psi(M) 
\qquad (M' \in \mathcal{M}),
\end{align}  
where for the first inequality we used that every $M' \in \mathcal{M} = \conv(\mathcal{A})$ can be represented as a convex combination $M' = \sum_{m\in\mathcal{A}} \lambda_m' m$ of elements of $\mathcal{A}$. 
In order to obtain the first inequality in~\eqref{eq:lin-convergence-rate-lm-1(i)}, we have only to insert the definition~(\ref{eq:M_M-and-N_M,S-def}.b) of $\mathcal{N}_{M,\mathcal{S}}$ into~(\ref{eq:N_M-characterization}.b).
In order to obtain the second inequality in~\eqref{eq:lin-convergence-rate-lm-1(i)}, we have only to notice that $M = \sum_{n\in\mathcal{S}^*} \lambda_n n$ for some $(\lambda_n) \in \Delta_{\mathcal{S}^*}^\circ$, so that
\begin{align} \label{eq:lin-convergence-rate-lm-1(i),2}
D\Psi(M)n_M \ge \sum_{n\in\mathcal{S}^*} \lambda_n D\Psi(M)n = D\Psi(M)M
\end{align}
by virtue of~(\ref{eq:N_M-characterization}.a). Setting now $M' := M^*$ in~\eqref{eq:lin-convergence-rate-lm-1(i),1} and combining the resulting inequality with~\eqref{eq:lin-convergence-rate-lm-1(i),2}, we obtain the second inequality in~\eqref{eq:lin-convergence-rate-lm-1(i)}. 
And finally, the claimed independence of $D\Psi(M)(m_M-n_M)$ is obvious by~(\ref{eq:M_M-and-N_M,S-def}.a) and~(\ref{eq:N_M-characterization}.b).
\smallskip

(ii) Suppose now $D\Psi(M)(M'-M) < 0$ for some $M, M' \in \mathcal{M}$ and let $m_M \in \mathcal{M}_M$ and $n_M \in \mathcal{N}_M$. We then obtain the claimed inequality
\begin{align}
D\Psi(M)(m_M-n_M) \le D\Psi(M)(m_M-M) \le D\Psi(M)(M'-M) < 0
\end{align} 
by~\eqref{eq:lin-convergence-rate-lm-1(i),2} and the first inequality in~\eqref{eq:lin-convergence-rate-lm-1(i),1}.
\end{proof}

We recall from~\cite{LaJa15} the notion %definition 
of the strong convexity constant $\convexconst$ and of the curvature constant $\curvconst$ of a differentiable map $\Psi: \mathcal{M} \to \R$ on a convex subset $\mathcal{M}$ of $\Rpsd^{d \times d}$. Specifically,
\begin{gather}
\convexconst := \inf_{\substack{M,M' \in \mathcal{M},\\ D\Psi(M)(M'-M) < 0}} \frac{2}{\alpha(M,M')^2} \big( \Psi(M') - \Psi(M) - D\Psi(M)(M'-M) \big)
\label{eq:strong-convexity-constant}\\
\curvconst := \sup_{\substack{M,M' \in \mathcal{M},\\ \alpha \in (0,1]}} \frac{2}{\alpha^2} \big( \Psi(M + \alpha(M'-M) - \Psi(M) - \alpha D\Psi(M)(M'-M) \big) 
\label{eq:curvature-constant}
\end{gather}
where $\alpha(M,M') := D\Psi(M)(M'-M)/D\Psi(M)(m_M-n_M)$ with arbitrary $m_M \in \mathcal{M}_M$ and $n_M \in \mathcal{N}_M$. In view of Lemma~\ref{lm:lin-convergence-rate-lm-1}, the quantity $\alpha(M,M')$ is a well-defined positive real number as soon as $D\Psi(M)(M'-M) < 0$. 
We also recall from~\cite{LaJa15, PeRo18} the notion of the pyramidal width $\pwidth(\mathcal{A})$ of a finite subset $\mathcal{A}$ of $\Rpsd^{d \times d}$. Specifically,
\begin{align}
\pwidth(\mathcal{A}) 
:=& \min_{\substack{\mathcal{F} \in \faces(\conv(\mathcal{A})),\\ \emptyset \ne \mathcal{F} \ne \conv(\mathcal{A})}} \dist(\mathcal{F}, \conv(\mathcal{F} \setminus \mathcal{A})) \notag\\ 
=& \min_{\substack{M,M' \in \conv(\mathcal{A}),\\ M \ne M'}} \pdirwidth(\mathcal{A},M'-M,M)
\end{align}
(Theorems~1 and 2 of~\cite{PeRo18}), where $\faces(\conv(\mathcal{A}))$ denotes the set of faces of the polytope $\conv(\mathcal{A})$ (Section~2.1 of~\cite{Ziegler}) and where the pyramidal directional width of $\mathcal{A}$ in the direction $J \ne 0$ with base point $M$ is defined by
\begin{align}
\pdirwidth(\mathcal{A},J,M) := \min_{\mathcal{S} \in \mathcal{S}_{\mathcal{A}}(M)} \max_{m \in \mathcal{A}, n \in \mathcal{S}} \scprd{J/|J|,m-n}.
\end{align}
See the first formula after Example~2 and the last formula before Theorem~2 in~\cite{PeRo18}. In the above formula, the scalar product $\scprd{\cdot, \cdot}$ is the one induced by the Frobenius norm $|\cdot|$.  

\begin{lm} \label{lm:lin-convergence-rate-lm-2}
Suppose that $\mathcal{A}$ is a finite subset of $\Rpsd^{d \times d}$ with $|\mathcal{A}| \ge 2$ and let $\mathcal{M} := \conv(\mathcal{A})$. Suppose further that $\Psi: \mathcal{M} \to \R$ is differentiable and $\mu$-strongly convex and $L$-smooth w.r.t.~$|\cdot|$ for some $\mu, L \in (0,\infty)$. Then
\begin{align}
0 < \mu \cdot (\pwidth(\mathcal{A}))^2 \le \convexconst \le \curvconst \le L \cdot (\diam(\mathcal{M}))^2, 
\end{align}
where $\diam(\mathcal{M}) := \sup\{|M-N|: M,N \in \mathcal{M}\}$. 
\end{lm}

\begin{thm} \label{thm:lin-convergence-rate}
Suppose that Conditions~\ref{cond:shared-assumption-1-algos-with-and-without-exchange} and~\ref{cond:shared-assumption-3-algos-with-and-without-exchange} are satisfied. Suppose further that the set $\mathcal{A} := \{m(x):x \in X\}$ of one-point information matrices is finite with $|\mathcal{A}| \ge 2$, that 
\begin{align} \label{eq:M-subset-of-dom-Psi,linear-convergence-rate}
\mathcal{M} := M(\Xi(X)) \subset \dom\Psi,
\end{align}
and that $\Psi|_{\mathcal{M}}$ is $\mu$-strongly convex and $L$-smooth w.r.t.~$|\cdot|$ for some $\mu, L \in (0,\infty)$. Suppose finally that $X^0$ is a finite subset of $X$ satisfying~\eqref{eq:finite-criterion-on-Xi(X^0)} and that $(\xi^k)_{k\in K}$ is generated by Algorithm~\ref{algo:no-exchange} or Algorithm~\ref{algo:exchange} with optimality tolerances $\eps \in [0,\infty)$ and $\ol{\delta}_k = 0 = \ul{\delta}_k$. We then have the following linear convergence rate estimate:
\begin{align} \label{eq:lin-convergence-rate-estimate}
\Psi(M(\xi^k)) - \Psi^* \le r^k \cdot (\Psi(M(\xi^0)) - \Psi^*) 
\end{align}
for all $k \in K$, where $\Psi^* := \min_{\xi \in \Xi(X)} \Psi(M(\xi))$ and $r := 1 - \min\{1/2, \convexconst/\curvconst\}$ and
\begin{align} \label{eq:lin-convergence-rate}
1/2 \le r \le 1 - \min\{ 1/2, (\mu/L) \cdot (\pwidth(\mathcal{A})/\diam(\mathcal{M}))^2 \} < 1.
\end{align}
\end{thm}

\begin{proof}
In the entire proof, we use the abbreviations
\begin{align} \label{eq:lin-convergence-rate,M_k,m_k,n_k}
M_k := M(\xi^k) \qquad \text{and} \qquad m_k := m(x^k) \qquad \text{and} \qquad n_k := n_{M_k,\mathcal{S}_k} \in \mathcal{N}_{M_k,\mathcal{S}_k},
\end{align}
where $n_{M_k,\mathcal{S}_k}$ is an arbitrary element of $\mathcal{N}_{M_k,\mathcal{S}_k} = \argmax_{n\in\mathcal{S}_k} D\Psi(M_k)n$ and
\begin{align} \label{eq:lin-convergence-rate,S_k}
\mathcal{S}_k := \{m(x): x \in \supp \xi^k \}.
\end{align}
Additionally, we use the abbreviations
\begin{align} \label{eq:lin-convergence-rate,g_k,h_k}
g_k := D\Psi(M_k)(n_k-m_k) \qquad \text{and} \qquad h_k := \Psi(M_k) - \Psi(M^*),
\end{align}
where $M^* := M(\xi^*)$ and $\xi^*$ is an arbitrary optimal design for $\Psi \circ M$ (Theorem~\ref{thm:ex-of-optimal-designs}). %In view of~(\ref{eq:lin-convergence-rate,M_k,m_k,n_k}.c) and~\eqref{eq:lin-convergence-rate,S_k}, it is clear that
It is clear by the definition of $n_k$ and $\mathcal{S}_k$ that
\begin{align} \label{eq:lin-convergence-rate,y_k}
n_k = m(y^k) \text{ for some } y^k \in \supp\xi^k.
\end{align}
Also, by the proof of Lemma~\ref{lm:designs-with-finite-supp}, 
\begin{align} \label{eq:lin-convergence-rate,M=conv(A)}
\mathcal{M} = M(\Xi(X)) = \conv(\{m(x): x\in X\}) = \conv(\mathcal{A}).
\end{align}
With these preliminaries at hand, we can now enter the core of the proof. We proceed in four steps, the third and fourth step following the lines of proof from~\cite{LaJa15} (Theorem~8).
\smallskip

As a first step, we show that
\begin{align} \label{eq:lin-convergence-rate-step1}
D\Psi(M_k)(m_k-M_k) = -g_k
\qquad \text{and} \qquad
g_k \ge D\Psi(M_k)(n_{M_k}-m_{M_k}) \ge h_k
\end{align}
for all $k \in K$ and arbitrary $m_{M_k} \in \mathcal{M}_{M_k}$ and $n_{M_k} \in \mathcal{N}_{M_k}$.
In order to prove~(\ref{eq:lin-convergence-rate-step1}.a), choose and fix $k \in K$. Since $\xi^k$ is an optimal design for $\Psi \circ M$ on $X^k$ by the algorithms' definitions with $\ol{\delta}_k = 0$, it follows that 
\begin{align} \label{eq:lin-convergence-rate-step1,1}
D\Psi(M_k)(n_k-M_k) = \psi(M(\xi^k),y^k) = 0,
\end{align}
where for the first equality we used~\eqref{eq:lin-convergence-rate,y_k} together with Lemma~\ref{lm:sensi-fct-in-terms-of-derivative} and for the second equality we used the last statement of Theorem~\ref{thm:equivalence-thm} with $X$ replaced by $X^k$. In view of~(\ref{eq:lin-convergence-rate,g_k,h_k}.a) and~\eqref{eq:lin-convergence-rate-step1,1}, the asserted equality~(\ref{eq:lin-convergence-rate-step1}.a) is now clear. 
In order to prove also~(\ref{eq:lin-convergence-rate-step1}.b), choose and fix $k \in K$ and $m_{M_k} \in \mathcal{M}_{M_k}$ and $n_{M_k} \in \mathcal{N}_{M_k}$. Since $x^k$ is a minimizer of $X \ni x \mapsto \psi(M_k,x) = D\Psi(M_k)(m(x)-M_k)$ by the algorithms' definitions with $\ul{\delta}_k = 0$, it follows that
\begin{align}  \label{eq:lin-convergence-rate-step1,2}
m_k = m(x^k) \in \argmin_{m\in\mathcal{A}} D\Psi(M_k)m = \mathcal{M}_{M_k}.
\end{align}
Since $\mathcal{S}_k \subset \mathcal{A}$ and $M_k = M(\xi^k) = \int_X m(x) \d\xi^k(x) = \sum_{x \in \supp\xi^k} \xi^k(\{x\}) m(x)$, it further follows that $\mathcal{S}_k$ is an active set for $M_k$, in short, 
\begin{align} \label{eq:lin-convergence-rate-step1,3}
\mathcal{S}_k \in \mathcal{S}_{\mathcal{A}}(M_k).
\end{align}
In view of~(\ref{eq:lin-convergence-rate,M_k,m_k,n_k}.c), (\ref{eq:lin-convergence-rate,g_k,h_k}.a), \eqref{eq:lin-convergence-rate-step1,2} and~\eqref{eq:lin-convergence-rate-step1,3}, the asserted inequality chain~(\ref{eq:lin-convergence-rate-step1}.b) is an immediate consequence of Lemma~\ref{lm:lin-convergence-rate-lm-1}(i). 
\smallskip

As a second step, we show that 
\begin{align} \label{eq:lin-convergence-rate-step2}
h_k > 0 
\qquad \text{and} \qquad
D\Psi(M_k)(M^*-M_k) < 0
\end{align}
for every non-terminal iteration index $k \in K$. So, let $k \in K$ be non-terminal. It then follows by the algorithms' termination rule that $\psi(M(\xi^k),x^k) < -\eps \le 0$.
%\begin{align*}
%\psi(M(\xi^k),x^k) < -\eps \le 0.
%\end{align*} 
Consequently, $\xi^k$ is not an optimal design for $\Psi \circ M$ %on $X$ 
by Theorem~\ref{thm:equivalence-thm} and therefore
\begin{align} \label{eq:lin-convergence-rate-step2,1}
h_k = \Psi(M(\xi^k)) - \min_{\xi \in \Xi(X)} \Psi(M(\xi)) > 0,
\end{align}
that is, (\ref{eq:lin-convergence-rate-step2}.a) is proved. 
Additionally, by the convexity of $\Psi$ and by~\eqref{eq:lin-convergence-rate-step2,1}, 
\begin{align} \label{eq:lin-convergence-rate-step2,2}
D\Psi(M_k)(M^*-M_k) 
&= \lim_{t\searrow 0} \frac{\Psi(M_k+t(M^*-M_k))-\Psi(M_k)}{t} \notag\\
&\le \Psi(M^*)-\Psi(M_k) = -h_k < 0,
\end{align}
that is, (\ref{eq:lin-convergence-rate-step2}.b) is proved as well.
\smallskip

As a third step, we show that for every non-terminal iteration index $k \in K$ the estimate
\begin{align} \label{eq:lin-convergence-rate-step3}
h_k \le \frac{g_k^2}{2 \convexconst}
\end{align}
holds true. 
So, let $k \in K$ be non-terminal. It then follows by Lemma~\ref{lm:lin-convergence-rate-lm-1}(ii) in conjunction with~(\ref{eq:lin-convergence-rate-step2}.b) that
\begin{align*} 
\alpha_k := \alpha(M_k,M^*) = \frac{D\Psi(M_k)(M^*-M_k)}{D\Psi(M_k)(m_{M_k}-n_{M_k})} > 0.
\end{align*}
So, by the definition~\eqref{eq:strong-convexity-constant} of the strong convexity constant $\convexconst$ combined with~(\ref{eq:lin-convergence-rate-step2}.b) and~(\ref{eq:lin-convergence-rate-step1}.b), it follows that
\begin{align*}
\frac{\alpha_k^2}{2} \convexconst 
&\le \Psi(M^*)-\Psi(M_k) - D\Psi(M_k)(M^*-M_k)
= -h_k + \alpha_k D\Psi(M_k)(n_{M_k}-m_{M_k}) \notag \\
&\le -h_k + \alpha_k g_k.
\end{align*}
And therefore we obtain
\begin{align}
h_k \le \alpha_k g_k - \alpha_k^2 \frac{\convexconst}{2} 
= \bigg(2\alpha_k \cdot \frac{g_k}{\convexconst}\bigg) \frac{\convexconst}{2}
\le \frac{g_k^2}{\convexconst^2} \frac{\convexconst}{2}
= \frac{g_k^2}{2 \convexconst},
\end{align}
as desired.
\smallskip

As a fourth step, we show that for every non-terminal iteration index $k \in K$ the estimate
\begin{align} \label{eq:lin-convergence-rate-step4}
h_{k+1} \le r \cdot h_k
\end{align}
holds true with $r := 1 - \min\{1/2, \convexconst/\curvconst\}$, and then conclude the proof of the theorem. 
So, let $k \in K$ be non-terminal. It then follows by the updating rule both of Algorithm~\ref{algo:no-exchange} and of Algorithm~\ref{algo:exchange} that
\begin{align}
\supp\big( (1-\alpha) \xi^k + \alpha \delta_{x^k} \big) \subset \supp \xi^k \cup \{x^k\} \subset X^{k+1}
\qquad (\alpha \in [0,1])
\end{align}
and therefore $(1-\alpha) \xi^k + \alpha \delta_{x^k} \in \Xi(X^{k+1})$ for all $\alpha \in [0,1]$. Since $\xi^{k+1}$ is an optimal design for $\Psi \circ M$ on $X^{k+1}$ by the algorithms' definitions with $\ol{\delta}_k = 0$, it follows that
\begin{align*}
\Psi(M_{k+1}) = \Psi(M(\xi^{k+1})) 
%= \min_{\xi \in \Xi(X^{k+1})} \Psi(M(\xi)) 
&\le \Psi\big( (1-\alpha)M(\xi^k) + \alpha m(x^k) \big) \notag\\
&= \Psi(M_k + \alpha(m_k-M_k)) 
\qquad (\alpha \in [0,1]).  
\end{align*}
So, by the definition~\eqref{eq:curvature-constant} of the curvature constant $\curvconst$ combined with~(\ref{eq:lin-convergence-rate-step1}.a), it follows that
\begin{align} \label{eq:lin-convergence-rate-step4,1}
\Psi(M_{k+1}) 
&\le \Psi(M_k) + \alpha D\Psi(M_k)(m_k-M_k) + \frac{\alpha^2}{2} \curvconst \notag \\
&= \Psi(M_k) - g_k \alpha + \curvconst \frac{\alpha^2}{2} 
\qquad (\alpha \in [0,1]).
\end{align}
And therefore we obtain
\begin{align} \label{eq:lin-convergence-rate-step4,2}
h_{k+1} 
\le \min_{\alpha \in [0,1]} \Big( h_k - g_k \alpha + \curvconst \frac{\alpha^2}{2} \Big)
= h_k - g_k \alpha_k^* + \curvconst \frac{(\alpha_k^*)^2}{2},
\end{align}
where $\alpha_k^*$ denotes a minimizer of the right-hand side of~\eqref{eq:lin-convergence-rate-step4,1} as a function of $\alpha \in [0,1]$. Since $g_k \ge h_k > 0$ by~(\ref{eq:lin-convergence-rate-step1}.b) and~(\ref{eq:lin-convergence-rate-step2}.a) and since $\curvconst \ge \convexconst > 0$ by Lemma~\ref{lm:lin-convergence-rate-lm-2}, $\alpha_k^*$ is the only minimizer and it is given by
\begin{align} \label{eq:lin-convergence-rate-step4,3}
\alpha_k^* = \min\Big\{\frac{g_k}{\curvconst}, 1 \Big\}.
\end{align} 
We now consider the two cases in~\eqref{eq:lin-convergence-rate-step4,3} separately. 
In case $g_k < \curvconst$, we conclude from~\eqref{eq:lin-convergence-rate-step4,2} with~\eqref{eq:lin-convergence-rate-step4,3} and~\eqref{eq:lin-convergence-rate-step3} that
\begin{align} \label{eq:lin-convergence-rate-step4,4}
h_{k+1} 
\le h_k - \frac{g_k^2}{\curvconst} 
= h_k - \frac{\convexconst}{\curvconst} \frac{g_k^2}{2 \convexconst} 
\le \big(1 - \convexconst / \curvconst\big) h_k.
\end{align}
In case $g_k \ge \curvconst$, we conclude from~\eqref{eq:lin-convergence-rate-step4,2} with~\eqref{eq:lin-convergence-rate-step4,3} and~(\ref{eq:lin-convergence-rate-step1}.b) that
\begin{align} \label{eq:lin-convergence-rate-step4,5}
h_{k+1}
\le h_k - g_k + \frac{\curvconst}{2}
\le h_k - \frac{g_k}{2}
\le (1-1/2) h_k.
\end{align}
So, combining~\eqref{eq:lin-convergence-rate-step4,4} and~\eqref{eq:lin-convergence-rate-step4,4}, we immediately obtain the desired estimate~\eqref{eq:lin-convergence-rate-step4}. And from~\eqref{eq:lin-convergence-rate-step4}, in turn, the linear convergence rate estimate~\eqref{eq:lin-convergence-rate-estimate} immediately follows by induction over $K$, while the relations~\eqref{eq:lin-convergence-rate} immediately follow by Lemma~\ref{lm:lin-convergence-rate-lm-2}. 
\end{proof}

\begin{cor} \label{cor:numbers-of-iterations-until-termination-under-linear-assumptions}
Suppose the assumptions of the above theorem are satisfied with a strictly positive tolerance $\eps > 0$. Algorithms~\ref{algo:no-exchange} and~\ref{algo:exchange} then terminate at the very latest in iteration 
\begin{align}
k^* := \min \big\{ k \in \N_0: r^k \big(\Psi(M(\xi^0)) - \ul{\Psi}^*\big) \le \eps \big\},
\end{align}
where $\ul{\Psi}^*$ is an arbitrary lower bound for $\Psi^* := \min_{\xi \in \Xi(X)} \Psi(M(\xi))$. 
\end{cor}

\begin{proof}
Completely analogous to the proof of Corollary~\ref{cor:numbers-of-iterations-until-termination-under-sublinear-assumptions}.
\end{proof}

\begin{ex}
Suppose that $X$, $m$, $M$ are as in Condition~\ref{cond:X-and-M} and that, in addition, the design space $X$ is finite but such that $\mathcal{A} := \{m(x): x \in X\}$ is not a singleton. %non-trivial (meaning that it consists of at least two elements). 
Suppose further that the two-stage design criterion $\Psi: \Rpsd^{d \times d} \to \R$ is defined by
\begin{align}
\Psi(M) := \tilde{\Psi}_p(\alpha M^0 + (1-\alpha)M) \qquad (M \in \Rpsd^{d \times d})
\end{align}
for some $\alpha \in (0,1)$ and $M^0 \in \Rpd^{d \times d}$ and for some $p \in \N_0$, see~\eqref{eq:tilde-Psi_p-def} and~\eqref{eq:two-stage-design-criterion}. In particular, we recall that this covers the D-criterion $\Psi_0 = \tilde{\Psi}_0$ and the A-criterion $\Psi_1 = \tilde{\Psi}_1$. If the sequence $(\xi^k)_{k\in K}$ is then generated by Algorithm~\ref{algo:no-exchange} or~\ref{algo:exchange} with tolarences $\eps \in [0,\infty)$ and $\ol{\delta}_k = 0 = \ul{\delta}_k $, then all assumptions of our linear convergence rate theorem are satisfied and therefore the error estimate~\eqref{eq:lin-convergence-rate-estimate} holds true. 
(Indeed, the assumption~\eqref{eq:M-subset-of-dom-Psi,linear-convergence-rate} is trivially satisfied because $\dom \Psi = \Rpsd^{d \times d}$ by the assumed positive definiteness of $M^0$. Additionally, Conditions~\ref{cond:shared-assumption-1-algos-with-and-without-exchange} and~\ref{cond:shared-assumption-3-algos-with-and-without-exchange} and the $L$-smoothness of $\Psi|_{\mathcal{M}}$ follow in a completely analogous way as in Example~\ref{ex:sublin-convergence-result,two-stage-design-criterion}. And finally, the $\mu$-strong convexity of $\Psi|_{\mathcal{M}}$ follows by Corollary~\ref{cor:strongly-convex-design-criteria}, using that $\mathcal{C} := \{\alpha M^0 + (1-\alpha)M: M \in \mathcal{M} \}$ is a compact subset of $\Rpd^{d \times d}$.) 
\end{ex}

\section{Application examples}

In this section, we apply the proposed adaptive discretization algorithms to two experimental design problems from chemical engineering. In both application examples, the underlying prediction model $f$ is implicitly defined, namely by nonlinear algebraic equations in the first example and by nonlinear differential equations in the second example. We have implemented all our algorithms in Python and we apply them both with the A- and with the D-criterion $\Psi$. As in~\cite{YaBiTa13} and in all other practically feasible optimal experimental design implementations we are aware of, %computations, 
we choose the design space $X$ to be a large but finite set. As a consequence, we can -- and do -- solve the occurring violator problems~\eqref{eq:Viol(xi^k)-intro} by enumeration. In order to solve the convex discretized design problems~\eqref{eq:OED(X^k)-intro}, in turn, we use the sequential quadratic programming solver~\cite{Sc11, Sc14}.

\subsection{An application with a stationary model}

%What are the system equations for the considered process? What are the state variables, the input quantities, and the model parameters?
In our first application example, we consider the chlorination of benzene in a continuously fed and stirred tank reactor. As explained in~\cite{StBa15}, this process can be described by algebraic system equations of the form
\begin{align} \label{eq:stationary-system-equation}
g(s,x,\theta) = 0,
\end{align}
where the map $g: \mathcal{S} \times \X \times \Theta \to \R^4$ is defined as 
\begin{align}
g_1(s,x,\theta) &:= F_f - s_1 s_4 - \theta_1 \frac{s_1 V}{d(s)}\\
g_2(s,x,\theta) &:= - s_2 s_4 + \theta_1 \frac{s_1 V}{d(s)} - \theta_2 \frac{s_2 V}{d(s)}\\
g_3(s,x,\theta) &:= - s_3 s_4 + \theta_2 \frac{s_2 V}{d(s)}\\
g_4(s,x,\theta) &:= s_1 + s_2 + s_3 - 1
\end{align}
and where $d(s) := v_A s_1 + v_B s_2 + v_C s_3$ with $v_A, v_B, v_C$ being the molar volumes of $A, B, C$ from~\cite{StBa15}. In these equations, $s = (s_1, \dots, s_4) := (l_A, l_B, l_C, F_o) \in \mathcal{S} := [0,1]^3 \times [60,70]$ describes the state of the system through the molar fractions $l_A, l_B, l_C$ of the components $A, B, C$ in the tank and through the total molar outlet stream $F_o$ from the tank. Additionally, the relevant input quantities 
\begin{align}
x := (F_f, V) \in \X := [60, 70] \times [10, 20]
\end{align}
of the system are given by the total molar feed stream $F_f$ into the tank and by the mixture's volume $V$ in the tank. And finally, the model parameters $\theta := (k_1, k_2) \in \Theta$ are the reaction rate constants for the two kinetic reactions taking place in the tank reactor.
\smallskip

%Which output quantities are we interested in? What is the considered prediction model?
What we are interested in here is to predict the molar composition in the tank reactor as a function of the aforementioned input quantities $x$ and model parameters $\theta$. So, the output quantity we are interested in here is $y = h(s) := (s_1, s_2, s_3) \in \Y := \R^3$ and the prediction model $f: \X \times \Theta \to \Y$ of interest is given by
\begin{align} \label{eq:stationary-model}
f(x,\theta) := h(s(x,\theta)),
\end{align}
where $s(x,\theta)$ for every given $(x,\theta) \in \X \times \Theta$ is the solution to the system  equation~\eqref{eq:stationary-system-equation}. 
%
%What is the considered design space and what is the chosen reference parameter value?
As the design space $X$, we choose the finite subset 
\begin{align} \label{eq:design-space-stationary-example}
X := \{60, 60.1, 60.2, \dots, 70\} \times \{10, 10.1, 10.2, \dots, 20\}
\end{align}
of %the input space 
$\X$ consisting of $10\, 201$ design points (candidate experiments). And as the reference parameter value $\ol{\theta}$, we choose the value $\ol{\theta} := (0.4, 0.0555)$ from~\cite{StBa15}.  
\smallskip

After having specified the prediction model $f$ along with the design space $X$ and the reference parameter $\ol{\theta}$, we can now turn to solving the corresponding locally optimal design problem~\eqref{eq:OED(X)-intro}. In order to do so, we use our strict algorithms with and without exchange. Specifically, we apply Algorithm~\ref{algo:no-exchange} and Algorithm~\ref{algo:exchange} with the initial discretization
\begin{align}
X^0 := \{(60, 10),\, (60, 20),\, (70, 10),\, (70, 10)\}
\end{align}
consisting of the $4$ corner points of the design space $X$, and with optimality tolerances $\eps := 10^{-4}$ and $\ol{\delta}_k = 0 = \ul{\delta}_k$. 
\smallskip

With these settings, our Algorithms~\ref{algo:no-exchange} and~\ref{algo:exchange} terminate already after $1$ iteration with an $\eps$-optimal design $\xi^*$, namely the point measure
\begin{align} \label{eq:optimal-design-stationary-example}
\xi^* := \delta_{x^*} \qquad \text{with} \qquad x^* := (60, 20).
\end{align}
As is to be expected, the objective value $\Psi(\xi^*)$ at the thus computed $\eps$-optimal design is significantly better than the objective values $\Psi(\xi^{\fact})$ and $\Psi(\xi^{\rand})$ at the factorial design and the random design
\begin{align} \label{eq:factorial-and-random-design-stationary-example}
\xi^{\fact} = \frac{1}{4} \sum_{x \in X^{\fact}} \delta_x
\qquad \text{and} \qquad
\xi^{\rand} = \frac{1}{4} \sum_{x \in X^{\rand}} \delta_x
\end{align}
with $X^{\fact} := X^0$ and $X^{\rand}:= \{ (62.5, 12.7),\, (66.2, 15.1),\, (68.9, 18.5),\, (63.2, 16.4)\}$, respectively. See Table~\ref{tab:stationary-example}. 
Additionally, our algorithms need substantially fewer iterations and computation time than the vertex-direction algorithm from~\cite{Fe, Wy70}. And moreover, they achieve a better objective function value at termination and terminate at a design with fewer support points. See Table~\ref{tab:stationary-example-algorithm-comparison}.

\begin{table}
\begin{center}	
	\caption{Values of the A-criterion and the D-criterion at the different designs $\xi^*$, $\xi^\fact$, $\xi^\rand$ from~\eqref{eq:optimal-design-stationary-example} and~\eqref{eq:factorial-and-random-design-stationary-example}.} 
	\label{tab:stationary-example}
	\begin{tabular}{l c c}
		\hline 
		& A-criterion $\Psi_1$ & D-criterion $\Psi_0$ \\
		\hline
		$\Psi(\xi^*)$ & $2.2630$ & $-0.6873$ \\
		$\Psi(\xi^{\fact})$ & $2.4974$ & $-0.0381$ \\
		$\Psi(\xi^{\rand})$ & $2.5178$ & $-0.0085$ \\
		\hline
	\end{tabular}
\end{center}
\end{table}

\begin{table}
\begin{center}
	\caption{Comparison of our algorithms with exchange and without exchange with the vertex-direction algorithm for the stationary example with the D-criterion $\Psi = \Psi_0$.}
	\label{tab:stationary-example-algorithm-comparison}
	\begin{tabular}{c c c c}
		\hline 
		& With exchange & Without exchange & Vertex-direction \\
		\hline
		Criterion value $\Psi(\xi^*)$ & $-0.6873$ & $-0.6873$ & $-0.6863$ \\
		Iterations  & $1$ & $1$ & $450$\\
		Approximation error bound  & $0.0$ & $0.0$ & $0.0010$ \\
		Support points  & $1$ & $1$ & $4$ \\
		Computation time (in s) & $0.5397$ & $0.5070$ & $36.39$ \\
		\hline
	\end{tabular}
\end{center}
\end{table}

%We obtain the following formulas for the derivatives $D_s g(s,x,\theta)$ by straightforward calculations:
%\begin{align}
%\partial_{s_1} g_1(s,x,\theta) &= -F_2 - \theta_1 \frac{s_4 d(s) - s_1 s_4 v_A}{d(s)^2}\\
%\partial_{s_2} g_1(s,x,\theta) &= \theta_1 \frac{s_1 s_4 v_B}{d(s)^2}\\
%\partial_{s_3} g_1(s,x,\theta) &= \theta_1 \frac{s_1 s_4 v_C}{d(s)^2}\\
%\partial_{s_4} g_1(s,x,\theta) &= -\theta_1 \frac{s_1}{d(s)}\\
%%
%\partial_{s_1} g_2(s,x,\theta) &= \theta_1 \frac{s_4 d(s) - s_1 s_4 v_A}{d(s)^2} + \theta_2 \frac{s_2 s_4 v_A}{d(s)^2}\\
%\partial_{s_2} g_2(s,x,\theta) &= -F_2 - \theta_1 \frac{s_1 s_4 v_B}{d(s)^2} - \theta_2 \frac{s_4 d(s) - s_2 s_4 v_B}{d(s)^2}\\
%\partial_{s_3} g_2(s,x,\theta) &= -\theta_1 \frac{s_1 s_4 v_C}{d(s)^2} + \theta_2 \frac{s_2 s_4 v_C}{d(s)^2}\\
%\partial_{s_4} g_2(s,x,\theta) &= \theta_1 \frac{s_1}{d(s)} - \theta_2 \frac{s_2}{d(s)} \\
%%
%\partial_{s_1} g_3(s,x,\theta) &= -\theta_2 \frac{s_2 s_4 v_A}{d(s)^2}\\
%\partial_{s_2} g_3(s,x,\theta) &= \theta_2 \frac{s_4 d(s) - s_2 s_4 v_B}{d(s)^2}\\
%\partial_{s_3} g_3(s,x,\theta) &= -F_2 - \theta_2 \frac{s_2 s_4 v_C}{d(s)^2}\\
%\partial_{s_4} g_3(s,x,\theta) &= \theta_2 \frac{s_2}{d(s)}\\
%%
%D_s g_4(s,x,\theta) &= (1, 1, 1, 0) 
%\end{align}
%Similarly, we obtain the following formulas for the derivatives $D_\theta g(s,x,\theta)$ by straightforward calculations:
%\begin{align}
%\partial_{\theta_1} g(s,x,\theta) &= (-1, 1, 0, 0)^\top \frac{s_1 s_4}{d(s)}\\
%\partial_{\theta_2} g(s,x,\theta) &= (0, -1, 1, 0)^\top \frac{s_2 s_4}{d(s)}
%\end{align}

\subsection{An application with a dynamic model}

%What are the system equations for the considered process? What are the state variables, the input quantities, and the model parameters?
In our second application example, we consider the Williams-Otto process which connects a continuously stirred tank reactor with a decanter and a distillation column. As explained in~\cite{ScTe20}, this process can be described by differential system equations of the form
\begin{align}\label{eq:dynamic-system-equation}
\dot{s}(t) = g(s(t),x,\theta) \qquad \text{and} \qquad s(0) = s_0,
\end{align}
where $g: \mathcal{S} \times X \times \Theta \to \R^6$ is defined as 
\begin{align}
g_1(s,x,\theta) &:= F_{fA} + \big((1-\eta)\mu - \mu\big) \frac{s_1}{d(s)} - k_{10}(T,\theta)\frac{s_1 s_2}{d(s)}\\
g_2(s,x,\theta) &:= F_{fB} + \big((1-\eta)\mu - \mu\big) \frac{s_2}{d(s)} - k_{10}(T,\theta)\frac{s_1 s_2}{d(s)} - k_{20}(T,\theta)\frac{s_2 s_3}{d(s)}\\
g_3(s,x,\theta) &:= \big((1-\eta)\mu - \mu\big) \frac{s_3}{d(s)} + 2 k_{10}(T,\theta)\frac{s_1 s_2}{d(s)} - 2 k_{20}(T,\theta)\frac{s_2 s_3}{d(s)} - k_{30}(T,\theta) \frac{s_3 s_5}{d(s)}\\
g_4(s,x,\theta) &:= \big((1-\eta)\mu - \mu\big) \frac{s_4}{d(s)} + 2 k_{20}(T,\theta)\frac{s_2 s_3}{d(s)}\\
g_5(s,x,\theta) &:= 0.1 (1-\eta)\mu \frac{s_4}{d(s)} - \mu \frac{s_5}{d(s)} + k_{20}(T,\theta)\frac{s_2 s_3}{d(s)} - 0.5 k_{30}(T,\theta) \frac{s_3 s_5}{d(s)}\\
g_6(s,x,\theta) &:= - \mu \frac{s_6}{d(s)} + 1.5 k_{30}(T,\theta) \frac{s_3 s_5}{d(s)}
\end{align}
and where $d(s) := s_1 + \dotsb + s_6$ and $k_{i0}(T, \theta) := a_i \exp(-b_i/T)$ for $i \in \{1,2,3\}$. In these equations, $s = (s_1,\dots, s_6) := (m_A, m_B, m_C, m_P, m_E, m_G) \in \mathcal{S} := [0,\infty)^6$ describes the state of the system through the masses $m_A, \dots, m_G$ of the components $A, \dots, G$ in the tank. We set the initial masses to
\begin{align}
s_0 = (m^0_A, \ldots, m^0_G) = (10, 1, 0, 0, 0, 0).
\end{align} 
and the values $F_{fA} = 10,\, \mu=129.5$ and $\eta=0.2$.  Additionally, the relevant input quantities 
\begin{align}
x := (F_{fB}, T, t_{meas}) \in \X := [20, 23] \times [650, 660] \times [1, 20]
\end{align}
of the system are given by the mass feed stream $F_{fB}$ of component $B$ into the tank, by the temperature $T$ in the tank, and by the measurement time $t_{meas}$. And finally, the model parameters $\theta := (a_1, b_1, a_2, b_2, a_3, b_3) \in \Theta$ are the parameters of the reaction rate constants $k_{10}, k_{20}, k_{30}$ for the three kinetic reactions taking place in the tank reactor.
\smallskip

%Which output quantities are we interested in? What is the considered prediction model?
What we are interested in here is to predict the masses $m_P, m_E, m_G$ of the target, side, and waste products in the tank reactor at time $t_{meas}$ as a function of the aforementioned input quantities $x$ and model parameters $\theta$. So, the output quantity we are interested in here is $y = h(s) := (s_4, s_5, s_6) \in \Y := \R^3$ and the prediction model $f: \X \times \Theta \to \Y$ of interest is given by
\begin{align} \label{eq:dynamic-model}
f(x,\theta) := h(s(t_{meas}, s_0, x,\theta)),
\end{align}
where $s(\cdot, s_0, x,\theta): [0,20] \to \R^6$ for every given initial state $s_0$ and every $(x,\theta) \in \X \times \Theta$ is the %global 
solution to the system equations~\eqref{eq:dynamic-system-equation}. 
%
%What is the considered design space and what is the chosen reference parameter value?
As the design space $X$, we choose the finite subset 
%\begin{align} \label{eq:design-space-dynamic-example}
%X := \{60, 60.1, 60.2, \dots, 70\} \times \{10, 10.1, 10.2, \dots, 20\}
%\end{align}
of %the input space 
$\X$ consisting of $808\,020$ design points (candidate experiments) obtained by discretizing the ranges $[20, 23]$ and $[650, 660]$ with $201$ equidistant points each, and the time range $[1, 20]$ with $20$ equidistant points. And as the reference parameter value $\ol{\theta}$, we choose the value $\theta := (a_1, b_1, a_2, b_2, a_3, b_3)$ from (20) in~\cite{ScTe20}.  
\smallskip

After having specified the prediction model $f$ along with the design space $X$ and the reference parameter $\ol{\theta}$, we can now turn to solving the corresponding locally optimal design problem~\eqref{eq:OED(X)-intro}. In order to do so, we use our strict algorithms with and without exchange. Specifically, we apply Algorithm~\ref{algo:no-exchange} and Algorithm~\ref{algo:exchange} with the initial discretization
\begin{align}
X^0 := \{20, 23\} \times \{550, 600\} \times \{1, 2, \dots, 20\} 
\end{align}
consisting of $80$ points, and with optimality tolerances $\eps := 10^{-4}$ and $\ol{\delta}_k = 0 = \ul{\delta}_k$. 
\smallskip

With these settings, our Algorithms~\ref{algo:no-exchange} and~\ref{algo:exchange} terminate already after $3$ iterations with the $\eps$-optimal designs $\xi^*$ specified in Table~\ref{tab:optimal-design-dynamic-model}. 
\begin{table}
\begin{center}
	\caption{Approxmiately optimal designs $\xi^*$ for the A- and D-criterion computed by the algorithm with exchange.}\label{tab:optimal-design-dynamic-model}
	\begin{tabular}{p{3cm} p{2cm} p{0.5cm} p{3cm} p{2cm}}
		\hline 
		A-criterion &  & & D-criterion &  \\
		Support points $x_i$ & Weights $w_i$ & & Support points $x_i$ & Weights $w_i$ \\
		\hline
		$(22.925, 592.5, 16)$ & $0.12$ & & $(22.925, 592.5, 16)$ & $0.12$\\
		$(22.925, 592.5, 16)$ & $0.12$ & & $(22.925, 592.5, 16)$ & $0.12$\\
		$(22.925, 592.5, 16)$ & $0.12$ & & $(22.925, 592.5, 16)$ & $0.12$\\
		$(22.925, 592.5, 16)$ & $0.12$ & & $(22.925, 592.5, 16)$ & $0.12$\\
		$(22.925, 592.5, 16)$ & $0.12$ & & $(22.925, 592.5, 16)$ & $0.12$\\
		$(22.925, 592.5, 16)$ & $0.12$ & & $(22.925, 592.5, 16)$ & $0.12$\\
		$(22.925, 592.5, 16)$ & $0.12$ & & $(22.925, 592.5, 16)$ & $0.12$\\
		\hline
	\end{tabular}
\end{center}
\end{table}
As is to be expected, the objective value $\Psi(\xi^*)$ at the thus computed $\eps$-optimal design is significantly better than the objective values $\Psi(\xi^{\fact})$ and $\Psi(\xi^{\rand})$ at the factorial design and the random design
\begin{align} \label{eq:factorial-and-random-design-dynamic-example}
\xi^{\fact} = \frac{1}{80} \sum_{x \in X^{\fact}} \delta_x
\qquad \text{and} \qquad
\xi^{\rand} = \frac{1}{8} \sum_{x \in X^{\rand}} \delta_x
\end{align}
with $X^{\fact} := X^0$ and with
\begin{align*}
X^{\rand} := \{ 
&(22.925, 592.5, 16),\, (20.36, 553.75, 16),\, (20.96, 550.75, 15), (22.22, 575.5, 6), \\ 
&(22.625, 565.25, 12), \, (20.675, 563.5, 12), (20.705, 552.0, 8), \, (20.075, 582.0, 14) \}, 
\end{align*}
respectively. See Table~\ref{tab:dynamic-example}. 
Additionally, our algorithms need substantially fewer iterations and computation time than the vertex-direction algorithm from~\cite{Fe, Wy70}. And moreover, they achieve a better objective function value. See Table~\ref{tab:dynamic-example-algorithm-comparison}.

\begin{table}
\begin{center}	
	\caption{Values of the A-criterion and the D-criterion at the different designs $\xi^*$, $\xi^\fact$, $\xi^\rand$ from Table~\ref{tab:optimal-design-dynamic-model} and~\eqref{eq:factorial-and-random-design-dynamic-example}.} 
	\label{tab:dynamic-example}
	\begin{tabular}{l c c}
		\hline 
		& A-criterion $\Psi_1$ & D-criterion $\Psi_0$ \\
		\hline
		$\Psi(\xi^*)$ & $\phantom{0}94.08$ & $\phantom{12}0.4213\cdot 10^{9}$ \\
		$\Psi(\xi^{\fact})$ & $105.20$ & $\phantom{12}4.9159\cdot 10^{9}$ \\
		$\Psi(\xi^{\rand})$ & $108.90$ & $123.7325 \cdot 10^{9}$ \\
		\hline
	\end{tabular}
\end{center}
\end{table}

\begin{table}
\begin{center}
	\caption{Comparison of our algorithms with exchange and without exchange with the vertex-direction algorithm for the dynamic example with the A-criterion $\Psi = \Psi_1$.}
	\label{tab:dynamic-example-algorithm-comparison}
	\begin{tabular}{c c c c}
		\hline 
		& With exchange & Without exchange & Vertex-direction \\
		\hline
		Criterion value $\Psi(\xi^*)$ & $94.0759$ & $94.0759$ & $-0.6863$ \\
		Iterations  & $3$ & $3$ & $450$\\
		Approximation error bound  & $0.0010$ & $0.0010$ & $0.0010$ \\
		Support points  & $7$ & $7$ & $4$ \\
		Computation time (in s) & $14.8141$ & $32.9620$ & $36.39$ \\
		\hline
	\end{tabular}
\end{center}
\end{table}

\section*{Acknowledgments}

We gratefully acknowledge funding from the Deutsche Forschungsgemeinschaft (DFG, German Research Foundation) – project number 466397921 – within the Priority Programme ``SPP 2331: Machine Learning in Chemical Engineering'' and within the research unit ``FOR 5359: Deep Learning on Sparse Chemical Process Data''.

\begin{small}

\end{small}

\end{document}